\documentclass[12pt,a4paper]{article}

\usepackage[margin=1in]{geometry}
\usepackage{graphicx}
\usepackage{xcolor}
\usepackage[colorlinks=true,allcolors=blue]{hyperref}
\usepackage{amsmath,amssymb,amsthm,mathtools,bm}
\usepackage{mathrsfs}
\usepackage{booktabs,array}
\usepackage{placeins}
\usepackage{microtype}
\usepackage{enumitem}
\usepackage[nameinlink,capitalise,noabbrev]{cleveref}
\usepackage[numbers,sort&compress]{natbib}

\allowdisplaybreaks
\numberwithin{equation}{section}

\newtheorem{theorem}{Theorem}[section]
\newtheorem{proposition}[theorem]{Proposition}
\newtheorem{lemma}[theorem]{Lemma}
\newtheorem{corollary}[theorem]{Corollary}
\theoremstyle{definition}
\newtheorem{definition}[theorem]{Definition}
\theoremstyle{remark}
\newtheorem{remark}[theorem]{Remark}

\AddToHook{env/theorem/begin}{\crefalias{section}{theorem}}
\AddToHook{env/proposition/begin}{\crefalias{theorem}{proposition}}
\AddToHook{env/lemma/begin}{\crefalias{theorem}{lemma}}
\AddToHook{env/corollary/begin}{\crefalias{theorem}{corollary}}
\AddToHook{env/definition/begin}{\crefalias{theorem}{definition}}
\AddToHook{env/remark/begin}{\crefalias{theorem}{remark}}

\crefname{theorem}{theorem}{theorems}
\Crefname{theorem}{Theorem}{Theorems}
\crefname{proposition}{proposition}{propositions}
\Crefname{proposition}{Proposition}{Propositions}
\crefname{lemma}{lemma}{lemmas}
\Crefname{lemma}{Lemma}{Lemmas}
\crefname{corollary}{corollary}{corollaries}
\Crefname{corollary}{Corollary}{Corollaries}
\crefname{definition}{definition}{definitions}
\Crefname{definition}{Definition}{Definitions}
\crefname{remark}{remark}{remarks}
\Crefname{remark}{Remark}{Remarks}

\crefname{equation}{equation}{equations}
\crefname{figure}{figure}{figures}
\crefname{section}{section}{sections}
\crefname{appendix}{appendix}{appendices}

\newcommand{\R}{\mathbb R}
\newcommand{\C}{\mathbb C}
\newcommand{\eps}{\varepsilon}
\newcommand{\dd}{\,\mathrm d}
\newcommand{\eexp}{\mathrm e}
\newcommand{\ord}{\mathcal O}
\newcommand{\abs}[1]{\lvert#1\rvert}
\newcommand{\norm}[1]{\lVert#1\rVert}

\newcommand{\att}{\mathrm a}
\newcommand{\rep}{\mathrm r}
\newcommand{\fl}{\mathrm f}
\newcommand{\rk}{\mathrm{RK}}
\newcommand{\loc}{\mathrm{loc}}
\newcommand{\out}{\mathrm{out}}
\newcommand{\phys}{\mathrm{phys}}
\newcommand{\one}{\boldsymbol 1}
\newcommand{\cvec}{\bm c}
\newcommand{\Amat}{\mathsf A}
\newcommand{\Gclass}{\mathscr G}

\newcommand{\ThetaRK}{\Theta}
\newcommand{\XiJ}{\Xi}

\DeclareMathOperator{\diag}{diag}

\hypersetup{
 pdftitle={Local maximal-canard threshold shifts under Runge--Kutta discretization: an observable-specific order condition},
 pdfauthor={Haibo Lu}
}

\title{Local maximal-canard threshold shifts under Runge--Kutta discretization:\\
an observable-specific order condition}
\author{Haibo Lu\\[0.4ex]
\small Shanghai Institute of Technology, Shanghai, China\\
\small \href{mailto:luhaibo1985@gmail.com}{luhaibo1985@gmail.com}\\
\small \href{https://orcid.org/0009-0000-2717-5968}{ORCID: 0009-0000-2717-5968}}
\date{23 August 2026}

\begin{document}

\maketitle

\begin{abstract}
Near a planar fast--slow fold, a local maximal canard is selected by the
parameter at which the attracting and repelling slow manifolds meet.  We
compare this threshold for a physical flow and a Runge--Kutta map, using
actual invariant manifolds on a common fold section.  An order-two
Runge--Kutta method has two independent order-three rooted-tree defects.
Both enter the pointwise one-step residual, but Gaussian transport through
the fold acts on their leading contribution by

\[
 (\alpha,\beta)\longmapsto-\frac38\beta\,\XiJ(J).
\]

Here \(\XiJ(J)\) is an explicit functional of the fold jet.  Thus the
singular passage filters the numerical defect space: it annihilates the
bushy-tree direction and can retain only the chain-tree direction.

For compact analytic classes of affinely normalizable folds and every fixed
compact, uniformly finite-stage family of real Runge--Kutta methods of order
at least two, the actual flow and map splittings obtained from independent
continuations have unique roots whose displacement satisfies a uniform
absolute estimate throughout the full small-step rectangle.  Whenever the
step-independent, exponentially small selection ambiguity is
\(o(h^2\eps^2)\), the joint-fold law is

\[
 \lambda_{\rk}-\lambda_{\fl}
 =K_\theta(J)h^2\eps^2+o(h^2\eps^2).
\]

Fold-matched continuations additionally give ordinary second-order
convergence as \(h\to0\) with \(\eps\) fixed.  The leading joint-fold bias
therefore vanishes under the single chain-tree condition
\(b^T\Amat\cvec=1/6\), without classical third order.  This is cancellation
in one nonlinear observable, not an increase in trajectory order or, in
general, in fixed-\(\eps\) threshold order.  Affine covariance transfers the
coefficient to physical fold germs, and a van der Pol invariant-graph
computation illustrates the sign change, cancellation, and fixed-\(\eps\)
convergence.

\end{abstract}


\section{Introduction}
\label{sec:introduction}

Canard trajectories spend an unexpectedly long time near a repelling slow
manifold.  In a planar fast--slow system, such a trajectory is commonly
selected by tuning a parameter until the attracting and repelling slow
manifolds meet near a fold.  The corresponding value is a threshold: a small
parameter change sends nearby trajectories onto macroscopically different
paths.  This sensitivity makes canards useful probes of how numerical
discretization changes nonlinear dynamics, but it also makes their thresholds
more delicate than finite-time trajectory errors.

The object studied here is a local maximal-canard threshold.  We continue
the attracting and repelling slow manifolds to a common section through the
fold and subtract their section heights.  A zero of this splitting is the
flow threshold.  Applying a Runge--Kutta method produces a discrete
dynamical system with its own attracting and repelling invariant manifolds;
the zero of their splitting on the same section is the numerical threshold.
Both objects are therefore defined by actual invariant sets.  In
particular, the numerical threshold is not a root inferred from a truncated
modified equation.

This distinction matters because a threshold does not sample the local
truncation error at one point.  A numerical defect generated on either side
of the fold is amplified or damped by the normal dynamics before it reaches
the section.  The relevant error is the transported response of the entire
fold passage.  In the joint fold limit along every positive power-law step
scale \(h=\eps^q\), \(q>0\), we prove that this response moves the threshold
by
\begin{equation}
 \label{eq:introduction-headline}
 \lambda_{\rk}-\lambda_{\fl}
 =K_\theta(J)h^2\eps^2+o(h^2\eps^2),
\end{equation}
where \(h\) is the normalized fast-time step,
\(\eps\) is the singular parameter, \(J\) records the local fold jet, and
\(\theta\) denotes the Runge--Kutta tableau.  The theorem applies to every
fixed compact, uniformly finite-stage family of real tableaux satisfying the
classical order-two conditions.  It includes explicit midpoint, Heun,
Ralston's second-order method, Kutta's third-order method, classical RK4,
implicit midpoint and the trapezoidal rule; neither symmetry nor implicit
stages are assumed.

The coefficient in \eqref{eq:introduction-headline} reveals a second,
conceptual point.  The order-three local error of an order-two Runge--Kutta
method has two independent rooted-tree defect coordinates,
\begin{equation}
 \label{eq:introduction-defects}
 \alpha_\theta
 =\frac12b_\theta^T\cvec_\theta^{\circ2}-\frac16,
 \qquad
 \beta_\theta
 =b_\theta^T\Amat_\theta\cvec_\theta-\frac16.
\end{equation}
Both coordinates enter the general pointwise residual of the exact one-step
map.  Fold transport tests that residual against a Gaussian kernel and acts
on the defect plane as the linear functional
\begin{equation}
 \label{eq:introduction-filter}
 (\alpha,\beta)\longmapsto -\frac38\beta\,\XiJ(J).
\end{equation}
Thus the bushy-tree direction is removed only after transport, while the
chain-tree direction is the only one that can survive.  A further canonical
residual generated by general tableaux has zero Gaussian moment and is
likewise absent from the leading response.
The threshold coefficient therefore vanishes whenever
\(b^T\Amat\cvec=1/6\).  This single chain-tree condition is weaker than
classical third order: it cancels the leading joint-fold coefficient without
changing the trajectory order of the integrator.

That the error in a prescribed quantity of interest need not mirror the
error of the full numerical solution is familiar from adjoint-weighted error
analysis for ordinary differential equations
\cite{CaoPetzold2004,ChaudhryEstepStevensTavener2021}.  Related work on
Runge--Kutta discretizations of optimal-control problems likewise shows that
state and transformed-adjoint accuracy must be analysed together
\cite{Hager2000}.  The novelty here is therefore not the general observation
that functional accuracy may differ from trajectory accuracy.  It is the
explicit singular mechanism: fold transport produces the response
functional \eqref{eq:introduction-filter} on the order-three B-series defect
plane of the actual one-step map and removes one complete rooted-tree
direction from an actual invariant-manifold threshold.

Two comparison regimes accompany \eqref{eq:introduction-headline}.
For independently continued local manifolds, exponentially small
non-uniqueness remains in the error bound.  The roots and the absolute
estimate nevertheless hold on a full small-step rectangle, and the relative
formula holds whenever this selection term is \(o(h^2\eps^2)\), in
particular for every \(h=\eps^q\), \(q>0\), including \(h\ll\eps\).  Thus
the local threshold is noncanonical beyond exponentially small accuracy,
but the algebraic coefficient \(K_\theta(J)\) is independent of those choices
at every scale resolved by the theorem.  For the ordinary limit \(h\to0\)
at fixed positive \(\eps\), we introduce a second-order fold-matching
condition on the actual invariant graphs.  The resulting companion theorem
gives both a uniform full-rectangle fold law and fixed-\(\eps\) second-order
convergence.  Matching specifies which exponentially close continuations are
compared; it imposes no lower bound on the step.
In this matched formulation, the chain-tree condition removes the
\(K_\theta(J)\eps^2\) term from the fixed-\(\eps\) \(h^2\)-coefficient.
Its remaining algebraic fold contribution is \(O(\eps^{5/2})\), up to the
exponentially small matching term, as \(\eps\downarrow0\).  Thus the gain is
in the singular asymptotic coefficient, not in the power of \(h\) at fixed
\(\eps\).

We work first in an affine normalized chart, the natural covariance class
for Runge--Kutta dynamics \cite{McLachlanModinMuntheKaasVerdier2016}.
Affine changes of state and parameters, together with constant time
rescaling, conjugate the stage equations exactly; nonlinear state changes
generally do not.  An explicit normalization transfers the result to every
physical planar canard germ satisfying the stated fold and unfolding
conditions and expresses the coefficient in physical jets.

Under the composition hypotheses recorded in
\cref{rem:global-composition}, the response computed here supplies the local
fold term in a global threshold displacement.  An outer return or matching
condition can contribute at the same order; without controlling that term
there is no universal local-to-global formula.

The continuous geometry of canards and their unfolding at planar folds was
developed in
\cite{BenoitCallotDienerDiener1981,DumortierRoussarie1996,
KrupaSzmolyan2001,DeMaesschalckDumortierRoussarie2021}; tracking invariant
manifolds to exponentially small accuracy is central to this setting
\cite{JonesKaperKopell1996}.  Numerical continuation of canard orbits and
the relation of maximal canards to computable bifurcation data are addressed
in \cite{DesrochesKrauskopfOsinga2010,Kuehn2010}.

For discrete fast--slow systems, invariant-manifold theory on normally
hyperbolic regions and its extension toward nonhyperbolic points are
developed in
\cite{NippStoffer1995,JelbartKuehn2023,JelbartKuehn2024}.  The nonhyperbolic
discrete geometric singular perturbation theory and finite-order
\(C^r\)-embedding results of Jelbart and Kuehn provide a general framework
near nonhyperbolic points \cite{JelbartKuehn2024}.  Finite-order embedding
accuracy alone does not resolve exponentially close slow manifolds or the
root of their canard splitting; here those issues are handled by an exact-map
residual and explicit shielding of the selection ambiguity.  Engel and
Kuehn analysed the explicit Euler map near a
transcritical singularity, including the regime \(h<\eps\), by discrete
blow-up and invariant-manifold estimates \cite{EngelKuehn2019}.  Their
nonhyperbolic geometry and passage observable differ from the fold
section-root studied here; our result also allows \(h\) on either side of
\(\eps\) and treats arbitrary compact order-two Runge--Kutta families.  Two
other particularly close strands are the study of extended and symmetric
loss of stability in
planar fast--slow maps by Engel and Jard\'on-Kojakhmetov
\cite{EngelJardon2020}, and the exact-map canard analysis for Kahan-type
discretizations by Engel, Kuehn, Petrera and Suris
\cite{EngelKuehnPetreraSuris2022}.  The Engel--Jard\'on-Kojakhmetov and
Engel--Kuehn--Petrera--Suris analyses do not derive the same-section actual
map--flow root displacement as a defect functional for arbitrary fixed
compact order-two Runge--Kutta families, with root estimates uniform over
local selections.  Difference equations with small step and
canard or delay phenomena have also been studied directly
\cite{Fruchard1992,FruchardSchafke2003,ElRabih2003}, while modified
equations have been used to quantify discretization-induced changes in
canard delay \cite{EngelGottwald2024}.  Our question is complementary: we
compare the roots of an actual flow splitting and an actual Runge--Kutta-map
splitting on the same physical section, uniformly over fold data, tableaux
within each fixed compact method family, and admissible local selections,
and identify the induced
functional on B-series defects.  The rooted-tree notation is standard
\cite{HairerLubichWanner2006,SofroniouOevel1997,KetchesonRanocha2023}.

The proof must pass through four distinct objects: actual invariant graphs,
the exact one-step residual on those graphs, its transported section
response, and the resulting actual threshold root.  The paper follows this
sequence.  \Cref{sec:physical-setting} defines the physical and normalized
problems and proves affine covariance, and \cref{sec:main-results} states the
arbitrary-selection and matched threshold laws.  \Cref{sec:fold-geometry}
constructs the invariant graphs; \cref{sec:exact-residual,sec:gaussian-root}
derive the residual, transport it through the fold, and capture the root.
\Cref{sec:matched} constructs matched continuations for the fixed-parameter
limit.  Physical coefficients, the van der Pol illustration, the
conditional local-to-global decomposition and comparisons among standard
methods are collected in \cref{sec:applications}.  \Cref{sec:conclusion}
summarizes the geometric message, the scope of the local observable and the
remaining global questions.  The appendices supply the uniform graph,
residual, transport and resolvent estimates needed to make these conclusions
hold for the actual maps and roots.

\section{Physical folds and the local numerical observable}
\label{sec:physical-setting}

We begin in physical variables.  This makes precise which planar canard
germs are represented by the normalized family used in the analysis and
keeps the Runge--Kutta map tied to the coordinates in which it is actually
formed.  After translating a singular canard point to the origin, write
\begin{equation}
 \label{eq:physical-germ}
 \frac{\dd u}{\dd s}=f(u,v,\eta,\mu),
 \qquad
 \frac{\dd v}{\dd s}=\eta g(u,v,\eta,\mu).
\end{equation}
Here \(\eta>0\) is the physical singular parameter, \(\mu\) is the
distinguished parameter, and \(s\) is the physical fast time.  The functions
\(f\) and \(g\) are real analytic near the origin.  Unless stated otherwise,
all derivatives of these functions in this section are evaluated at the
origin.

The fold and passage conditions are
\begin{equation}
 \label{eq:physical-fold-conditions}
 f=f_u=g=0,
 \qquad
 f_vf_{uu}g_u\ne0,
 \qquad
 f_vg_u<0.
\end{equation}
To describe a nondegenerate one-parameter unfolding, put
\begin{equation}
 \label{eq:physical-hatted-jets}
 \widehat f_{u\mu}
 =f_{u\mu}-\frac{f_{uv}f_\mu}{f_v},
 \qquad
 \widehat g_\mu
 =g_\mu-\frac{g_vf_\mu}{f_v},
\end{equation}
and define
\begin{equation}
 \label{eq:physical-unfolding-determinant}
 \mathcal T
 =f_{uu}\widehat g_\mu-g_u\widehat f_{u\mu}.
\end{equation}
We assume \(\mathcal T\ne0\).  Geometrically, \(\mathcal T/f_{uu}\) is
the derivative of the slow drift along the singular fold curve with respect
to \(\mu\); see \cref{app:affine-normalization}.

\subsection{A compact normalized analytic class}

Assign the fold weights
\begin{equation}
 \label{eq:fold-weights}
 w(x)=1,
 \qquad
 w(y)=w(\eps)=w(\lambda)=2.
\end{equation}
Fix conjugation-invariant complex polydiscs
\(\Omega\Subset\Omega^+\) about the origin in
\(\C^4_{x,y,\eps,\lambda}\).

\begin{definition}[Normalized analytic fold class]
 \label{def:normalized-fold-class}
A normalized analytic fold datum is
\[
 J=(A,B,C_\eps,D,E,F_s,S_\eps,R_f,R_g)
\]
and determines the vector field
\begin{equation}
 \label{eq:normalized-fold-field}
 \begin{aligned}
  \dot x={}&f_J(x,y,\eps,\lambda)\\
   ={}&x^2-y+Ax^3+Bxy+C_\eps x\eps+Dx\lambda+R_f,\\
  \dot y={}&\eps g_J(x,y,\eps,\lambda)\\
   ={}&\eps\bigl(x-\lambda+Ex^2+F_sy
                     +S_\eps\eps+R_g\bigr).
 \end{aligned}
\end{equation}
The remainders are holomorphic on \(\Omega^+\), real under complex
conjugation, and contain only Taylor monomials of weighted degree at least
four in \(R_f\) and at least three in \(R_g\).

A normalized analytic fold class \(\Gclass\) is any nonempty compact set of
such data, where compactness is taken in the Euclidean topology for the
displayed coefficients and in the \(H^\infty(\Omega^+)\) norm for the
remainders.  We require a uniform unfolding bound
\begin{equation}
 \label{eq:normalized-unfolding-bound}
 \inf_{J\in\Gclass}(D(J)+2)>0.
\end{equation}
All constants below may depend on \(\Gclass\) and on
\(\Omega\Subset\Omega^+\), but not on an individual datum.
\end{definition}

This definition deliberately uses an arbitrary compact analytic class rather
than a fixed coefficient cube.  It is stable under affine normalization of a
compact family of physical germs, and it records exactly the uniform bounds
used in the invariant-graph and transport estimates.

\subsection{Runge--Kutta families and exact one-step maps}

\begin{definition}[Compact order-two Runge--Kutta family]
 \label{def:rk-family}
Let
\(\ThetaRK=\bigsqcup_{j=1}^{N}\Theta_j\) be a finite disjoint union
of nonempty compact parameter sets.  On \(\Theta_j\), let
\((\Amat_\theta,b_\theta,\cvec_\theta)\) be a continuously varying real
Runge--Kutta tableau with a fixed number \(s_j\) of stages, and put
\(s_* = \max_{1\le j\le N}s_j<\infty\).  We assume
\begin{equation}
 \label{eq:rk-order-two-conditions}
 \Amat_\theta\one=\cvec_\theta,
 \qquad
 b_\theta^T\one=1,
 \qquad
 b_\theta^T\cvec_\theta=\frac12
\end{equation}
for every \(\theta\in\ThetaRK\).  No positivity, symmetry, or
parabola-preservation condition is imposed.  In particular, every fixed real
Runge--Kutta method of classical order at least two is included as a
singleton family.
\end{definition}

The two order-three defect coordinates used below are
\begin{equation}
 \label{eq:rk-defect-coordinates}
 \alpha_\theta
 =\frac12b_\theta^T\cvec_\theta^{\circ2}-\frac16,
 \qquad
 \beta_\theta
 =b_\theta^T\Amat_\theta\cvec_\theta-\frac16.
\end{equation}
It is also convenient to record the comparator data
\begin{equation}
 \label{eq:rk-comparator-data}
 \begin{aligned}
  \delta_\theta&=\alpha_\theta-\beta_\theta,
  &q_\theta&=\frac{\delta_\theta}{2},\\
  d_\theta&=\cvec_\theta^{\circ2}
       -2\Amat_\theta\cvec_\theta-2\delta_\theta\one,
  &\gamma_\theta&=-b_\theta^T\Amat_\theta d_\theta.
 \end{aligned}
\end{equation}
The order-two conditions give \(b_\theta^Td_\theta=0\).  The proof does not
assume \(d_\theta=0\); for a general tableau the associated canonical
residual is instead removed at leading order by the Gaussian fold response.

For \(F_J=(f_J,\eps g_J)\), define the stages and exact one-step map by
\begin{equation}
 \label{eq:normalized-rk-map}
 \begin{aligned}
  Z_i&=z+h\sum_{m=1}^{s_j}a_{im}(\theta)
                  F_J(Z_m;\eps,\lambda),\\
  \Phi_{J,h,\theta,\eps,\lambda}(z)
   &=z+h\sum_{i=1}^{s_j}b_i(\theta)
                  F_J(Z_i;\eps,\lambda),
  \qquad \theta\in\Theta_j.
 \end{aligned}
\end{equation}
For implicit methods, \(\Phi\) always denotes the analytic stage branch
continuing from \(Z_i=z\) at \(h=0\).  Thus \(\Phi\) is the actual
Runge--Kutta map, not a truncated modified equation.  Its contained local
inverse is represented by the adjoint tableau
\begin{equation}
 \label{eq:rk-adjoint}
 \Amat_\theta^\dagger=\one b_\theta^T-\Amat_\theta,
 \qquad
 b_\theta^\dagger=b_\theta,
 \qquad
 \cvec_\theta^\dagger=\one-\cvec_\theta,
\end{equation}
through
\begin{equation}
 \label{eq:rk-contained-inverse}
 \Phi_{J,h,\theta,\eps,\lambda}^{-1}
 =\Phi_{J,-h,\theta^\dagger,\eps,\lambda}
\end{equation}
whenever both sides are taken on their contained analytic branches.
The adjoint image of \(\ThetaRK\) is again a compact, uniformly
finite-stage family, so the forward and inverse branches admit common
small-step stage bounds.

Choose a class-uniform \(\delta>0\) such that the relevant graphs remain in
the real part of \(\Omega\), and put
\begin{equation}
 \label{eq:local-sections}
 \Sigma_{\rm in}=\{x=-\delta\},
 \qquad
 \Sigma_{\rm c}=\{x=0\},
 \qquad
 \Sigma_{\rm out}=\{x=\delta\}.
\end{equation}
The graph constructions are carried out later.  The following definition
only names the observable once actual invariant graphs have been selected;
the uniform admissibility conditions on those selections are stated in
\cref{def:admissible-selection} before the threshold theorems.

\begin{definition}[Actual local splitting and threshold]
 \label{def:local-splitting}
Let \(\bullet\in\{\fl,\rk\}\).  Suppose a selection \(\sigma_\bullet\)
provides attracting and repelling actual invariant graphs
\[
 M_{\bullet,q}^{\sigma_\bullet}
 =\{(x,m_{\bullet,q}^{\sigma_\bullet}(x;\lambda))\},
 \qquad q\in\{\att,\rep\},
\]
from the corresponding side section to \(\Sigma_{\rm c}\).  For the flow,
actual means that the graph satisfies the exact tangency equation.  For the
Runge--Kutta map, it means local invariance under
\eqref{eq:normalized-rk-map}, using the contained inverse
\eqref{eq:rk-contained-inverse} on the repelling side.  Define
\begin{equation}
 \label{eq:local-splitting}
 \Delta_\bullet^{\sigma_\bullet}(\lambda)
 =m_{\bullet,\rep}^{\sigma_\bullet}(0;\lambda)
  -m_{\bullet,\att}^{\sigma_\bullet}(0;\lambda).
\end{equation}
A zero of \(\Delta_\bullet^{\sigma_\bullet}\) is the selected local
maximal-canard threshold.  This definition concerns the actual flow or
actual numerical map; neither a formal invariant graph nor a modified-flow
root is called a threshold.  When the tableau must be displayed, we write
\(M_{\rk,q}^{\sigma_\rk}(\theta)\) and
\(\Delta_{\rk,\theta}^{\sigma_\rk}\).
\end{definition}

\subsection{Affine normalization and covariance}

\begin{proposition}[Affine normalization of a physical canard germ]
 \label{prop:affine-normalization}
Suppose \eqref{eq:physical-germ} satisfies
\eqref{eq:physical-fold-conditions} and
\(\mathcal T\ne0\).  Fix \(\tau>0\), and set
\begin{equation}
 \label{eq:affine-scales}
 a_x=\frac{2}{\tau f_{uu}},
 \qquad
 a_y=-\frac{a_x}{\tau f_v},
 \qquad
 a_\eps=-\frac{1}{\tau^2f_vg_u},
 \qquad
 a_\lambda=-\frac{2g_u}{\tau\mathcal T},
\end{equation}
and
\begin{equation}
 \label{eq:affine-fold-shear}
 \ell=\frac{2g_u\widehat f_{u\mu}}
              {\tau f_{uu}\mathcal T}.
\end{equation}
Then \(a_\eps>0\), and the jointly affine change
\begin{equation}
 \label{eq:affine-normalization-map}
 \begin{aligned}
  u&=a_x x+\ell\lambda,\\
  v&=a_y y-\frac{f_\eta}{f_v}a_\eps\eps
           -\frac{f_\mu}{f_v}a_\lambda\lambda,\\
  \eta&=a_\eps\eps,
  &\mu&=a_\lambda\lambda,
  &s&=\tau t
 \end{aligned}
\end{equation}
puts the physical germ into \eqref{eq:normalized-fold-field} with \(D=0\).
The remaining displayed coefficients are
\begin{equation}
 \label{eq:affine-normalized-coefficients}
 \begin{aligned}
 A&=\frac{\tau a_x^2}{6}f_{uuu},
 &B&=\tau a_y f_{uv},\\
 C_\eps&=\tau a_\eps
       \left(f_{u\eta}-\frac{f_{uv}f_\eta}{f_v}\right),
 &E&=\frac{a_x g_{uu}}{2g_u},\\
 F_s&=\frac{a_y g_v}{a_x g_u},
 &S_\eps&=\frac{a_\eps}{a_x g_u}
       \left(g_\eta-\frac{g_vf_\eta}{f_v}\right).
 \end{aligned}
\end{equation}
The Taylor remainders have the weighted orders required in
\cref{def:normalized-fold-class}.

Let \(\mathcal A_{\eps,\lambda}\) denote the state part of
\eqref{eq:affine-normalization-map}, and let \(\Psi^{\phys}\) be the
Runge--Kutta map obtained by applying the same tableau directly to
\eqref{eq:physical-germ}.  If \(k\) is its physical step and
\begin{equation}
 \label{eq:physical-normalized-variables}
 h=\frac{k}{\tau},
 \qquad
 \eps=\frac{\eta}{a_\eps},
 \qquad
 \lambda=\frac{\mu}{a_\lambda},
\end{equation}
then the stage equations and outputs are exactly conjugate:
\begin{equation}
 \label{eq:rk-affine-conjugacy}
 \Psi^{\phys}_{k,\theta,\eta,\mu}
 =\mathcal A_{\eps,\lambda}
  \circ\Phi_{J,h,\theta,\eps,\lambda}
  \circ\mathcal A_{\eps,\lambda}^{-1}.
\end{equation}
The normalized central section corresponds to
\begin{equation}
 \label{eq:physical-fold-following-section}
 \Sigma_{\rm c}^{\phys}(\mu)
 =\left\{u=\frac{\ell}{a_\lambda}\mu\right\}
 =\left\{u=-\frac{\widehat f_{u\mu}}{f_{uu}}\mu\right\}.
\end{equation}
Corresponding actual graph splittings satisfy
\begin{equation}
 \label{eq:physical-normalized-splitting}
 \Delta_\bullet^{\phys}(\mu)
 =a_y\Delta_\bullet\left(\frac{\mu}{a_\lambda}\right),
 \qquad \bullet\in\{\fl,\rk\},
\end{equation}
with the same selection labels transported by
\(\mathcal A_{\eps,\lambda}\).  Hence their threshold roots obey
\(\mu_\bullet=a_\lambda\lambda_\bullet\).

In particular, if along a specified asymptotic regime a normalized threshold
law has coefficient \(K_\theta(J)\),
\begin{equation}
 \label{eq:normalized-coefficient-law}
 \lambda_{\rk}-\lambda_{\fl}
 =K_\theta(J)h^2\eps^2+o(h^2\eps^2),
\end{equation}
then its coefficient in the original physical units is
\begin{equation}
 \label{eq:physical-coefficient-transform}
 \boxed{
 K_\theta^{\phys}
 =\frac{a_\lambda}{\tau^2a_\eps^2}K_\theta(J),}
 \qquad
 \mu_{\rk}-\mu_{\fl}
 =K_\theta^{\phys}k^2\eta^2+o(k^2\eta^2).
\end{equation}
\end{proposition}

The section in \eqref{eq:physical-fold-following-section} follows the
linearized singular fold location.  This is the general affine reduction.
When the parameter directly unfolds the slow drift, a version preserving a
fixed physical section is available.

\begin{remark}[Fixed physical section]
 \label{rem:fixed-physical-section}
If \(\widehat g_\mu\ne0\), set \(\ell=0\) and replace the parameter scale
in \eqref{eq:affine-scales} by
\begin{equation}
 \label{eq:fixed-section-parameter-scale}
 a_\lambda=-\frac{2g_u}{\tau f_{uu}\widehat g_\mu}.
\end{equation}
Then \(x=0\) corresponds to the fixed physical section \(u=0\), and the
normalized field still has the form \eqref{eq:normalized-fold-field}, now
with
\begin{equation}
 \label{eq:fixed-section-D}
 D=-\frac{2g_u\widehat f_{u\mu}}
          {f_{uu}\widehat g_\mu},
 \qquad
 D+2=\frac{2\mathcal T}{f_{uu}\widehat g_\mu}.
\end{equation}
This version lies in a class satisfying
\eqref{eq:normalized-unfolding-bound} whenever the last expression has a
uniform positive lower bound.  Substitution of this \(D\) and
\(a_\lambda\) into the covariance formula
\eqref{eq:physical-coefficient-transform} gives the same gauge-independent
physical jet coefficient as \eqref{eq:physical-jet-coefficient} below.
Thus \cref{cor:physical-threshold} also holds on the fixed section \(u=0\)
whenever the hypotheses of this remark are satisfied.
\end{remark}

\begin{corollary}[Uniformity over compact physical families]
 \label{cor:uniform-physical-families}
Let a compact family of physical germs be continuous in a common
\(H^\infty\) neighbourhood and satisfy uniform versions of
\eqref{eq:physical-fold-conditions} and \(\mathcal T\ne0\); in particular,
assume that \(\abs{f_v}\), \(\abs{f_{uu}}\), \(\abs{g_u}\),
\(\abs{\mathcal T}\), and \(-f_vg_u\) have positive uniform lower bounds.
Then one may choose common normalized polydiscs and a single \(\tau>0\) so
that the affine images form a normalized analytic fold class in the sense of
\cref{def:normalized-fold-class}.  The affine scales and their inverses are
uniformly bounded, and every uniform normalized threshold estimate transfers
to the physical family through \eqref{eq:physical-normalized-variables} and
\eqref{eq:physical-coefficient-transform}.
\end{corollary}

\begin{remark}[Natural covariance boundary]
 \label{rem:affine-boundary}
The transformation above is affine in state and parameters for every frozen
parameter value, and the clock is changed only by the constant factor
\(\tau\).  These are precisely the operations used in the exact stage
conjugacy \eqref{eq:rk-affine-conjugacy}.  A nonlinear state change or a
state-dependent time rescaling followed by a fresh Runge--Kutta
discretization generally produces a different numerical map.  Accordingly,
\cref{prop:affine-normalization} is an affine-covariance result, not a claim
of nonlinear coordinate invariance.
\end{remark}

Finally introduce the fold variables
\begin{equation}
 \label{eq:fold-scaling}
 r=\sqrt\eps,
 \qquad
 x=rX,
 \qquad
 y=r^2Y,
 \qquad
 \lambda=r^2L,
 \qquad
 H=hr,
 \qquad
 \zeta=rt.
\end{equation}
If \(\mathcal S_r(X,Y)=(rX,r^2Y)\), the inner numerical map is the exact
conjugate
\begin{equation}
 \label{eq:inner-map}
 \widehat\Phi_{J,r,H,\theta,L}
 =\mathcal S_r^{-1}\circ
   \Phi_{J,h,\theta,r^2,r^2L}\circ\mathcal S_r.
\end{equation}
At \(r=0\), the inner flow contains the canonical orbit
\begin{equation}
 \label{eq:canonical-fold-orbit}
 Y=X^2-\frac12,
 \qquad
 \frac{\dd X}{\dd\zeta}=\frac12.
\end{equation}
Its normal variational equation is responsible for the Gaussian transport
functional used below.

\section{Threshold laws and the fold-response functional}
\label{sec:main-results}

We now state the two forms of the threshold law.  The first compares
arbitrary admissible local continuations.  It is uniform over these choices,
but necessarily retains their exponentially small ambiguity.  The second
matches the flow and numerical continuations to second order at the outer
edge of the fold region.  That additional relation removes the ambiguity at
the scale needed for the ordinary limit \(h\to0\) with \(\eps\) fixed.

We use the inner scales in \cref{eq:fold-scaling}.  For a normalized fold
datum \(J\in\Gclass\), write
\begin{equation}
 \label{eq:fold-jet-combinations}
 F_3(J)=A+B,\qquad G_2(J)=E+F_s,
 \qquad \XiJ(J)=F_3(J)-2G_2(J).
\end{equation}
The transverse slope of the flow splitting is governed by
\begin{equation}
 \label{eq:flow-slope-coefficient}
 a(J)=(D+2)\sqrt{\frac{\pi}{2}},
\end{equation}
which is bounded away from zero on the fixed normalized class.  The
threshold-response coefficient is
\begin{equation}
 \label{eq:threshold-coefficient}
 \boxed{
 K_\theta(J)
 =-\frac{3}{4(D+2)}
   \left(b_\theta^T\Amat_\theta\cvec_\theta-\frac16\right)
   \XiJ(J)
 =-\frac{3\beta_\theta}{4(D+2)}\XiJ(J).}
\end{equation}

\paragraph{The result in informal form.}
For any fixed families of folds and methods just described, the admissible
classes of actual flow and map invariant manifolds defined below are
nonempty for all sufficiently small \(r=\sqrt{\eps}\) and \(h\), with no
lower relation between the two parameters.  Independently chosen local
continuations satisfy an absolute threshold law with leading algebraic term
\(K_\theta(J)h^2\eps^2\), controlled higher algebraic remainders, and a
step-independent, exponentially small continuation term.  Hence every
positive power-law step scale resolves the same leading coefficient
\eqref{eq:threshold-coefficient}.  If the flow and map continuations are
matched to second order at the outer collar, the normalized fold law is
uniform on the full small-step rectangle and, for each fixed sufficiently
small \(\eps>0\), the numerical threshold converges to its matched flow
threshold with order two as \(h\to0\).  In both formulations, all dependence
of the leading joint-fold bias on the Runge--Kutta method passes through the
single chain-tree defect \(\beta_\theta\).  The definitions and quantitative
remainders that make these statements precise follow next.

Define also the leading inner flow root and its response window by
\begin{align}
 \label{eq:main-L0-definition}
 L_0(J)&=
 \frac{-3A-B-4C_\eps+2E-2F_s+8S_\eps}{4(D+2)},\\
 \label{eq:main-response-window}
 \mathcal W_r(J)&=\{L:\abs{L-L_0(J)}\le C_Wr\},
\end{align}
where \(C_W>0\) is uniform.

\subsection{Arbitrary local continuations}

The selections appearing in \cref{def:local-splitting} reflect the
exponentially small non-uniqueness of locally continued flow and map slow
manifolds.  We now specify the uniform class over which the first theorem is
quantified.
Choose a compact interval \(\Lambda\) in the inner parameter \(L\) that
contains \(\mathcal W_r(J)\) for every \(J\in\Gclass\) and all sufficiently
small \(r\).  On each
side \(q\in\{\att,\rep\}\), fix an outer collar
\(I_q^{\rm col}\), a bridge \(I_q^{\rm br}\) joining that collar to
\(x=0\), and their nonempty overlap \(I_q^{\rm ov}\); these intervals are
independent of \(J,r,h\), and \(\theta\).  A convenient common choice is
given in \eqref{app-geom:side-intervals}.  Let \(\phi_J\) be the local
critical graph determined by \(f_J(x,\phi_J(x),0,0)=0\), and write
\(a_+=\max\{a,0\}\).

\begin{definition}[Uniformly admissible actual selections]
 \label{def:admissible-selection}
Given \(J,r\), and, in the map case, \(\theta,h\), set
\(\eps=r^2\) and use \(L\in\Lambda\) to parameterize
\(\lambda=r^2L\).  On side
\(q\in\{\att,\rep\}\), an admissible selection for
\(\bullet\in\{\fl,\rk\}\) consists of a collar graph
\(m_{\bullet,q}\) on \(I_q^{\rm col}\times\Lambda\) and a bridge graph
\(u_{\bullet,q}\) on \(I_q^{\rm br}\times\Lambda\).  The two graphs agree
on \(I_q^{\rm ov}\times\Lambda\), lie in one fixed real neighbourhood in
\(\Omega\), and, for constants \(B_{\rm sel},c_d,C_d>0\), satisfy
\begin{equation}
 \label{eq:main-selection-bounds}
 \begin{aligned}
  &\max_{0\le i\le4,\ 0\le j\le2}r^{-2}
    \sup_{I_q^{\rm col}\times\Lambda}
    \abs{\partial_x^i\partial_L^j(m_{\bullet,q}-\phi_J)}
    \le B_{\rm sel},\\
  &\max_{0\le i\le4,\ 0\le j\le2}r^{(i-2)_+}
    \sup_{I_q^{\rm br}\times\Lambda}
    \abs{\partial_x^i\partial_L^j u_{\bullet,q}}
    \le B_{\rm sel}.
 \end{aligned}
\end{equation}
Write \(\widehat m_{\bullet,q}\) for the pasted graph.  In the flow case it
satisfies exact tangency,
\(r^2g_J=\partial_x\widehat m_{\fl,q}f_J\), and directed drift
\(c_dr^2\le f_J\le C_dr^2\), with all functions evaluated on that graph at
\((\eps,\lambda)=(r^2,r^2L)\).  In the map case let
\(\Phi^+=\Phi_{J,h,\theta,r^2,r^2L}\),
\(\Phi^-=\Phi_{J,-h,\theta^\dagger,r^2,r^2L}\), and define
\[
 z_q(x,L)=(x,\widehat m_{\rk,q}(x,L)),\qquad
 T_q^\pm(x,L)=\pi_x\Phi^\pm(z_q(x,L)).
\]
Whenever \(x\) and \(T_q^\pm(x,L)\) belong to the pasted domain, require
\(\Phi^\pm(z_q(x,L))=z_q(T_q^\pm(x,L),L)\).  All stages remain in a fixed
common stage neighbourhood, and
\(c_dhr^2\le T_q^+(x,L)-x\le C_dhr^2\).
An admissible selection is the pair of its attracting and repelling side
selections.  A selection class is uniformly admissible on a parameter
rectangle if the same intervals, \(\Lambda,B_{\rm sel},c_d,C_d\), and stage
neighbourhood work for every member.  Nonemptiness means one flow selection
for every \((J,r)\) and one map selection for every
\((J,\theta,r,h)\).  No continuity in \(J,r,h\), or \(\theta\) is required;
\cref{app-geom:selection-definition} records the exact graph identities in
components.
\end{definition}

\begin{theorem}[Uniform local threshold displacement]
 \label{thm:all-selection}
Fix a compact normalized fold class \(\Gclass\) and a compact order-two
Runge--Kutta family \(\ThetaRK\).  There are constants
\(r_0,h_0,C,c,C_W>0\), an integer \(M\), and a choice of uniform admissibility constants
such that the flow classes \(\mathfrak S_{\fl}(J,r)\) and map classes
\(\mathfrak S_{\rk}(J,\theta,r,h)\) of
\cref{def:admissible-selection} are nonempty throughout the following
parameter rectangle.  For every
\[
 J\in\Gclass,
 \quad \theta\in\ThetaRK,
 \quad 0<r\le r_0,
 \quad 0<h\le h_0,
\]
and every independently chosen
\(\sigma_f\in\mathfrak S_{\fl}(J,r)\) and
\(\sigma_m\in\mathfrak S_{\rk}(J,\theta,r,h)\), the flow
splitting and the exact-map splitting have unique simple zeros in
\(\mathcal W_r(J)\).  Denote the corresponding unscaled normalized
parameter values by \(\lambda_{\fl}^{\sigma_f}\) and
\(\lambda_{\rk,\theta}^{\sigma_m}\).  Then
\begin{equation}
 \label{eq:all-selection-law}
 \boxed{
 \lambda_{\rk,\theta}^{\sigma_m}
 -\lambda_{\fl}^{\sigma_f}
 =K_\theta(J)h^2r^4
 +\ord(h^2r^5+h^3r^5)
 +\ord\!\left(r^{-M}\eexp^{-c/r^2}\right).}
\end{equation}
The estimates, the response neighbourhood, and the simplicity constants are
uniform in all the displayed variables and in both selections.  In
particular, existence and uniqueness hold on the full small-step rectangle;
no lower relation between \(h\) and \(r\) is imposed.
\end{theorem}

The exponential term in \eqref{eq:all-selection-law} is independent of
\(h\).  The stated uniform remainder is therefore negligible relative to
the leading scale whenever
\begin{equation}
 \label{eq:selection-resolution}
 \Phi(r,h)
 :=h^{-2}r^{-M-4}\eexp^{-c/r^2}\longrightarrow0.
\end{equation}
Every positive power-law step scale \(h=r^p\), with arbitrary \(p>0\), satisfies
\eqref{eq:selection-resolution}.  Along any such path,
\begin{equation}
 \label{eq:all-selection-relative-law}
 \frac{\lambda_{\rk,\theta}^{\sigma_m}
       -\lambda_{\fl}^{\sigma_f}}{h^2\eps^2}
 \longrightarrow K_\theta(J)
\end{equation}
uniformly over the fold data, methods, and selections.  This includes
\(h\ll\eps\).  At a fixed positive \(r\), however, independently selected
flow and map thresholds need not approach one another at order \(h^2\);
\cref{prop:pairing-necessary} shows that such a statement would be false.

\subsection{Second-order matched continuations}

We formulate the fixed-\(\eps\) result through a property of the actual
outer invariant graphs.  Let \(m_{\bullet,q}\) denote their unscaled normalized
\(y\)-graphs, where \(\bullet\in\{\fl,\rk\}\) and
\(q\in\{\att,\rep\}\).  If \(\phi_J\) is the local critical-manifold graph,
set
\[
 V_{\bullet,q}=\frac{m_{\bullet,q}-\phi_J}{r^2}.
\]

\begin{definition}[Enhanced admissible flow selections]
 \label{def:enhanced-selection}
An admissible flow selection is enhanced if, uniformly on both sides, its
collar and bridge graphs satisfy the bounds in
\eqref{eq:main-selection-bounds} also for \(i=5\), with a common constant
\(B_{\rm sel}^+\).  In addition, each one-sided bridge has an exact-tangency
continuation across \(x=0\) over a fixed inner interval
\(-\ell_{\rm ev}\le \varsigma x/r\le0\), where
\(\varsigma=-1\) on the attracting side and \(\varsigma=1\) on the
repelling side; the same weighted \(C_x^5C_L^2\) bridge bound holds on this
extension.  The extension is used only to evaluate the same selected flow
graph across a final cell and does not identify the attracting and repelling
graphs.
\end{definition}

\begin{definition}[Fold-matched actual continuations]
 \label{def:fold-matched}
A uniformly second-order fold-matched rule \(\mathfrak p\) assigns, for
every \(J\in\Gclass\), \(\theta\in\ThetaRK\),
\(0<r\le r_0\), and \(0<h\le h_0\), selections and functions
\[
 \begin{gathered}
  \sigma_f^{\mathfrak p}(J,r)\in\mathfrak S_{\fl}(J,r),\qquad
  \sigma_m^{\mathfrak p}(J,\theta,r,h)
    \in\mathfrak S_{\rk}(J,\theta,r,h),\\
  U_{q,2}(J,\theta,r;\,\cdot\,,\,\cdot\,)
    \in C_x^3C_L^1(I_q^{\rm col}\times\Lambda),
  \qquad q\in\{\att,\rep\}.
 \end{gathered}
\]
Here the flow selection is enhanced and independent of \(h\) and
\(\theta\), while each \(U_{q,2}\) is independent of \(h\).  They satisfy
\begin{equation}
 \label{eq:fold-matched-condition}
 \norm{
  V_{\rk,q}-V_{\fl,q}-H^2U_{q,2}
 }_{C_x^3C_L^1(I_q^{\rm col}\times\Lambda)}
 \le C_{\mathfrak p}hH^2,
 \qquad
 \norm{U_{q,2}}_{C_x^3C_L^1(I_q^{\rm col}\times\Lambda)}
 \le C_{\mathfrak p}.
\end{equation}
The enhanced admissibility constant and the constants in
\eqref{eq:fold-matched-condition} are uniform over \(J\), \(\theta\),
\(r\), \(h\), and both sides.  In the unscaled normalized \(y\)-coordinate,
\eqref{eq:fold-matched-condition}
is equivalent to
\[
 m_{\rk,q}-m_{\fl,q}
 =h^2r^4U_{q,2}+\ord(h^3r^4).
\]
\end{definition}

The definition allows any outer selection rule with the stated second-order
relation between its invariant graphs.  A common completion of the outer
flow and map dynamics supplies one such rule; this construction is proved
in \cref{prop:matched-existence}.

\begin{theorem}[Full-rectangle and fixed-parameter laws]
 \label{thm:matched-threshold}
Under the hypotheses of \cref{thm:all-selection}, the class of uniformly
second-order fold-matched rules is nonempty.  Fix any such rule
\(\mathfrak p\).  After
reducing \(r_0,h_0\) if necessary, for every \(J\in\Gclass\),
\(\theta\in\ThetaRK\), \(0<r\le r_0\), and \(0<h\le h_0\), its actual flow
and map splittings have unique roots
\(\lambda_{\fl}^{\mathfrak p}(J,r)\) and
\(\lambda_{\rk,\theta}^{\mathfrak p}(J,r,h)\).  Uniformly on the complete
small-step rectangle,
\begin{equation}
 \label{eq:matched-uniform-law}
 \left|
 \lambda_{\rk,\theta}^{\mathfrak p}
 -\lambda_{\fl}^{\mathfrak p}
 -K_\theta(J)h^2r^4
 \right|
 \le C_{\mathfrak p}h^2r^4
 \left(r+H+r^{-M}\eexp^{-c/r^2}\right).
\end{equation}
Consequently,
\begin{equation}
 \label{eq:matched-uniform-limit}
 \lim_{r\downarrow0}
 \sup_{\substack{J\in\Gclass,\ \theta\in\ThetaRK\\0<h\le h_0}}
 \left|
 \frac{\lambda_{\rk,\theta}^{\mathfrak p}
       -\lambda_{\fl}^{\mathfrak p}}{h^2r^4}
 -K_\theta(J)
 \right|=0.
\end{equation}

For every fixed \(0<r\le r_0\), there are \(h_*(r)>0\), a coefficient
\(\kappa_\theta^{\mathfrak p}(J,r)\), and a constant
\(C_{r,\mathfrak p}\), uniform in \(J\) and \(\theta\), such that
\begin{equation}
 \label{eq:matched-fixed-r-law}
 \lambda_{\rk,\theta}^{\mathfrak p}
 -\lambda_{\fl}^{\mathfrak p}
 =h^2\kappa_\theta^{\mathfrak p}(J,r)
 +\ord_{r,\mathfrak p}(h^3),
 \qquad 0<h\le h_*(r),
\end{equation}
and
\begin{equation}
 \label{eq:matched-kappa}
 \kappa_\theta^{\mathfrak p}(J,r)
 =K_\theta(J)r^4
 +\ord\!\left(r^5+r^{-M}\eexp^{-c/r^2}\right).
\end{equation}
Thus the coherently selected local threshold converges with order two as
\(h\to0\) for every fixed sufficiently small \(\eps=r^2\).
\end{theorem}

Different matching rules may give different values of
\(\kappa_\theta^{\mathfrak p}(J,r)\) at a fixed \(r\).  Equation
\eqref{eq:matched-kappa} shows that they have the same leading normalized
fold limit.  Thus the leading fold coefficient is independent of the
matching rule, whereas the fixed-\(r\) coefficient can retain information
from the chosen outer boundary data.

\subsection{The response seen by this observable}

For a classical order-two Runge--Kutta method, the two order-three rooted
trees give the independent defects
\begin{equation}
 \label{eq:order-three-defects-results}
 \alpha_\theta
 =\frac12b_\theta^T\cvec_\theta^{\circ2}-\frac16,
 \qquad
 \beta_\theta
 =b_\theta^T\Amat_\theta\cvec_\theta-\frac16.
\end{equation}
Both defect coordinates enter the general pointwise formula for the exact
one-step residual.  The fold passage nevertheless acts on the defect plane
through the linear response functional
\begin{equation}
 \label{eq:rank-one-response}
 (\alpha,\beta)
 \longmapsto -\frac38\beta\,\XiJ(J).
\end{equation}
This functional is derived in \cref{sec:exact-residual,sec:gaussian-root}.
It is nonzero when \(\XiJ(J)\ne0\), with the \(\alpha\)-axis in its kernel,
and is the zero functional when \(\XiJ(J)=0\).

\begin{corollary}[Observable-specific cancellation]
 \label{cor:chain-cancellation}
For every fold datum in the normalized class, the coefficient of
\(h^2\eps^2\) vanishes whenever
\begin{equation}
 \label{eq:chain-condition}
 b_\theta^T\Amat_\theta\cvec_\theta=\frac16.
\end{equation}
Condition \eqref{eq:chain-condition} is the chain-tree order condition.  It
does not require the bushy-tree condition and is therefore strictly weaker
than classical third order.  For a particular system, the same leading
coefficient also vanishes when \(\XiJ(J)=0\).
\end{corollary}

The corollary concerns cancellation in the leading error of one nonlinear
dynamical observable.  It does not raise the trajectory order of the method
and does not, in general, raise fixed-\(\eps\) threshold convergence above
order two.  Its content is that fold transport discards a defect direction
which classical local error analysis must retain.

\subsection{Proof architecture}

Four steps separate the formal defect coordinates from the actual roots in
\cref{thm:all-selection,thm:matched-threshold}.

First, uniform graph transforms and fold bridges construct actual attracting
and repelling invariant manifolds and show that different admissible
continuations are exponentially close; their flow splitting has a unique
transverse zero (\cref{sec:fold-geometry}).  Second, the exact Runge--Kutta
stage equations are divided by the step size to identify the leading
residual along the canonical fold orbit (\cref{sec:exact-residual}).  Third,
an exact finite Green identity transports this residual to the central
section.  Its products converge to a Gaussian weight, which produces
\eqref{eq:rank-one-response} (\cref{sec:gaussian-root}).  Finally, a buffered
\(C^0\)-to-\(C^1\) argument transfers the section response to a unique root
of the exact numerical splitting.

For \cref{thm:matched-threshold}, the same transport and root arguments are
used with an \(H^2\)-scaled outer datum.  A normalized graph resolvent shows
that a common completion of the normalized collar dynamics produces
precisely this datum.  Thus the
fixed-parameter theorem adds one construction at the outer collar; it does
not introduce a second fold-response mechanism.

\section{Invariant manifolds and the transverse flow threshold}
\label{sec:fold-geometry}

Away from the fold, the attracting and repelling slow manifolds are
normally hyperbolic.  We continue them through two overlapping regions: an
outer collar, where graph transforms are uniformly contracting, and a fold
bridge, where the graph equation is scalar after division by the slow drift.
The same scalar equation shows that different admissible continuations are
exponentially close by the time they reach the central section.

For a selection \(\sigma_\bullet\), let
\(y_{\bullet,q}^{\sigma_\bullet}(\lambda)\) be its value on
\(\Sigma_{\rm c}\), where \(\bullet\in\{\fl,\rk\}\) and
\(q\in\{\att,\rep\}\).  We orient the splitting as repelling minus
attracting:
\begin{equation}
 \label{eq:section-splitting}
 \Delta_\bullet^{\sigma_\bullet}(\lambda)
 =y_{\bullet,\rep}^{\sigma_\bullet}(\lambda)
  -y_{\bullet,\att}^{\sigma_\bullet}(\lambda).
\end{equation}

\begin{proposition}[Fold geometry and flow transversality]
 \label{prop:fold-geometry}
Fix \(\Gclass\) and \(\ThetaRK\) as in
\cref{def:normalized-fold-class,def:rk-family}.  There are
\(r_0,h_0,C,c>0\), an integer \(M\), and non-empty selection classes in the
sense of \cref{def:admissible-selection}, with
the following properties.

For every \(J\in\Gclass\), \(\theta\in\ThetaRK\),
\(0<r\le r_0\), and \(0<h\le h_0\), the selected attracting and repelling
flow and numerical-map graphs reach \(\Sigma_{\rm c}\).  Their section
values are twice continuously differentiable in the inner parameter \(L\)
on a common fixed outer interval.

If \(\sigma_\bullet\) and \(\widetilde\sigma_\bullet\) are two selections
for the same flow, or for the same numerical map with fixed \(\theta\), then
\begin{equation}
 \label{eq:selection-shielding}
 \max_{0\le j\le2}
 \sup_L
 \left|
 \partial_L^j
 \left(
 y_{\bullet,q}^{\sigma_\bullet}(r^2L)
 -y_{\bullet,q}^{\widetilde\sigma_\bullet}(r^2L)
 \right)
 \right|
 \le Cr^{-M}\eexp^{-c/r^2}
\end{equation}
for \(q=\att,\rep\).

There is a continuous function \(L_0(J)\), uniformly bounded on
\(\Gclass\), such that every flow selection satisfies
\begin{equation}
 \label{eq:flow-splitting-profile}
 r^{-3}\Delta_\fl^{\sigma_f}(r^2L)
 =a(J)\{L-L_0(J)\}
 +\ord(r)+\ord\!\left(r^{-M}\eexp^{-c/r^2}\right)
\end{equation}
in \(C_L^2\).  In particular, the flow splitting has one simple zero
\(L_\fl^{\sigma_f}\) in a common response neighbourhood, and its
\(\lambda\)-slope satisfies
\begin{equation}
 \label{eq:flow-physical-slope}
 \partial_\lambda\Delta_\fl^{\sigma_f}
   (\lambda_\fl^{\sigma_f})
 =a(J)r+\ord(r^2)
  +\ord\!\left(r^{-M}\eexp^{-c/r^2}\right).
\end{equation}
All constants and neighbourhoods are uniform over the fold class, the
method family, and the selections.
\end{proposition}

\begin{proof}
We indicate the three estimates that will be used later and defer their
uniform construction to \cref{app:fold-geometry}.

On the attracting side, the positive flow and the positive Runge--Kutta
map contract normal graphs.  On the repelling side we orient time toward the
fold.  For the map this means applying the negative step of the adjoint
tableau
\[
 \Amat_\theta^\dagger=\one b_\theta^T-\Amat_\theta,
 \qquad
 \cvec_\theta^\dagger=\one-\cvec_\theta,
\]
which is the contained inverse of the original positive-step map.
Compactness of \(\ThetaRK\) gives a common small-step contraction margin.
The resulting collar graphs are continued to the fold by the scalar graph
equation and its exact discrete counterpart.

For two continuations of the same dynamics, subtraction gives a homogeneous
normal secant equation.  Its multiplier product has the form
\begin{equation}
 \label{eq:gaussian-shielding-product}
 \prod_{k<j}\{1-2Hz_{k+1}+\ord(HrP(z_{k+1})+H^2P(z_{k+1}))\},
\end{equation}
where \(z\) is the signed fold coordinate and \(P\) is a fixed polynomial
majorant.  The product is bounded by a Gaussian and is exponentially small
when transported from \(z\asymp r^{-1}\) to \(z=0\).  Differentiating the
same recurrence twice in \(L\) proves \eqref{eq:selection-shielding}.

Finally, differentiation of the flow graph equation with respect to \(L\)
gives opposite half-Gaussian responses on the two sides.  Their difference
is \(a(J)r^3\) to leading order.  A second expansion locates the centre
\(L_0(J)\), which proves \eqref{eq:flow-splitting-profile}.  The uniform
positive lower bound on \(a(J)\) gives monotonicity and uniqueness.  Since
\(\partial_\lambda=r^{-2}\partial_L\), the \(\lambda\)-slope is
\eqref{eq:flow-physical-slope}.
\end{proof}

At a zero of \(\Delta_\bullet\), the two selected invariant graphs meet on
\(\Sigma_{\rm c}\).  Uniqueness of the flow, or the forward/inverse identity
for the map, continues the meeting point to a local orbit.  Thus the zeros
used in the paper are local maximal-canard thresholds for the chosen
continuations.

\section{The exact numerical defect along the fold orbit}
\label{sec:exact-residual}

Classical order conditions describe a one-step trajectory error.  The
quantity needed here is different: the exact Runge--Kutta map is evaluated
on a flow-invariant graph, and its failure to preserve a nearby shifted
graph is then transported over \(O(H^{-1})\) cells.  This section identifies
the residual before transport.

Fix \(\theta\in\ThetaRK\), suppress its subscript temporarily, and write
\begin{equation}
 \label{eq:alpha-beta}
 \alpha=\frac12b^T\cvec^{\circ2}-\frac16,
 \qquad
 \beta=b^T\Amat\cvec-\frac16,
 \qquad
 \delta=\alpha-\beta.
\end{equation}
The shifted canonical parabola is chosen with
\begin{equation}
 \label{eq:q-d-gamma}
 q=\frac{\delta}{2},\qquad
 d=\cvec^{\circ2}-2\Amat\cvec-2\delta\one,
 \qquad
 \gamma=-b^T\Amat d.
\end{equation}
The order-two identities imply \(b^Td=0\).  The special case \(d=0\)
means that the shifted canonical parabola is preserved exactly, but this
identity is not assumed in the paper.

The adjoint tableau
\begin{equation}
 \label{eq:adjoint-tableau}
 \Amat^\dagger=\one b^T-\Amat,
 \qquad b^\dagger=b,
 \qquad \cvec^\dagger=\one-\cvec
\end{equation}
satisfies
\begin{equation}
 \label{eq:adjoint-defects}
 \alpha^\dagger=\alpha,\qquad
 \beta^\dagger=\beta,\qquad
 d^\dagger=d,\qquad
 \gamma^\dagger=-\gamma.
\end{equation}
It also represents the exact local inverse,
\begin{equation}
 \label{eq:adjoint-inverse}
 \Phi_{h,\theta}^{-1}=\Phi_{-h,\theta^\dagger},
\end{equation}
on every contained continuation branch.  These algebraic facts allow the
attracting and repelling sides to be treated with a common coreward
orientation even for a non-symmetric or explicit method.

\subsection{Two zero-moment modes}

Let \(X\) denote the canonical inner fold coordinate.  For the shifted
canonical graph, the first residual which is present even in the unperturbed
fold is
\begin{equation}
 \label{eq:canonical-hermite-mode}
 \widehat H^4\gamma\mathscr H_2(X),
 \qquad
 \mathscr H_2(X)=X^2-\frac14.
\end{equation}
It has zero half-Gaussian moment:
\begin{equation}
 \label{eq:hermite-zero-moment}
 \int_0^\infty \eexp^{-2X^2}\mathscr H_2(X)\dd X=0.
\end{equation}

The first coefficient that couples the method defects to the perturbed fold
jet is the polynomial
\begin{equation}
 \label{eq:Q10}
 \boxed{
 Q_{10}(X;\alpha,\beta)
 =-\frac38\beta F_3
  +\left(\frac12\alpha+\frac14\beta\right)G_2
  +2(\beta-\alpha)G_2X^2.}
\end{equation}
Both defect coordinates enter this general pointwise formula.  Introduce
the normalized half-Gaussian expectation
\begin{equation}
 \label{eq:gaussian-expectation}
 \mathbb E[Q]
 =\sqrt{\frac8\pi}
  \int_0^\infty\eexp^{-2X^2}Q(X)\dd X.
\end{equation}
Since \(\mathbb E[X^2]=1/4\),
\begin{equation}
 \label{eq:Q10-projection}
 \boxed{
 \mathbb E[Q_{10}]
 =-\frac38\beta\{F_3-2G_2\}
 =-\frac38\beta\XiJ(J).}
\end{equation}
Thus the bushy coordinate is present in the general local-error formula but
is annihilated by the fold-response functional.
\Cref{fig:fold-response} shows the canonical geometry and the weighted
zero-moment profile behind this cancellation.

\begin{figure}[t]
 \centering
 \includegraphics[width=0.92\textwidth]{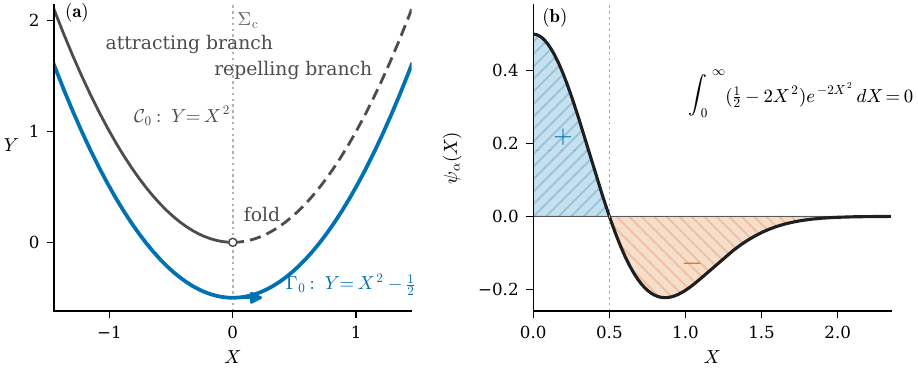}
 \caption{Canonical fold geometry and the defect-response functional.
 Panel (a) is drawn in canonical inner coordinates.  The grey solid
 \((X<0)\) and dashed \((X>0)\) curves are, respectively, the attracting and
 repelling branches of the singular critical manifold
 \(\mathcal C_0=\{Y=X^2\}\), not actual finite-\(r\) invariant manifolds.
 The blue curve is the exact canonical inner orbit
 \(\Gamma_0=\{Y=X^2-1/2\}\), its arrow gives the increasing-\(\zeta\) flow
 orientation, and the dotted line is the central section.  Panel (b) shows
 the exact weighted bushy-defect profile
 \(\psi_\alpha(X)=(1/2-2X^2)e^{-2X^2}\); its positive and negative lobes
 have equal absolute area, so its half-line integral vanishes.  Line style,
 opposing hatches and signs make the distinctions independent of colour.}
 \label{fig:fold-response}
\end{figure}

\subsection{From the stage equations to an exact residual}

Let \(\mathcal D_r\) denote the normalized analytic vector field together
with any admissible actual flow graph satisfying the enhanced
\(C_x^5C_L^2\) bounds in \cref{def:enhanced-selection}; the constructed reference
graphs are special cases.  Apply the exact Runge--Kutta stage equations with
signed inner step \(\widehat H\) to the graph shifted by \(qH^2\), and
subtract the shifted graph at the numerical output abscissa.  Denote the
result by \(\mathfrak E_\vartheta(\widehat H;\mathcal D_r)(X)\), where
\(\vartheta\) is \(\theta\) on the attracting side and \(\theta^\dagger\)
on the repelling side.

\begin{proposition}[Uniform exact residual]
 \label{prop:exact-residual}
On a common real weighted stage tube about the fold chains, there is a fixed
polynomial majorant \(P\) such that
\begin{equation}
 \label{eq:exact-residual}
 \boxed{\begin{aligned}
 \mathfrak E_\vartheta(\widehat H;\mathcal D_r)(X)
 =\widehat H^3\bigl[&rQ_{10}(X)
 +\widehat H\gamma_\vartheta\mathscr H_2(X)\\
 &+\ord\!\left(
 r^2P(X)+r\abs{\widehat H}P(X)+\widehat H^2P(X)
 \right)\bigr].
 \end{aligned}}
\end{equation}
The displayed value estimate is uniform over \(J\in\Gclass\),
\(\theta\in\ThetaRK\), both oriented sides, and every graph with the fixed
enhanced admissibility bound.  The divided quotient also
extends with one \(L\)-derivative; its differentiated remainder may lose one
power of \(r\), as stated precisely in \cref{app-res:datum-jet}.  Moreover,
the map residual is divided exactly: the remainder in
\eqref{eq:exact-residual} is defined by the integral Hadamard quotient of
the actual \(C^5\) stage equations, not by a truncated modified equation.
\end{proposition}

\begin{proof}
At \(\widehat H=0\), consistency, flow invariance of the actual graph,
and the two order conditions give a triple zero.  On the canonical datum,
direct elimination of the stages gives
\(\widehat H^4\gamma_\vartheta\mathscr H_2\) as the first remaining term.
Subtracting that canonical residual from the residual for \(\mathcal D_r\)
produces an additional factor \(r\).  Integral Hadamard division in
\(\widehat H\), followed by the weighted Taylor expansion of a reference
fold datum, gives \eqref{eq:exact-residual}.  The homogeneous flow-graph
comparison in \cref{app-res:enhanced-selection-comparison} extends the same
estimate to every enhanced admissible graph.  The coefficient of
\(r\widehat H^3\) is \eqref{eq:Q10}; its algebra is independent of the
number of stages.  Compactness of the tableau family and the analytic fold
class supplies the common stage neighbourhood and polynomial majorant.
Full stage differentiation and the adjoint identities are given in
\cref{app:exact-residual}.
\end{proof}

The two cancellations now have different origins.  The coefficient of
\(\alpha\) in \(Q_{10}\) has zero Gaussian response by
\eqref{eq:Q10-projection}.  The extra mode caused by failure to preserve the
shifted parabola has zero Gaussian response by
\eqref{eq:hermite-zero-moment}.  Neither cancellation is a pointwise
identity of the numerical map.

\section{Gaussian transport and the exact numerical threshold}
\label{sec:gaussian-root}

The local residual in \cref{prop:exact-residual} is accumulated from both
sides of the fold.  We first propagate it to \(\Sigma_{\rm c}\), then use
the transverse flow slope to locate the zero of the exact-map splitting.

\subsection{An exact finite Green identity}

Let \(w\) be a reference numerical graph, let \(E\) be the corresponding
flow graph in divided coordinates, and include the common offset by writing
\(\overline E=E+q_\theta H^2\).  Put \(d=w-\overline E\).  Choose a source
on each graph so that the two sources have the same target abscissa.  Exact
subtraction of their invariance relations yields a scalar recurrence
\begin{equation}
 \label{eq:cell-recurrence}
 d(z_k)=M_kd(z_{k+1})+\mathfrak r_k.
\end{equation}
Iteration over
a finite chain gives
\begin{equation}
 \label{eq:finite-green}
 d(z_0)=P_Nd(z_N)+\sum_{j=0}^{N-1}P_j\mathfrak r_j,
 \qquad
 P_0=1,\qquad P_j=\prod_{k<j}M_k.
\end{equation}
Adding or deleting a complete terminal cell leaves
\eqref{eq:finite-green} unchanged, which is important when the stopping
index varies with parameters.

On either coreward chain,
\begin{equation}
 \label{eq:mesh-and-multiplier}
 z_{k+1}-z_k=\frac H2+\ord(HrP(z_{k+1})),
 \qquad
 M_k=1-2Hz_{k+1}
 +\ord(HrP(z_{k+1})+H^2P(z_{k+1})).
\end{equation}
Hence \(P_j\) converges to \(\eexp^{-2z_j^2}\).  Because every forcing
term retains its mesh factor, the estimates use
\begin{equation}
 \label{eq:weighted-sum}
 H\sum_jP_j(1+z_j)^m\le C_m
\end{equation}
rather than the number \(O(H^{-1})\) of cells.  This is why the transport
estimate remains valid for arbitrarily small \(h\).

Substituting \eqref{eq:exact-residual} into \eqref{eq:finite-green} gives,
on the side with sign \(\varsigma\in\{-1,1\}\),
\begin{equation}
 \label{eq:one-side-response}
 d_\varsigma(0,L)
 =-2\varsigma rH^2
   \int_0^\infty\eexp^{-2X^2}Q_{10}(X)\dd X
 +\ord(r^2H^2+rH^3)+\mathcal E_{\varsigma,\rm end}.
\end{equation}
The transported Hermite term is part of the displayed remainder: its
limiting moment vanishes, and its discrete quadrature error is
\(O(r+H)\).  For arbitrary independent selections,
\(\mathcal E_{\varsigma,\rm end}\) is exponentially small but independent
of \(h\).  For a fold-matched pair, it carries the same factor \(H^2\) as
the response.

The common offset \(q_\theta H^2\) cancels before the two sides are
subtracted.  One factor two in the final coefficient comes from the
effective mesh \(H/2\) in \eqref{eq:mesh-and-multiplier}; a second comes
from subtracting the oppositely oriented sides.  Define
\begin{equation}
 \label{eq:section-response-coefficient}
 S_\theta(J)
 =-\frac34\beta_\theta\sqrt{\frac\pi2}\,\XiJ(J)
 =a(J)K_\theta(J).
\end{equation}

\begin{proposition}[Response on the central section]
 \label{prop:fold-response}
On a common response neighbourhood of the flow root, arbitrary admissible
selections satisfy
\begin{equation}
 \label{eq:all-selection-section-response}
 -\left(
 \Delta_{\rk,\theta}^{\sigma_m}(r^2L)
 -\Delta_\fl^{\sigma_f}(r^2L)
 \right)
 =S_\theta(J)r^3H^2
 +\mathcal R_{\sigma_f,\sigma_m}(L),
\end{equation}
where
\begin{equation}
 \label{eq:all-selection-response-remainder}
 \norm{\mathcal R_{\sigma_f,\sigma_m}}_{C^0}
 \le C\left(
 r^4H^2+r^3H^3+r^{-M}\eexp^{-c/r^2}
 \right).
\end{equation}
For a uniformly second-order fold-matched rule \(\mathfrak p\), the same
identity holds with
\begin{equation}
 \label{eq:matched-response-remainder}
 \norm{\mathcal R_{\mathfrak p}}_{C^0}
 \le C_{\mathfrak p}H^2\left(
 r^4+r^3H+r^{-M}\eexp^{-c/r^2}
 \right).
\end{equation}
At fixed \(r\), the matched response divided by \(h^2\) converges in
\(C_L^1\) with an \(O_r(h)\) remainder.
\end{proposition}

\begin{proof}
Equation \eqref{eq:one-side-response}, the two zero moments in
\cref{eq:Q10-projection,eq:hermite-zero-moment}, and exact cancellation of
the common offset give the reference response.  Exponential shielding from
\cref{eq:selection-shielding} transfers it to arbitrary labels and gives
\eqref{eq:all-selection-response-remainder}.  Under
\eqref{eq:fold-matched-condition}, the outer datum in
\eqref{eq:finite-green} is \(H^2\) times a uniformly bounded function; the
Gaussian terminal product preserves this factor and gives
\eqref{eq:matched-response-remainder}.  For fixed \(r\), divide the entire
Green identity by \(H^2\) and compare the finite sum with its continuous
variation-of-constants formula at the same moving terminal point.  The
terminal phase cancels between the terminal value and the integral, giving
the stated \(C_L^1\) rate.  Uniform product, quadrature, and endpoint
estimates are proved in \cref{app:transport-root}.
\end{proof}

\subsection{From the response to the unique map root}

Write the splittings in inner variables as
\begin{equation}
 \label{eq:root-functions-new}
 f(L)=\Delta_\fl^{\sigma_f}(r^2L),
 \qquad
 m(L)=\Delta_{\rk,\theta}^{\sigma_m}(r^2L),
\end{equation}
and let \(u_f\) be the zero of \(f\).  By
\cref{prop:fold-geometry},
\begin{equation}
 \label{eq:inner-flow-slope}
 n:=f'(u_f)=a(J)r^3+\ord(r^4)+\text{exponentially small terms}.
\end{equation}
The section response determines \(-m(u_f)\), but a value estimate alone
does not imply that \(m\) has a unique nearby zero.  We therefore correct
the difference by its predicted displacement,
\begin{equation}
 \label{eq:corrected-splitting}
 F(L)=m(L)-f(L)+nK_\theta(J)H^2.
\end{equation}
On a slightly larger response interval,
\begin{equation}
 \label{eq:corrected-splitting-bounds}
 \norm F_{C^0}
 \le C\left(r^4H^2+r^3H^3+r^{-M}\eexp^{-c/r^2}\right),
 \qquad
 \norm{F''}_{C^0}\le C.
\end{equation}
An optimized interior interpolation estimate then controls \(F'\) without
differentiating the Gaussian asymptotic.

\begin{lemma}[Transfer of the transverse root]
 \label{lem:root-transfer}
After reducing \(r_0,h_0\), the exact-map splitting is strictly monotone on
the common inner response neighbourhood and has one zero there.  For
arbitrary selections,
\begin{equation}
 \label{eq:inner-root-law}
 L_{\rk,\theta}^{\sigma_m}-L_\fl^{\sigma_f}
 =K_\theta(J)H^2
 +\ord\!\left(rH^2+H^3+r^{-M}\eexp^{-c/r^2}\right).
\end{equation}
For a fold-matched pair, the last term in
\eqref{eq:inner-root-law} is instead
\(\ord(H^2r^{-M}\eexp^{-c/r^2})\).  At fixed \(r\), the
\(C_L^1\) convergence in \cref{prop:fold-response} gives an expansion of
the root with remainder \(O_r(h^3)\) in the unscaled normalized
\(\lambda\)-coordinate.
\end{lemma}

\begin{proof}
From \eqref{eq:corrected-splitting-bounds}, the one-dimensional interior
interpolation inequality gives
\[
 \norm{F'}_{C^0}
 \le C\sqrt{\norm F_{C^0}\norm{F''}_{C^0}}.
\]
After division by \(n\asymp r^3\), this is small uniformly on the inner
interval.  The corresponding flow secants differ from \(n\) by
\(O(rn)\), so \(m'>0\).  A buffered fixed-point form of the scalar root
equation gives existence and the location estimate.  At the root, use
\(m=f-nK_\theta H^2+F\) and the flow secant expansion; only the value bound
for \(F\), not its interpolated derivative, enters the leading location.
The matched statement follows from the matched value and \(C^1\) estimates.
The detailed interpolation and containment argument is in
\cref{app:transport-root}.
\end{proof}

Multiplying \eqref{eq:inner-root-law} by \(r^2\) and using
\(r^2H^2=h^2r^4\), \(r^3H^2=h^2r^5\), and
\(r^2H^3=h^3r^5\) proves \cref{thm:all-selection}.  The matched estimates
in \cref{prop:fold-response,lem:root-transfer} reduce
\cref{thm:matched-threshold} to the construction and fixed-parameter
response in \cref{sec:matched-continuations}.

\section{Paired invariant manifolds and the step-first limit}
\label{sec:matched}
\label{sec:matched-continuations}

The threshold in this paper is local: it is defined after attracting and
repelling slow manifolds have been continued from normally hyperbolic
collars to the fold section.  Such continuations are exponentially close,
but they are not identical.  This distinction is immaterial on algebraic
joint scales and becomes decisive if \(r>0\) is fixed while \(h\) tends to
zero.  Indeed, the ambiguity between two flow continuations then remains
fixed while the numerical displacement is divided by \(h^2\).

The next proposition records that this obstruction is genuine.  It does
not rule out a distinguished continuation supplied by global boundary
data; it only rules out a step-first statement quantified over independent
local choices.

\begin{proposition}[Independent local selections obstruct a common step-first law]
\label{prop:pairing-necessary}
The collection of normalized analytic fold data contains a datum \(J_0\)
with the following property.  For every sufficiently small fixed \(r>0\),
there are two
admissible flow continuations whose local threshold roots

\[
 \lambda_{\fl}^{0}(r)\ne \lambda_{\fl}^{1}(r)
\]

are distinct.  Fix a compact order-two Runge--Kutta family, one of its
methods \(\theta\), and any admissible family of actual-map continuations.
Then the two quotients

\begin{equation}
 \frac{\lambda_{\rk,\theta}(r,h)-\lambda_{\fl}^{i}(r)}{h^2},
 \qquad i=0,1,
 \label{eq:unpaired-step-first}
\end{equation}

cannot both have finite limits as \(h\downarrow0\).  Consequently no
fixed-\(r\), relative \(h^2\) law can hold uniformly over independently
chosen flow and map continuations.
\end{proposition}

\begin{proof}
The construction of the two flow continuations and the short contradiction
argument are given in \cref{app:matched-no-go}.
\end{proof}

The fold-matching condition in \cref{def:fold-matched} removes exactly this
ambiguity.  It compares actual invariant graphs on the normally hyperbolic
collars before either graph is transported to the fold.  In divided normal
coordinates their difference has an \(H^2=h^2r^2\) first term and an
\(O(hH^2)\) remainder.  The condition does not select a preferred local
slow manifold: different fixed constructions can give different
fixed-\(r\) coefficients.  Its content is that the flow and map choices in
one pair vary together to second order in the step size.

The full-rectangle \(C^0\) response is already part of
\cref{prop:fold-response}.  The extra conclusion needed for the ordinary
step-first limit is its normalized \(C^1\) form.  We denote the common
response neighbourhood in the inner parameter \(L\) by
\(\mathcal W_r(J)\), as in \eqref{eq:main-response-window}.

\begin{lemma}[Step-first response of a fold-matched pair]
\label{lem:matched-step-response}
Fix a compact normalized fold class \(\Gclass\), a compact order-two
Runge--Kutta family \(\ThetaRK\), and a uniformly second-order fold-matched
rule \(\mathfrak p\).  For every fixed sufficiently small \(r>0\), there is
a function

\[
 \mathscr R_\theta^{\mathfrak p}(J,r,\lambda)
 =-\lim_{h\downarrow0}
   \frac{\Delta_{\rk,\theta}^{\mathfrak p}(\lambda)
         -\Delta_{\fl}^{\mathfrak p}(\lambda)}{h^2}
\]

where the limit exists in \(C_\lambda^1\) on the unscaled normalized response
interval, and

\begin{equation}
 \Delta_{\rk,\theta}^{\mathfrak p}(\lambda)
 =\Delta_{\fl}^{\mathfrak p}(\lambda)
  -h^2\mathscr R_\theta^{\mathfrak p}(J,r,\lambda)
  +\mathcal E_{r,h}(\lambda),
 \qquad
 \norm{\mathcal E_{r,h}}_{C^1_\lambda}\le C_{r,\mathfrak p}h^3 .
\label{eq:matched-fixed-r-response}
\end{equation}

Uniformly for \(L\in\mathcal W_r(J)\) as \(r\downarrow0\),

\begin{equation}
 \mathscr R_\theta^{\mathfrak p}(J,r,r^2L)
 =S_\theta(J)r^5
  +\ord\!\left(r^6+r^{-M}\eexp^{-c/r^2}\right).
\label{eq:matched-response-coefficient}
\end{equation}
\end{lemma}

\begin{proof}
On either side of the fold, divide the exact common-target recurrence from
\cref{prop:fold-response} by \(H^2\).  Fold matching changes its terminal
datum from an uncontrolled bounded quantity to

\[
 d_{q,h}(z_N,L)=H^2\gamma_{q,h}(L),
 \qquad
 \norm{\gamma_{q,h}}_{C_L^1}\le C_{\mathfrak p}r^{-M}.
\]

The normalized terminal datum converges in \(C_L^1\) at rate \(O_r(h)\) by
the matching condition.  Target-aligned multipliers and sources have
one-step errors \(O_r(H^2)\) and \(O_r(H)\), respectively.  The finite
product--Duhamel identity and the resulting Riemann sum therefore converge
in \(C_L^1\) at rate \(O_r(h)\).  Restoring the unscaled normalized factor
\(r^2\)
gives \eqref{eq:matched-fixed-r-response}.  Finally the same Gaussian
sum and zero-moment estimates used in \cref{prop:fold-response}, now after
\(h\downarrow0\), give
\eqref{eq:matched-response-coefficient}.
\end{proof}

It remains to show that the class of matching rules is nonempty.  The
construction uses one fixed completion of the normally hyperbolic collars.
The flow graph is the invariant graph of the completed vector field.  The
Runge--Kutta method is first applied in the unscaled normalized affine
\((x,y)\) variables; only the difference between that exact numerical map
and the exact flow map is then extended with the same cutoffs.  Hence the two completed dynamics agree
with the uncompleted normalized local flow and map on the collar buffer.

\begin{proposition}[Existence of fold-matched actual rules]
\label{prop:matched-existence}
For every compact normalized fold class \(\Gclass\) and every compact
order-two Runge--Kutta family \(\ThetaRK\), there are constants
\(r_0,h_0>0\) and a fixed common-completion convention whose restrictions
to the unscaled normalized collars and fold bridges form a uniformly
second-order fold-matched rule \(\mathfrak p\).  Its flow graphs are
invariant under the actual flow, and its map graphs are invariant under the
actual Runge--Kutta map and its contained local inverse.  The collar buffers,
cutoffs, and extension operators are fixed uniformly over \(\Gclass\) and
\(\ThetaRK\) and are independent of \(r,h\) and of the individual tableau.
The completed vector field still depends on \(J\) and \(r\), while the
completed numerical dynamics depends on \(h\) and \(\theta\) through the
actual Runge--Kutta map being extended.
\end{proposition}

\begin{proof}
See \cref{app:matched-construction}.  The decisive estimate there is the
normalized graph-resolvent expansion

\[
 S_{\rk,q}-S_{\fl,q}
 =h^2r^2U_{q,2}+\ord_{C_x^3C_L^1}(h^3r^2),
 \qquad q\in\{\att,\rep\}.
\]

Restriction to the unscaled normalized collar buffer gives precisely
\cref{def:fold-matched}.
\end{proof}

\begin{proof}[Proof of \cref{thm:matched-threshold}]
The non-emptiness assertion follows from
\cref{prop:matched-existence}.  For an arbitrary fold-matched rule,
\cref{prop:fold-response} supplies the full-rectangle splitting estimate,
and \cref{lem:matched-step-response} supplies the fixed-\(r\), \(C^1\)
second-order response.  Apply
\cref{lem:root-transfer} to the transverse flow zero.  Since

\[
 \partial_\lambda\Delta_{\fl}^{\mathfrak p}
 =a(J)r+\ord(r^2)+\ord(r^{-M}\eexp^{-c/r^2})
\]

and \(K_\theta(J)=S_\theta(J)/a(J)\), the root displacement has the
coefficient asserted in \cref{thm:matched-threshold}.  The
full-rectangle error follows by dividing
\eqref{eq:matched-response-remainder} by the flow slope.  At fixed \(r\),
\eqref{eq:matched-fixed-r-response} and the ordinary implicit-function
expansion give the \(O_{r,\mathfrak p}(h^3)\) remainder.  Existence and
uniqueness of the actual-map root come from the monotonicity part of
\cref{lem:root-transfer}, not from the \(C^0\) Gaussian estimate alone.
\end{proof}

\section{Physical consequences and interpretation}
\label{sec:applications}

\subsection{The coefficient in physical jets}

The affine covariance in \cref{prop:affine-normalization} turns the
normalized theorem into a statement about the system as it is modelled and
discretized.  The resulting coefficient can be written without retaining
any of the auxiliary normalization scales.  For the physical germ
\eqref{eq:physical-germ}, define
\begin{equation}
 \label{eq:physical-jet-functional}
 \mathcal X(f,g)
 =\frac{f_{uuu}}{3f_{uu}^{\,2}}
  -\frac{f_{uv}}{f_{uu}f_v}
  -\frac{g_{uu}}{f_{uu}g_u}
  +\frac{g_v}{f_vg_u},
\end{equation}
where all derivatives are evaluated at the singular canard point.

\begin{corollary}[Threshold law in physical coordinates]
 \label{cor:physical-threshold}
Suppose that the physical germ \eqref{eq:physical-germ} satisfies
\eqref{eq:physical-fold-conditions} and
\(\mathcal T\ne0\), and form the exact Runge--Kutta map in the physical
variables with step \(k\).  On the fold-following section
\eqref{eq:physical-fold-following-section}, the leading coefficient is
\begin{equation}
 \label{eq:physical-jet-coefficient}
 \boxed{
 K_\theta^{\phys}
 =\frac{3\beta_\theta f_v^{\,2}g_u^{\,3}}{2\mathcal T}
   \mathcal X(f,g).}
\end{equation}
For independently chosen actual continuations obtained as affine images of
admissible normalized selections, there are
constants \(C,c>0\) and an integer \(M\) such that
\begin{equation}
 \label{eq:physical-all-selection-law}
 \mu_{\rk,\theta}^{\sigma_m}-\mu_{\fl}^{\sigma_f}
 =K_\theta^{\phys}k^2\eta^2
  +\ord(k^2\eta^{5/2}+k^3\eta^{5/2})
  +\ord\!\left(\eta^{-M/2}\eexp^{-c/\eta}\right).
\end{equation}
The roots exist and are unique throughout a full small-
\((k,\sqrt\eta)\) rectangle.  After division by \(k^2\eta^2\), the
remainder tends to zero on every positive power-law scale
\(k=\eta^q\), \(q>0\), including \(k\ll\eta\).

If the physical continuations are the affine images of a uniformly
second-order fold-matched rule \(\mathfrak p\), then
\begin{equation}
 \label{eq:physical-matched-law}
 \left|
 \mu_{\rk,\theta}^{\mathfrak p}-\mu_{\fl}^{\mathfrak p}
 -K_\theta^{\phys}k^2\eta^2
 \right|
 \le C_{\mathfrak p}k^2\eta^2
 \left(\sqrt\eta+k\sqrt\eta
       +\eta^{-M/2}\eexp^{-c/\eta}\right),
\end{equation}
after changing the constants to physical units.  For each fixed sufficiently
small \(\eta>0\),
\begin{equation}
 \label{eq:physical-fixed-eta-law}
 \mu_{\rk,\theta}^{\mathfrak p}-\mu_{\fl}^{\mathfrak p}
 =k^2\kappa_{\theta,\phys}^{\mathfrak p}(\eta)
  +\ord_{\eta,\mathfrak p}(k^3),
 \qquad k\downarrow0,
\end{equation}
with
\(
 \kappa_{\theta,\phys}^{\mathfrak p}(\eta)
 =K_\theta^{\phys}\eta^2
  +\ord(\eta^{5/2}+\eta^{-M/2}\eexp^{-c/\eta})
\).
\end{corollary}

\begin{proof}
By \cref{prop:affine-normalization}, the physical and normalized
Runge--Kutta stage equations are exactly conjugate, their splittings differ
by the constant factor \(a_y\), and their roots satisfy
\(\mu_\bullet=a_\lambda\lambda_\bullet\).  Substituting
\eqref{eq:affine-normalized-coefficients} into
\eqref{eq:threshold-coefficient} and then using
\cref{eq:affine-scales,eq:affine-fold-shear} gives
\eqref{eq:physical-jet-coefficient}; the dependence on the arbitrary
constant time gauge \(\tau\) cancels.  Transforming
\cref{thm:all-selection} with
\(h=k/\tau\) and \(\eps=\eta/a_\eps\) yields
\eqref{eq:physical-all-selection-law}.  An affine conjugacy preserves
actual invariance and the matching bounds up to fixed factors, so
\cref{thm:matched-threshold} gives
\cref{eq:physical-matched-law,eq:physical-fixed-eta-law}.
\end{proof}

The corollary is invariant under the admissible choice of affine fold gauge,
but not under an arbitrary nonlinear coordinate change.  It therefore
resolves the coordinate issue at the level at which Runge--Kutta stage
equations genuinely commute with a change of variables.

\subsection{The right fold of van der Pol}

Consider the fast-time van der Pol system
\begin{equation}
 \label{eq:vdp-physical-system}
 \dot x=y-\frac{x^3}{3}+x,
 \qquad
 \dot y=\eps(a-x).
\end{equation}
The affine coordinates
\begin{equation}
 \label{eq:vdp-physical-affine-map}
 u=1-x,
 \qquad v=y+\frac23,
 \qquad \mu=1-a
\end{equation}
preserve fast time and put the right fold into the already normalized form
\begin{equation}
 \label{eq:vdp-normalized-fold}
 \dot u=u^2-v-\frac13u^3,
 \qquad
 \dot v=\eps(u-\mu).
\end{equation}
Here \(\mathcal T=-2\), \(\mathcal X=-1/6\), and hence the physical
coefficient for \(\mu\) is \(\beta_\theta/8\).  Reversing the parameter
through \(a=1-\mu\) gives the following specialization.

\begin{corollary}[Local van der Pol threshold]
 \label{cor:vdp}
For every method in a fixed compact order-two Runge--Kutta family and every
admissible pair of actual local continuations, the selected right-fold
thresholds satisfy
\begin{equation}
 \label{eq:vdp-threshold-law}
 a_{\rk,\theta}^{\loc,\sigma_m}
 -a_{\fl}^{\loc,\sigma_f}
 =-\frac{\beta_\theta}{8}h^2\eps^2
  +\ord(h^2\eps^{5/2}+h^3\eps^{5/2})
  +\ord\!\left(\eps^{-M/2}\eexp^{-c/\eps}\right).
\end{equation}
The relative conclusion of \cref{thm:all-selection} holds on every
positive power-law scale \(h=\eps^q\), \(q>0\).  For a fold-matched rule,
all conclusions of
\cref{thm:matched-threshold} hold as well; in particular, at each fixed
sufficiently small \(\eps>0\), the actual numerical threshold converges to
its paired flow threshold with order two as \(h\to0\).
\end{corollary}

For a numerical approximation to one fold-matched observable, consider the
self-adjoint two-stage order-two family
\begin{equation}
 \label{eq:vdp-two-stage-family}
 \Amat_\rho=
 \begin{pmatrix}
  \frac18+2\rho&\frac18-2\rho\\
  \frac38+2\rho&\frac38-2\rho
 \end{pmatrix},
 \qquad
 b=\frac12\binom11,
 \qquad
 \cvec=\binom{1/4}{3/4}.
\end{equation}
It has \(\beta_\rho=1/12-\rho\), so
\begin{equation}
 \label{eq:vdp-rho-coefficient}
 \frac{a_{\rk,\rho}^{\loc,\sigma_m}
       -a_{\fl}^{\loc,\sigma_f}}
      {h^2\eps^2}
 \longrightarrow \frac{12\rho-1}{96},
 \qquad \eps\downarrow0,\quad h=\eps^q,\quad q>0,
\end{equation}
uniformly over the admissible selection labels \(\sigma_f,\sigma_m\).
Equation \eqref{eq:vdp-rho-coefficient} concerns the actual local thresholds
of \cref{cor:vdp}.  We instantiate the common-completion construction in
\cref{prop:matched-existence} as follows.  On each oriented circle
\(\xi\in\R/(2\mathbb Z)\), put \(u=\eta_q\xi\), where
\(\eta_{\att}=1\) and \(\eta_{\rep}=-1\), and use
\(v=\phi(u)+\eps V\).  The local collar is
\(I=[-1/2,-1/8]\), and the completed oriented field is
\begin{equation}
 \label{eq:vdp-numerical-completion}
 \begin{aligned}
  \dot\xi&=\eps\widetilde A, & \widetilde A&=-V,\\
  \dot V&=\widetilde B_q,
  & \widetilde B_q
   &=\chi(\xi)\{\xi-\eta_q\eps L
        +(2\xi-\eta_q\xi^2)V\}\\
  &&&\quad-\frac{1-\chi(\xi)}2\left(V+\frac35\right).
 \end{aligned}
\end{equation}
Here \(\chi\) is a fixed smooth cutoff equal to one on
\([-0.55,-0.09]\) and supported in \([-0.625,-0.0625]\).
On the strip \(-6/5\le V\le-1/5\), its base drift has one sign,
\(\partial_V\widetilde B_q\le-31/256\), and both normal boundaries point
inward for the parameters used below.  If \(\widetilde F_h\) denotes the
time-\(h\) map of \eqref{eq:vdp-numerical-completion}, we complete the
numerical map by
\[
 \widetilde P_{\rho,h}
 =\widetilde F_h+\zeta(\xi)
       \{P_{\rho,h}^{\loc}-F_h^{\loc}\},
\]
where \(\zeta=1\) on \(I\), its support lies inside \(\{\chi=1\}\), and
\(P_{\rho,h}^{\loc}\) is formed first in the original affine variables.
The repelling branch uses the contained inverse.  Thus the completed map
equals the corresponding signed actual Runge--Kutta branch throughout the
local collar.

A periodic cubic-spline graph transform approximates the four completed
flow and map invariant graphs.  Their restrictions to the collar--bridge
overlap \(0.20\le |u|\le0.25\) are continued to \(u=0\) by the local flow
and the actual signed numerical branches.  Rather than subtracting two
nearby physical roots, we solve
\[
 \Delta_{\fl}^{\mathfrak p_{\rm cc}}(r^2L_{\fl})=0,
 \qquad
 \Delta_{\rk,\rho}^{\mathfrak p_{\rm cc}}
   \bigl(r^2[L_{\fl}+h^2r^2Q]\bigr)=0.
\]
Consequently the plotted coefficient is exactly \(-Q\) within the numerical
solve.  The retained code, tabular data, graph samples and metadata report
the off-grid local-collar and full-completion residuals, root slopes, refinement checks,
stage residuals and contained-inverse round trip.  No interval certification
is claimed.

\begin{figure}[!t]
 \centering
 \includegraphics[width=0.96\textwidth]
   {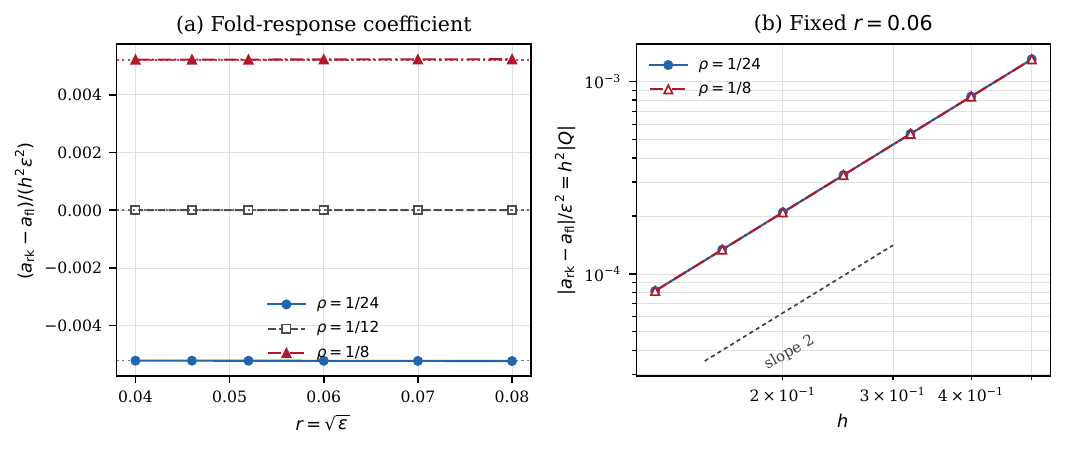}
 \caption{Numerical approximations to the local invariant-graph thresholds
 selected by the fixed common completion described in and following
 \eqref{eq:vdp-numerical-completion},
 computed with \(1024\) periodic graph nodes.  (a) For \(h=0.4\), normalized
 displacements for \(\rho=1/24,1/12,1/8\) (markers) and the predicted limits
 \((12\rho-1)/96\) (dotted lines).  (b) At fixed \(r=0.06\),
 \(\lvert a_{\rk,\rho}-a_{\fl}\rvert/\eps^2\) versus \(h\) for the two
 non-cancelling methods; their absolute-value curves overlap at plotting
 resolution, and the reference segment has slope two.  Connecting segments
 are visual guides.}
 \label{fig:vdp-invariant-graph-thresholds}
\end{figure}

\subsection{A conditional decomposition for a global matching root}

The following conditional decomposition isolates what a global theorem must
add to the local response; its hypotheses must be proved separately for each
global system.

\begin{remark}[Local and outer contributions]
 \label{rem:global-composition}
Let \(G_{\fl}(\lambda;r)\) and
\(G_{\rk}(\lambda;r,h,\theta)\) be scalar matching functions whose zeros
define the global flow and numerical thresholds.  Let \(q_r(\lambda)\)
measure how the local section splitting enters the global matching
condition, write \(\mathscr R_{\loc,\theta}^{\mathfrak p}\) for the local
response in \cref{lem:matched-step-response}, and let
\(\mathscr R_{\out,\theta}^{\mathfrak p}\) collect the order-\(h^2\)
response generated outside the fold neighbourhood.  Fix \(r>0\), suppress
the remaining arguments of \(G_{\fl}\) and \(G_{\rk}\), and suppose
\(G_{\fl}\) has a simple zero
\(\lambda_{\fl}^{\mathrm g}=r^2L_{\mathrm g}\), with \(L_{\mathrm g}\) in
the local response window, and set
\(D_{\mathrm g}=G_{\fl}'(\lambda_{\fl}^{\mathrm g})\ne0\).  If a global
analysis gives the first relation below in \(C^1\) near that zero, with the
local term supplied by the matched theorem and the outer term independent
of \(h\), then the scalar implicit-function argument gives the second:
\begin{equation}
 \label{eq:global-composition}
 \begin{aligned}
 G_{\rk}-G_{\fl}
 &=-h^2\{q_r(\lambda)\mathscr R_{\loc,\theta}^{\mathfrak p}(\lambda;r)
          +\mathscr R_{\out,\theta}^{\mathfrak p}(\lambda;r)\}
   +\ord_r(h^3),\\
 \lambda_{\rk}^{\mathrm g}-\lambda_{\fl}^{\mathrm g}
 &=\frac{h^2}{D_{\mathrm g}}
   \{q_r(\lambda_{\fl}^{\mathrm g})
        \mathscr R_{\loc,\theta}^{\mathfrak p}
             (\lambda_{\fl}^{\mathrm g};r)
     +\mathscr R_{\out,\theta}^{\mathfrak p}
             (\lambda_{\fl}^{\mathrm g};r)\}
   +\ord_r(h^3).
 \end{aligned}
\end{equation}
If, in addition, \(D_{\mathrm g}/r\to d_0\ne0\),
\(q_r(\lambda_{\fl}^{\mathrm g})\to q_0\), and the two responses divided
by \(r^5\) tend to \(S_\theta(J)\) and \(S_{\out,\theta}\), respectively,
then
\(\lim_{r\downarrow0}r^{-4}\lim_{h\downarrow0}
(\lambda_{\rk}^{\mathrm g}-\lambda_{\fl}^{\mathrm g})/h^2
=(q_0S_\theta(J)+S_{\out,\theta})/d_0\).
Thus the fold term alone controls the global coefficient only after an
outer cancellation, such as \(S_{\out,\theta}=0\), has been established.
This remark neither constructs an outer return nor proves a global canard
theorem.
\end{remark}

\subsection{Standard methods under the fold-response filter}

The method class in the theorems is substantially broader than the
two-stage family used in \cref{fig:vdp-invariant-graph-thresholds}.  The following
values illustrate how methods with different pointwise residuals can be
indistinguishable to the leading threshold observable.
\Cref{tab:standard-method-defects} lists the two defects and the method
multiplier \(c_\theta=-3\beta_\theta/4\), defined without division by
\(\XiJ(J)\) through
\(K_\theta(J)=c_\theta\XiJ(J)/(D+2)\).

\begin{table}[!ht]
 \centering
 \small
 \caption{Order-three defects and the method multiplier in the leading
 threshold coefficient.}
 \label{tab:standard-method-defects}
\begin{tabular}{lccc}
 \toprule
 Method&\(\alpha_\theta\)&\(\beta_\theta\)&\(c_\theta\)\\
 \midrule
 Explicit midpoint&\(-1/24\)&\(-1/6\)&\(1/8\)\\
 Heun RK2&\(1/12\)&\(-1/6\)&\(1/8\)\\
 Ralston RK2&\(0\)&\(-1/6\)&\(1/8\)\\
 Kutta RK3&\(0\)&\(0\)&\(0\)\\
 Classical RK4&\(0\)&\(0\)&\(0\)\\
 Implicit midpoint&\(-1/24\)&\(1/12\)&\(-1/16\)\\
 Trapezoidal rule&\(1/12\)&\(1/12\)&\(-1/16\)\\
 Two-stage family, \(\rho=1/12\)&\(-1/96\)&\(0\)&\(0\)\\
 \bottomrule
\end{tabular}
\end{table}

In particular, the last row has strict classical order two but its leading
threshold shift vanishes.  Conversely, the equality of the three explicit
RK2 coefficients does not mean that their local errors agree: their bushy
defects are different, and the response functional annihilates those
differences.

\FloatBarrier

\section{Conclusion and outlook}
\label{sec:conclusion}

The central conclusion is geometric.  The order-three local error of an
order-two Runge--Kutta method has two independent rooted-tree defect
coordinates, and both enter the general pointwise one-step residual.
Transport through a fast--slow fold does not pass those coordinates
unchanged to the maximal-canard threshold.  Its Gaussian response annihilates
the bushy-tree direction, so only the chain-tree direction can contribute,
yielding
\((\alpha,\beta)\mapsto-3\beta\XiJ(J)/8\).  The leading joint-fold bias can
therefore vanish under a condition weaker than classical third order.  This
is a cancellation in one nonlinear observable, not an elevation of the
trajectory order or, in general, of the fixed-\(\eps\) threshold order.

The observable is local, but its algebraic response is robust at precisely
the scale resolved by the theorem.  Independently continued slow manifolds
can change the local threshold by an exponentially small amount; once that
selection term is negligible relative to \(h^2\eps^2\), every admissible
choice has the same leading coefficient.  Fold matching supplies a coherent
comparison for the ordinary limit \(h\to0\) at fixed \(\eps\), although its
complete fixed-parameter coefficient may retain outer boundary data.  The
affine covariance result transfers the local coefficient to physical fold
germs, but not through arbitrary nonlinear changes of state.  Likewise, the
conditional decomposition in \cref{rem:global-composition} identifies how
the fold response enters a global matching equation; a global return can
contribute at the same order and must be analysed for the system at hand.

Output-dependent accuracy is familiar in numerical analysis.  What is
specific here is that a singular dynamical passage gives an explicit
response functional on the Runge--Kutta defect space for a threshold defined
by actual invariant manifolds.  Other nonhyperbolic passages---including
transcritical passage and delayed loss of stability---have different
variational kernels and may therefore retain different combinations of
B-series defects.  Determining those functionals, and coupling them to
natural global return conditions, are the next geometric questions.

\appendix
\section{Affine normalization and exact Runge--Kutta covariance}
\label{app:affine-normalization}

This appendix proves \cref{prop:affine-normalization,cor:uniform-physical-families}.
The calculation is included because it identifies the precise physical
nondegeneracy condition and because exact affine covariance, rather than a
formal coordinate argument, is needed for the numerical threshold.

\subsection{The fold curve and the unfolding determinant}

At \(\eta=0\), the singular fold curve is determined locally by
\[
 f(u_{\rm fold}(\mu),v_{\rm fold}(\mu),0,\mu)=0,
 \qquad
 f_u(u_{\rm fold}(\mu),v_{\rm fold}(\mu),0,\mu)=0.
\]
The Jacobian of these two equations with respect to \((v,u)\) is triangular
at the origin, with diagonal entries \(f_v\) and \(f_{uu}\).  The implicit
function theorem and \eqref{eq:physical-fold-conditions} therefore give a
unique analytic fold curve.  Differentiation at \(\mu=0\) yields
\begin{equation}
 \label{eq:app-fold-curve-derivatives}
 v_{\rm fold}'(0)=-\frac{f_\mu}{f_v},
 \qquad
 u_{\rm fold}'(0)
 =-\frac{1}{f_{uu}}
    \left(f_{u\mu}-\frac{f_{uv}f_\mu}{f_v}\right)
 =-\frac{\widehat f_{u\mu}}{f_{uu}}.
\end{equation}
Consequently,
\begin{equation}
 \label{eq:app-slow-drift-transversality}
 \begin{aligned}
 \left.\frac{\dd}{\dd\mu}
 g(u_{\rm fold}(\mu),v_{\rm fold}(\mu),0,\mu)
 \right|_{\mu=0}
 &=g_\mu+g_u u_{\rm fold}'(0)+g_v v_{\rm fold}'(0)\\
 &=\widehat g_\mu-\frac{g_u}{f_{uu}}\widehat f_{u\mu}
 =\frac{\mathcal T}{f_{uu}}.
 \end{aligned}
\end{equation}
This proves the geometric interpretation following
\eqref{eq:physical-unfolding-determinant}.

\subsection{The normalized Taylor jet}

Let \(P=\diag(a_x,a_y)\), and write the state part of
\eqref{eq:affine-normalization-map} as
\[
 \mathcal A_{\eps,\lambda}(z)=Pz+q_{\eps,\lambda},
 \qquad z=(x,y),
\]
where
\[
 q_{\eps,\lambda}
 =\left(
 \ell\lambda,
 -\frac{f_\eta}{f_v}a_\eps\eps
 -\frac{f_\mu}{f_v}a_\lambda\lambda
 \right).
\]
Since \(s=\tau t\), the transformed vector field is
\begin{equation}
 \label{eq:app-transformed-vector-field}
 F_J(z;\eps,\lambda)
 =\tau P^{-1}
 \begin{pmatrix}
  f(\mathcal A_{\eps,\lambda}(z),a_\eps\eps,
                                      a_\lambda\lambda)\\
  a_\eps\eps\,
  g(\mathcal A_{\eps,\lambda}(z),a_\eps\eps,
                                      a_\lambda\lambda)
 \end{pmatrix}.
\end{equation}

The first three scales in \eqref{eq:affine-scales} satisfy
\begin{equation}
 \label{eq:app-leading-normalizations}
 \frac{\tau a_xf_{uu}}2=1,
 \qquad
 \frac{\tau a_yf_v}{a_x}=-1,
 \qquad
 \frac{\tau a_\eps a_xg_u}{a_y}=1.
\end{equation}
The orientation assumption \(f_vg_u<0\) gives \(a_\eps>0\).  The pure
\(\eps\)- and \(\lambda\)-terms in the fast component vanish because of
the second component of \(q_{\eps,\lambda}\) and \(f_u=0\).

It remains to check the two terms involving the distinguished parameter.
The coefficient of \(x\lambda\) in the fast equation and the coefficient of
\(\lambda\) inside the slow bracket are, respectively,
\begin{equation}
 \label{eq:app-parameter-coefficients}
 \tau\left(f_{uu}\ell+\widehat f_{u\mu}a_\lambda\right),
 \qquad
 \frac{1}{a_xg_u}
 \left(g_u\ell+\widehat g_\mu a_\lambda\right).
\end{equation}
Using \eqref{eq:affine-scales}, \eqref{eq:affine-fold-shear}, and
\(\mathcal T=f_{uu}\widehat g_\mu-g_u\widehat f_{u\mu}\), one obtains
\begin{equation}
 \label{eq:app-parameter-normalizations}
 f_{uu}\ell+\widehat f_{u\mu}a_\lambda=0,
 \qquad
 g_u\ell+\widehat g_\mu a_\lambda=-a_xg_u.
\end{equation}
Thus the two coefficients in \eqref{eq:app-parameter-coefficients} are
\(0\) and \(-1\).  This proves that \(D=0\) and that the slow equation
contains \(x-\lambda\).

The coefficients of \(x^3\), \(xy\), \(x\eps\), \(x^2\), \(y\), and
\(\eps\) are obtained by differentiating
\eqref{eq:app-transformed-vector-field}; they are exactly
\eqref{eq:affine-normalized-coefficients}.  All Taylor monomials of lower
weight have now been listed.  Taylor's formula therefore puts every
remaining fast monomial in weighted degree at least four and every remaining
term inside the slow bracket in weighted degree at least three.  Analyticity
and reality under conjugation are preserved by the real affine map.  This
proves the normal-form part of \cref{prop:affine-normalization}.

For completeness, the fixed-section choice in
\cref{rem:fixed-physical-section} follows by setting \(\ell=0\).  The slow
coefficient in \eqref{eq:app-parameter-coefficients} equals \(-1\) for
\(a_\lambda=-a_xg_u/\widehat g_\mu\), while the fast coefficient becomes
\[
 D=\tau\widehat f_{u\mu}a_\lambda
  =-\frac{2g_u\widehat f_{u\mu}}
          {f_{uu}\widehat g_\mu}.
\]
Adding \(2\) and using the definition of \(\mathcal T\) gives the
second identity in \eqref{eq:fixed-section-D}.

\subsection{Exact stage conjugacy and the section splitting}

We first verify the contained inverse identity
\eqref{eq:rk-contained-inverse}.  Suppose the stages \(Z_i\) take an input
\(z_0\) to
\[
 z_1=z_0+h\sum_j b_jF_J(Z_j)
\]
under the tableau \((\Amat,b,\cvec)\).  Starting at \(z_1\), apply the
adjoint tableau \eqref{eq:rk-adjoint} with step \(-h\).  Substitution of the
same stages gives
\[
 z_1-h\sum_j(b_j-a_{ij})F_J(Z_j)
 =z_0+h\sum_j a_{ij}F_J(Z_j)=Z_i,
\]
and its output is
\(z_1-h\sum_i b_iF_J(Z_i)=z_0\).  Analytic continuation from \(h=0\)
selects these stages on the contained branches and proves the exact inverse
relation.

Let \(\mathcal F^{\phys}=(f,\eta g)\).  Equation
\eqref{eq:app-transformed-vector-field} is equivalently
\begin{equation}
 \label{eq:app-vector-field-conjugacy}
 F_J(z;\eps,\lambda)
 =\tau P^{-1}\mathcal F^{\phys}
  (\mathcal A_{\eps,\lambda}(z);\eta,\mu).
\end{equation}
Let \(Z_i\) solve the normalized stage equations with step \(h\), and put
\(W_i=\mathcal A_{\eps,\lambda}(Z_i)\).  Since the translation
\(q_{\eps,\lambda}\) is fixed during a step, multiplication of the stage
equations by \(P\), followed by \eqref{eq:app-vector-field-conjugacy}, gives
\begin{equation}
 \label{eq:app-physical-stage-conjugacy}
 W_i=\mathcal A_{\eps,\lambda}(z)
 +\tau h\sum_j a_{ij}(\theta)
       \mathcal F^{\phys}(W_j;\eta,\mu).
\end{equation}
Thus the \(W_i\) are exactly the physical stages with step
\(k=\tau h\).  The same calculation for the weighted stage sum gives the
output identity \eqref{eq:rk-affine-conjugacy}.  No expansion in \(h\) has
been used.

The affine conjugacy carries invariant graph germs to invariant graph
germs.  On \(x=0\), the physical fast coordinate is
\(u=\ell\lambda=(\ell/a_\lambda)\mu\), which proves
\eqref{eq:physical-fold-following-section}.  The parameter-dependent
translation in the slow coordinate is common to the attracting and
repelling graphs, whereas their vertical difference is multiplied by
\(a_y\).  Therefore
\[
 \Delta_\bullet^{\phys}(\mu)
 =a_y\Delta_\bullet(\mu/a_\lambda),
\]
and the splitting zeros obey \(\mu_\bullet=a_\lambda\lambda_\bullet\).
Finally,
\begin{align*}
 \mu_{\rk}-\mu_{\fl}
 &=a_\lambda K_\theta(J)
       \left(\frac{k}{\tau}\right)^2
       \left(\frac{\eta}{a_\eps}\right)^2
   +o(k^2\eta^2)\\
 &=\frac{a_\lambda}{\tau^2a_\eps^2}
       K_\theta(J)k^2\eta^2+o(k^2\eta^2),
\end{align*}
which proves \eqref{eq:physical-coefficient-transform} and completes the
proof of \cref{prop:affine-normalization}.

\subsection{Uniformity over physical families}

Consider the compact family in
\cref{cor:uniform-physical-families}.  The formulas
\eqref{eq:affine-scales}--\eqref{eq:affine-fold-shear} depend continuously
on the finite jet of \((f,g)\).  The uniform lower bounds therefore make
\[
 a_x,\quad a_y,\quad a_\eps,\quad a_\lambda,
 \quad\ell
\]
uniformly bounded, and their nonzero scaling factors have uniformly bounded
inverses.  After reducing a common normalized polydisc, all its affine images
are contained in the common physical holomorphic neighbourhood.  Taylor
projection onto the displayed weighted jets is continuous in the
\(H^\infty\) topology on nested polydiscs.  Hence the transformed
coefficients and remainders form a compact set with a common holomorphic
bound.  Since the fold-following normalization has \(D=0\), its unfolding
bound is \(D+2=2\).  The image is therefore a normalized analytic fold class
as in \cref{def:normalized-fold-class}.  The scale conversions in
\eqref{eq:physical-normalized-variables} are uniform, proving
\cref{cor:uniform-physical-families}.

\section{Uniform invariant-graph geometry}
\label{app:fold-geometry}

This appendix proves \cref{prop:fold-geometry}.  The construction is local
in the state variables but uniform over the compact analytic class and the
compact Runge--Kutta family.  We first construct invariant graphs on
normally hyperbolic collars, continue them to the fold by regraphing,
and then prove that the fold transport suppresses the freedom in the outer
continuation at a Gaussian rate.

Recall \(L_0(J)\) from \eqref{eq:main-L0-definition}.  It is continuous and
bounded on \(\Gclass\).  Fix \(L_*>0\) so that
\(\abs{L_0(J)}\le L_*-2\) for every \(J\in\Gclass\), and put
\(\Lambda=[-L_*,L_*]\).  All reductions below preserve this fixed
parameter interval.

\subsection{Admissible invariant graphs}

Choose \(\delta>0\) uniformly small and set
\begin{equation}
 \label{app-geom:side-intervals}
 \begin{array}{c|c|c|c}
 q&I_q^{\rm col}&I_q^{\rm br}&I_q^{\rm ov}\\ \hline
 \att&[-2\delta,-\delta/2]&[-\delta,0]&[-\delta,-\delta/2]\\
 \rep&[\delta/2,2\delta]&[0,\delta]&[\delta/2,\delta].
 \end{array}
\end{equation}
Let \(\phi_J\) be the critical graph determined by
\begin{equation}
 \label{app-geom:critical-graph}
 f_J(x,\phi_J(x),0,0)=0,
 \qquad \abs{x}\le3\delta.
\end{equation}
Compactness and \(\partial_yf_J(0)=-1\) give a common analytic
neighbourhood of these graphs.

\begin{definition}[Component identities for admissible selections]
 \label{app-geom:selection-definition}
Use the admissibility definition and bounds in
\cref{def:admissible-selection,eq:main-selection-bounds}, with the intervals
in \eqref{app-geom:side-intervals}.  Denote the pasted graph by
\(\widehat m_{\bullet,q}\) and its central value by
\begin{equation}
 \label{app-geom:core-value}
 y_{\bullet,q}(r^2L)=u_{\bullet,q}(0;L).
\end{equation}

For the flow, admissibility means tangency and directed drift:
\begin{equation}
 \label{app-geom:flow-graph-equation}
 r^2g_J(x,\widehat m_{\fl,q};r^2,r^2L)
 =\partial_x\widehat m_{\fl,q}
  f_J(x,\widehat m_{\fl,q};r^2,r^2L),
 \qquad
 c_dr^2\le f_J\le C_dr^2.
\end{equation}
For the map, let \(\Phi^+\) be the positive-step branch of
\eqref{eq:normalized-rk-map} and let
\(\Phi^-=\Phi_{-h,\theta^\dagger}\) be its contained inverse.  Whenever
the source and target remain on the pasted graph, both maps satisfy their
graph invariance equations.  More precisely, with
\(z_q(x)=(x,\widehat m_{\rk,q}(x))\) and
\(T_q^\pm(x)=\pi_x\Phi^\pm(z_q(x))\),
\begin{equation}
 \label{app-geom:map-graph-equations}
 \Phi^\pm(z_q(x))=z_q(T_q^\pm(x))
\end{equation}
on every contained source--target pair, and
\begin{equation}
 \label{app-geom:map-inverse-identities}
 \Phi^-\circ\Phi^+=\operatorname{id},
 \qquad
 \Phi^+\circ\Phi^-=\operatorname{id}.
\end{equation}
The common stage neighbourhood and positive-step drift required in
\cref{def:admissible-selection} take the component form
\begin{equation}
 \label{app-geom:map-drift}
 c_dhr^2\le T_q^+(x)-x\le C_dhr^2.
\end{equation}
\end{definition}

The construction below fixes \(B_{\rm sel},r_0,h_0\) for which the classes
in \cref{def:admissible-selection} are nonempty.  Those classes retain more continuations than the
constructed references because the exponentially small freedom in a local
slow manifold is part of the theorem's quantifiers.

\subsection{The divided flow equation}

Put \(\nu=r^2\) temporarily and write
\begin{equation}
 \label{app-geom:divided-coordinate}
 \eps=\nu,
 \qquad \lambda=\nu L,
 \qquad y=\phi_J(x)+\nu v.
\end{equation}
The critical equation in \eqref{app-geom:critical-graph} makes
\(f_J(x,\phi_J(x)+\nu v;\nu,\nu L)\) analytically divisible by \(\nu\).
Consequently the normalized flow has the divided form
\begin{equation}
 \label{app-geom:divided-flow}
 \dot x=\nu\mathcal A_J(x,v,L,\nu),
 \qquad
 \dot v=\mathcal B_J(x,v,L,\nu).
\end{equation}

On the critical graph define
\begin{equation}
 \label{app-geom:critical-coefficients}
 \begin{aligned}
  a_J(x)&=\partial_\eps f_J(x,\phi_J(x);0,0),
  &b_J(x)&=\partial_\lambda f_J(x,\phi_J(x);0,0),\\
  c_J(x)&=\partial_y f_J(x,\phi_J(x);0,0),
  &q_J(x)&=g_J(x,\phi_J(x);0,0).
 \end{aligned}
\end{equation}
All apparent singularities at \(x=0\) in
\begin{equation}
 \label{app-geom:reference-graph}
 Q_J(x)=\frac{q_J(x)}{\phi_J'(x)},
 \qquad
 \Gamma_J^0(x,L)=\frac{Q_J(x)-a_J(x)-Lb_J(x)}{c_J(x)},
 \qquad
 \kappa_J(x)=-c_J(x)\frac{\phi_J'(x)}x
\end{equation}
are removable.  After reducing \(\delta\), compactness gives constants
independent of \(J\) such that
\begin{equation}
 \label{app-geom:coercivity}
 0<Q_-\le Q_J(x)\le Q_+,
 \qquad
 0<\kappa_-\le\kappa_J(x)\le\kappa_+.
\end{equation}

Set \(v=\Gamma_J^0+e\).  Analytic division of the graph equation for
\eqref{app-geom:divided-flow} gives
\begin{equation}
 \label{app-geom:scalar-fold-equation}
 \nu\mathcal A(x,e,L,\nu)\partial_xe
 =x\kappa_J(x)e+\nu\mathcal F_J(x,e,L,\nu),
 \qquad
 0<A_-\le\mathcal A\le A_+.
\end{equation}
Thus the only loss of normal hyperbolicity is the factor \(x/\nu\).  In
the signed coordinate
\begin{equation}
 \label{app-geom:signed-coordinate}
 x=\varsigma rz,
 \qquad
 \varsigma=-1\quad(q=\att),
 \qquad
 \varsigma=1\quad(q=\rep),
\end{equation}
the coreward homogeneous propagator is bounded by
\begin{equation}
 \label{app-geom:gaussian-propagator}
 \abs{G(z,t)}\le C\exp\{-c(t^2-z^2)\},
 \qquad 0\le z\le t\le\delta/r.
\end{equation}

\subsection{Collar graphs and continuation to the fold}

We first record two facts about a general tableau needed in the collar
construction.

\begin{lemma}[Uniform divided stages and completed collar maps]
 \label{app-geom:completed-collar-map}
Let \(q\in\{\att,\rep\}\).  On the attracting collar take the positive
step with tableau \(\theta\); on the repelling collar take the negative
step with tableau \(\theta^\dagger\).  After reducing \(r_0,h_0\), the
corresponding stage branch exists on a fixed divided-coordinate
neighbourhood and has the form
\begin{equation}
 \label{app-geom:local-divided-map}
 P_{q,p}(w,v)=
 \bigl(w+h\nu\mathcal A_{q,p}(w,v),
       v+h\mathcal B_{q,p}(w,v)\bigr),
 \qquad w=(x,L),
\end{equation}
where the \(L\)-component of \(\mathcal A_{q,p}\) is zero and
\(p=(J,\nu,h,\theta)\).  The factors and their state derivatives through
order eight are bounded uniformly.  On a fixed normal strip one can
complete \eqref{app-geom:local-divided-map} to a global diffeomorphism
\(\widetilde P_{q,p}\), equal to the local branch on a smaller fixed
neighbourhood, such that the strip is invariant, every fixed-\(v\) base map is
invertible, and
\begin{equation}
 \label{app-geom:completed-map-derivatives}
 \begin{aligned}
  \mathcal L_{11},\mathcal L_{12}&\le Ch\nu,&
  \mathcal L_{21}&\le Ch,&
  \mathcal L_{22}&\le1-\kappa h,\\
  \norm{D^j(h\nu\widetilde{\mathcal A}_{q,p})}&\le C_jh\nu
       &&(1\le j\le8),&
  \norm{D_w(v+h\widetilde{\mathcal B}_{q,p})}&\le C_1h,\\[-1mm]
  \norm{D^j(v+h\widetilde{\mathcal B}_{q,p})}&\le C_jh
       &&(2\le j\le8).
 \end{aligned}
\end{equation}
Here \(\mathcal L_{ij}\) are the nonnegative Lipschitz constants of the
four base--normal blocks, and the identity in the first normal derivative
is excluded from the last two bounds.
\end{lemma}

\begin{proof}
Let \(s_*\) be the maximal number of stages in the family and set
\begin{equation}
 \label{app-geom:absolute-tableau-bounds}
 B_A=\sup_{\vartheta,i}\sum_m\abs{a_{im}(\vartheta)},
 \qquad
 B_b=\sup_\vartheta\sum_i\abs{b_i(\vartheta)},
\end{equation}
where \(\vartheta\) ranges over \(\ThetaRK\) and its adjoint image.
Both numbers are finite, and all estimates below use these absolute sums.

Write an input and its stages as
\[
 y=\phi_J(x)+\nu v,
 \qquad
 Y_i=\phi_J(X_i)+\nu V_i.
\]
There is an analytic function \(\widehat f_J\), on a fixed neighbourhood
of the divided collar, for which
\[
 f_J(x,\phi_J(x)+\nu v;\nu,\nu L)
 =\nu\widehat f_J(x,v,L,\nu).
\]
For a step \(\tau\), the stage equations are therefore exactly
\begin{equation}
 \label{app-geom:divided-stage-system}
 \begin{aligned}
  X_i&=x+\tau\nu\sum_m a_{im}(\vartheta)
        \widehat f_J(X_m,V_m,L,\nu),\\
  V_i&=v+\tau\sum_m a_{im}(\vartheta)
       g_J(X_m,\phi_J(X_m)+\nu V_m;\nu,\nu L)\\
     &\hspace{17mm}
       -\frac{\phi_J(X_i)-\phi_J(x)}{\nu}.
 \end{aligned}
\end{equation}
The last quotient has the nonsingular representation
\begin{equation}
 \label{app-geom:stage-phi-division}
 \frac{\phi_J(X_i)-\phi_J(x)}{\nu}
 =\frac{X_i-x}{\nu}
   \int_0^1\phi_J'(x+t(X_i-x))\dd t.
\end{equation}
Substitution of the first stage equation gives
\((X_i-x)/\nu=\tau\sum_m a_{im}\widehat f_J(X_m,V_m,L,\nu)\);
thus \eqref{app-geom:stage-phi-division} is analytic also at \(\nu=0\)
and carries the same factor \(\tau\).
Thus the right-hand side of the product stage system has derivative at
most \(C\abs\tau B_A\) on a common product neighbourhood.  If
\(C h_0B_A<1/2\), it is a contraction there and its Jacobian is a
uniformly invertible perturbation of the identity.  The unique branch
continuing from \(X_i=x,V_i=v\) consequently satisfies
\begin{equation}
 \label{app-geom:uniform-stage-increments}
 \max_i\abs{X_i-x}\le Ch\nu,
 \qquad
 \max_i\abs{V_i-v}\le Ch,
\end{equation}
with the same bounds after every state differentiation used here.
Thus all stages remain in one fixed neighbourhood, for explicit and
implicit tableaux alike.

The output equations, divided in the same way and estimated using
\(B_b\), give \eqref{app-geom:local-divided-map}.  On the attracting
collar, at \((\nu,h)=(0,0)\),
\(\partial_v\mathcal B_{\att,p}\) is the normal derivative of the divided
flow and is uniformly negative.  On the repelling collar the reversed
branch has the same property.  Hence there are nested fixed base neighbourhoods
\(K_q\Subset U_q\), a number \(R>1\), and \(\kappa>0\) such that
\begin{equation}
 \label{app-geom:local-normal-gap}
 \partial_v\mathcal B_{q,p}\le-4\kappa
 \quad\hbox{on }U_q\times[-R-2,R+2].
\end{equation}
The choice is uniform because the two tableau families and the analytic
fold class are compact.

Choose fixed smooth functions
\[
 \chi_q=1\ \hbox{on }K_q,
 \quad \operatorname{supp}\chi_q\Subset U_q,
 \qquad
 \zeta=1\ \hbox{on }[-R,R],
 \quad \operatorname{supp}\zeta\Subset(-R-2,R+2).
\]
Extend the normal derivative first:
\begin{equation}
 \label{app-geom:normal-derivative-extension}
 \begin{aligned}
  d_{q,p}(w,v)
  &=-2\kappa+\zeta(v)
     \{\partial_v\mathcal B_{q,p}(w,v)+2\kappa\},\\
  \overline{\mathcal B}_{q,p}(w,v)
  &=\mathcal B_{q,p}(w,0)+\int_0^v d_{q,p}(w,t)\dd t.
 \end{aligned}
\end{equation}
The product containing the local derivative is extended by zero outside
its compact support.  Then define
\begin{equation}
 \label{app-geom:deterministic-map-completion}
 \begin{aligned}
  \widetilde{\mathcal B}_{q,p}(w,v)
  &=\chi_q(w)\overline{\mathcal B}_{q,p}(w,v)
    +(1-\chi_q(w))(-2\kappa v),\\
  \widetilde{\mathcal A}_{q,p}(w,v)
  &=\chi_q(w)\zeta_A(v)\mathcal A_{q,p}(w,v),
 \end{aligned}
\end{equation}
where \(\zeta_A=1\) on \([-R,R]\) and has support in the larger normal
neighbourhood.  Formula \eqref{app-geom:normal-derivative-extension} is the point
of this construction: a direct cutoff of \(\mathcal B\) would introduce a
normal derivative of uncontrolled sign.  Here instead
\begin{equation}
 \label{app-geom:global-normal-gap}
 -C_B\le\partial_v\widetilde{\mathcal B}_{q,p}\le-2\kappa
\end{equation}
globally, while \(\widetilde{\mathcal B}_{q,p}+2\kappa v\) is bounded
with all required derivatives.

Choose \(R\) larger, if necessary, so that the normal component points
strictly inward at \(v=\pm R\).  For
\(h_0C_B<1\), its derivative is positive and at most
\(1-2\kappa h\), so the whole strip is invariant.  Both the full map and
each fixed-\(v\) base map are an identity plus a globally Lipschitz
perturbation of norm at most \(Ch\) and \(Ch\nu\), respectively.  Solving
the inverse equation by contraction proves global invertibility.  Finally,
differentiating \eqref{app-geom:deterministic-map-completion} gives
\eqref{app-geom:completed-map-derivatives}.  This completes the
construction of the stated maps.
\end{proof}

\begin{lemma}[Nonempty uniform selection classes]
 \label{app-geom:construction}
There are \(B_{\rm sel}<\infty\) and \(r_0,h_0>0\) for which the flow and
map selection classes of \cref{app-geom:selection-definition} are
nonempty.  The construction supplies reference graphs with all the
displayed bounds; its flow references have one additional uniform
\(x\)-derivative.
\end{lemma}

\begin{proof}
On either collar, the normal linearization of
\eqref{app-geom:divided-flow} is bounded away from zero.  We use positive
time on the attracting collar and reverse time on the repelling collar.
After a fixed completion of the vector fields, Nipp's bounded-domain
theorem applies directly: in its notation the base is \((x,L)\), the slow
parameter is \(\nu\), the normal variable is \(v\), and the reduced graph
is \(\Gamma_J^0\).  Its bounded-domain, smoothness, reduced-root, and
spectral-gap assumptions follow respectively from the fixed collar
neighbourhoods, analyticity, \eqref{app-geom:reference-graph}, and
\eqref{app-geom:coercivity}.  The parameter observation in that paper
allows \(L\) to be included as a base coordinate with zero velocity.  The
relevant source statements are assumptions 1--5 and parameter Remark 0,
Theorem 1, and the completion used in its proof
\cite[pp.~1--7, especially Theorem~1 on pp.~2--3]
 {Nipp1992Smooth}.  Uniform Cauchy bounds and the common normal gap make
all constants uniform over the compact class.  Taking the theorem through
total order seven gives the required \(C_x^4C_L^2\) bounds, and one more
\(x\)-derivative gives the stated \(C_x^5C_L^2\) reference bound.

For the numerical collars use the completed maps of
\cref{app-geom:completed-collar-map}.  To apply the Nipp--Stoffer theorem,
set
\begin{equation}
 \label{app-geom:NS-identification}
 X=\mathbb R^2_w,\qquad Y=\mathbb R_v,\qquad
 f_0(w)=w,\qquad
 \widehat f=h\nu\widetilde{\mathcal A}_{q,p},\qquad
 g=v+h\widetilde{\mathcal B}_{q,p}.
\end{equation}
The closed interval \(Y_R=[-R,R]\subset Y\) is the invariant normal strip
in Remark~0; it is not used as the ambient Banach space.  Assumption H1 and
that remark apply, with inverse constant \(\alpha=1\).  Theorem 3 supplies
the invariant Lipschitz graph;
Condition B4 and Theorem 5 give its smoothness.  In the printed report
these are H1 and Remark 0 on p.~3, conditions (9) and Theorem 3 on
pp.~7--12, and B4--B5 and Theorem 5 on pp.~14--15
\cite{NippStoffer1992}.  Condition B5 is immediate from
\(Df_0=I\).  The cited theorems apply for each \(h>0\); the following
calculation verifies their rate conditions and, crucially, makes their
derivative bounds uniform as \(h\downarrow0\).

Let \(d=1-\mathcal L_{11}-\mathcal L_{22}\).  From
\eqref{app-geom:completed-map-derivatives}, after reducing \(r_0,h_0\),
\begin{equation}
 \label{app-geom:strict-graph-gap}
 d\ge\frac34\kappa h,
 \qquad
 2\sqrt{\mathcal L_{12}\mathcal L_{21}}
 \le Ch\sqrt\nu<\frac14\kappa h<d.
\end{equation}
Thus the strict source condition \((9\mathrm a^*)\) holds.  If \(\ell\)
is the smaller root of its graph-slope quadratic, then
\[
 \ell=
 \frac{2\mathcal L_{21}}
 {d+\sqrt{d^2-4\mathcal L_{12}\mathcal L_{21}}}
 \le\ell_0
\]
for a class-uniform \(\ell_0\).  The source base and normal rates satisfy
\begin{equation}
 \label{app-geom:source-rates}
 \begin{aligned}
 \beta=1-\mathcal L_{11}-\mathcal L_{12}\ell
       &\ge1-C_*h\nu,\\
 N=\mathcal L_{22}+\mathcal L_{12}\ell
       &\le1-\frac34\kappa h.
 \end{aligned}
\end{equation}
For \(1\le j\le7\), Bernoulli's inequality gives
\(\beta^j\ge1-jC_*h\nu\).  Requiring
\(7C_*\nu\le\kappa/2\) yields
\begin{equation}
 \label{app-geom:B4-rates}
 N<\min\{1,\beta^j\},
 \qquad 1\le j\le7,
\end{equation}
which is precisely B4 through order seven.  In particular, B4(1)
supplies the remaining conditions (9b)--(9c).

We now prove the uniform jet estimate not supplied directly by the source
theorem.  Let
\[
 F_S(w)=w+h\nu\widetilde{\mathcal A}_{q,p}(w,S(w)),
 \qquad \mathscr H_S=F_S^{-1}.
\]
The target-aligned invariance equation is
\begin{equation}
 \label{app-geom:target-aligned-invariance}
 S(w)=S(\mathscr H_S(w))
  +h\widetilde{\mathcal B}_{q,p}
       (\mathscr H_S(w),S(\mathscr H_S(w))).
\end{equation}
After the terms containing the highest jet have been collected, its
\(j\)th derivative, \(2\le j\le7\), has the triangular form
\begin{equation}
 \label{app-geom:collar-jet-equation}
 D^jS
 =\mathscr W
   (D^jS\circ\mathscr H_S)[D\mathscr H_S]^{\otimes j}
   +\mathscr T_j,
 \qquad
 \norm{\mathscr W}\norm{D\mathscr H_S}^{j}
 \le\frac{N}{\beta^j}=:\chi_j.
\end{equation}
This is equations (29)--(32) in the source
\cite[pp.~21--22]{NippStoffer1992}.  To estimate its numerator here,
differentiate also the inverse-base identity
\(F_S\circ\mathscr H_S=\operatorname{id}\).  Once its term containing
\(D^jS\) has been put into \(\mathscr W\), the remaining terms obey
\begin{equation}
 \label{app-geom:inverse-base-jets}
 D\mathscr H_S=I+O(h\nu),
 \qquad
 D^i\mathscr H_S=O(h\nu)\quad(2\le i<j),
 \qquad
 (D^j\mathscr H_S)_{\rm lower}=O(h\nu),
\end{equation}
provided the lower graph jets are bounded.  Here ``lower'' denotes the
part left after the term linear in \(D^jS\) has been placed in the
principal operator.  Every term in
\(\mathscr T_j\) then contains at least one of the following:
\begin{enumerate}
 \item a derivative of order at least two of the normal output;
 \item a base derivative of its nonidentity part;
 \item a derivative of the base perturbation;
 \item a derivative \(D^i\mathscr H_S\) with \(i\ge2\).
\end{enumerate}
The normal identity contributes only to the first term in
\eqref{app-geom:collar-jet-equation}; it cannot occur in
\(\mathscr T_j\).  Equations
\eqref{app-geom:completed-map-derivatives} and
\eqref{app-geom:inverse-base-jets} therefore give
\begin{equation}
 \label{app-geom:closing-gap-forcing}
 \norm{\mathscr T_j}\le C_jh.
\end{equation}
Moreover, \eqref{app-geom:source-rates}--\eqref{app-geom:B4-rates}
give \(\chi_j\le1-\kappa_jh\).  Taking suprema in
\eqref{app-geom:collar-jet-equation} yields
\begin{equation}
 \label{app-geom:uniform-collar-jets}
 \norm{D^jS}
 \le\frac{C_jh}{1-\chi_j}\le C_j',
 \qquad 2\le j\le7.
\end{equation}
The \(C^0\) strip bound and the slope bound \(\ell\le\ell_0\) start the
induction.  For \(j\le6\), applying one difference quotient to the same
recurrence again differentiates a nonidentity map factor or a higher
inverse-base derivative, so its numerator is still \(O(h)\).  This is the
companion Lipschitz induction needed at the next order.  The eighth-order
map bounds in \eqref{app-geom:completed-map-derivatives} close the argument
through order seven.  Hence the derivative bounds remain uniform as the
normal gap closes like \(h\).

The completed invariant graph satisfies
\begin{equation}
 \label{app-geom:collar-location}
 \norm{S-\Gamma_J^0}_{C^0}\le C(h+\nu).
\end{equation}
Indeed, at \(\nu=0\) the base map is the identity, and the unique invariant
graph is the pointwise root of
\(\widetilde{\mathcal B}_{q,(J,0,h,\theta)}(w,v)=0\).  On the local
neighbourhood, consistency of the stage branch and the normal implicit-function
bound place this root within \(Ch\) of \(\Gamma_J^0\).  The maps at
\(\nu>0\) and \(\nu=0\), made with the same cutoffs, differ by
\(O(h\nu)\) in both components.  Corollary 4 of Nipp--Stoffer gives the
explicit fixed-graph comparison
\[
 \norm{S_\nu-S_0}_{C^0}
 \le\frac{\ell\,O(h\nu)+O(h\nu)}{1-N}
 \le C\nu,
 \qquad 1-N\ge\frac12\kappa h
\]
\cite[Corollary~4, pp.~12--14]{NippStoffer1992}.  This proves
\eqref{app-geom:collar-location} without a constant singular in \(h\).

Restriction to the local neighbourhood gives
\(m_{\rk,q}=\phi_J+\nu S\), with the collar estimates and branch
identities in \cref{app-geom:selection-definition}.  The graph-transform
identity is equality, not merely forward inclusion.  On the branch used to
construct a side it gives every graph target a unique graph predecessor;
global injectivity of the completion and the contained inverse identity
\eqref{eq:rk-contained-inverse} give the opposite branch identity whenever
both endpoints remain in the original normalized collar.  Finally, the base factor on
the graph converges uniformly to \(Q_J\).  Hence
\eqref{app-geom:coercivity} gives \eqref{app-geom:map-drift} after one
last reduction.  This argument uses consistency and absolute tableau
bounds, not positivity of \(b\).

It remains to reach \(x=0\).  Put \(X=x/r\) and \(H=hr\).  In the signed
coordinate \eqref{app-geom:signed-coordinate}, orient both sides toward the
fold by
\begin{equation}
 \label{app-geom:oriented-branches}
 \tau_\varsigma=-\varsigma h,
 \qquad
 \widehat H_\varsigma=-\varsigma H,
 \qquad
 \vartheta_\varsigma=
 \begin{cases}
  \theta,&\varsigma=-1,\\
  \theta^\dagger,&\varsigma=1.
 \end{cases}
\end{equation}
Put
\(W=(y-\phi_J(x))/r^2-\Gamma_J^0(x,L)\).  We next justify the division
used on the whole bridge, including \(h=0\) and \(r=0\).  Apply
\eqref{app-geom:divided-stage-system} after the fold scaling
\(x=rX\).  The product stage Jacobian remains invertible because its
small quantity is
\begin{equation}
 \label{app-geom:bridge-stage-smallness}
 H\langle X\rangle=h(r+\abs x)\le h_0(r_0+\delta).
\end{equation}
Hence the implicit stage branch is analytic on a common weighted tube, with
bounds controlled by \eqref{app-geom:absolute-tableau-bounds}.

Let \(\mathcal N_r\) be the normal output after subtracting the frozen
linear stability term below, and let \(\mathcal N_0\) be its value for the
canonical inner datum \(r=0\).  Two different divisions are needed.  The
difference \(\mathcal N_r-\mathcal N_0\) vanishes both at \(r=0\) and at
\(H=0\), so the exact double Hadamard formula gives a factor \(Hr\).
For \(\mathcal N_0\), consistency makes the value and first step
derivative vanish after the stability term has been removed; hence it has
a factor \(H^2\).  Thus
\begin{equation}
 \label{app-geom:bridge-numerator-division}
 \mathcal N
 =Hr\,\mathcal F_\varsigma+H^2\mathcal G_\varsigma.
\end{equation}
More explicitly, these quotients are defined by integrating
\(\partial_r\partial_H(\mathcal N_r-\mathcal N_0)\) on the parameter
rectangle and \(\partial_H^2\mathcal N_0\) on the step interval.  They are
therefore quotients of the stage branch, with one fixed polynomial
majorant for the state and \(L\)-derivatives used below.

The factor \(H^2\mathcal G_\varsigma\) must be retained for a general
order-two tableau.  It is not, in general, divisible by \(Hr\): already
the canonical vector field away from \(W=0\) produces an order-\(H^2\)
term.  On the canonical graph \(W=0\), the second order conditions give
one additional zero, and integral division in \(H\) and then in \(X\)
gives
\begin{equation}
 \label{app-geom:canonical-graph-quotient}
 \mathcal G_\varsigma(X,W,H,\theta)
 =\overline{\mathcal G}_\varsigma(X,W,H,\theta)W
  +H\{\mathcal C_{0,\varsigma}+X\mathcal C_{1,\varsigma}\}.
\end{equation}
Here the first term is the integral difference quotient between \(W\)
and zero.  The second term is the familiar
\(O(H^3\langle X\rangle)\) unshifted canonical graph residual.  The
full \(H^2\mathcal G_\varsigma\) term is nevertheless regular: the stage
functions depend on \(X\) through the bounded combinations controlled by
\eqref{app-geom:bridge-stage-smallness}, and the prefactor is
\(H^2=h^2r^2\).

As a direct check, take the canonical field and explicit midpoint, for
which \(R(\zeta)=1+\zeta+\zeta^2/2\).  Exact elimination of its two
stages gives
\begin{equation}
 \label{app-geom:midpoint-bridge-check}
 \begin{aligned}
 W^+-R(2HX)W
 ={}&H^2\left(\frac12W-W^2\right)\\
 &+H^3\left(-\frac18X+\frac32XW-\frac52XW^2\right)
   +O(H^4P(X,W)).
 \end{aligned}
\end{equation}
At \(W=0\) the same elimination gives the exact polynomial identity
\begin{equation}
 \label{app-geom:midpoint-canonical-graph-exact}
 W^+=-\frac18H^3X-\frac1{16}H^4-\frac1{256}H^6.
\end{equation}
This both confirms the \(H^3X\) canonical-graph residual at \(W=0\)
and shows why the preceding \(H^2W\)-dependent term cannot be hidden in
an \(Hr\) quotient for a general method.

Substitution gives the exact factorization
\begin{equation}
 \label{app-geom:bridge-factorization}
 \begin{aligned}
  X^+&=X+\widehat H_\varsigma
       \mathcal P_\varsigma(X,W,L,r,h,\theta),\\
  W^+&=R_{\vartheta_\varsigma}
       (\tau_\varsigma x\kappa_J(x))W
       +H^2\mathcal G_\varsigma
       +\widehat H_\varsigma r\mathcal F_\varsigma,
 \end{aligned}
\end{equation}
where
\[
 R_\vartheta(\zeta)
 =1+\zeta b_\vartheta^T
       (I-\zeta\Amat_\vartheta)^{-1}\one.
\]
The base quotient is defined by one integral division in the signed step
and satisfies, uniformly in the bridge tube,
\begin{equation}
 \label{app-geom:bridge-base-quotient}
 \mathcal P_\varsigma
 =Q_J(x)+c_J(x)W
  +\ord\!\left(r^2+h(r+\abs x)\right).
\end{equation}
Equations \eqref{app-geom:bridge-numerator-division}--
\eqref{app-geom:bridge-base-quotient} are the closed-rectangle
factorization.  No division by an unpaired \(h\) or \(r\) occurs.  The
canonical term is kept with its natural factor \(H^2\), while the
fold-dependent difference has the exact factor \(Hr\).  Since
\(H=hr\le h_0r\), both terms have uniformly small accumulated response;
no polynomial lower bound on \(h\) is used.

Consistency gives \(R_\theta(-u)=1-u+O(u^2)\) uniformly over
\(\ThetaRK\).  On the repelling side, exact inversion gives
\begin{equation}
 \label{app-geom:adjoint-stability}
 R_{\theta^\dagger}(-u)=\frac1{R_\theta(u)}=1-u+O(u^2).
\end{equation}
Thus only a sufficiently small stability interval, not global
\(A\)-stability, is used.  Define the effective normal multiplier
\begin{equation}
 \label{app-geom:effective-bridge-multiplier}
 \widehat R_\varsigma
 =R_{\vartheta_\varsigma}(-Hz\kappa_J(\varsigma rz))
   +H^2\overline{\mathcal G}_\varsigma.
\end{equation}
Since \(H^2\le h_0Hr\), one further reduction gives
\begin{equation}
 \label{app-geom:effective-bridge-damping}
 0<\widehat R_\varsigma
 \le\exp\{-cHz+CHr\}.
\end{equation}
From \eqref{app-geom:coercivity} and
\eqref{app-geom:bridge-base-quotient}, first choose the fixed
\(W\)-corridor and then reduce \(\delta,r_0,h_0\) so that
\(0<P_-\le\mathcal P_\varsigma\le P_+\).  Together with
\eqref{app-geom:bridge-factorization}, this gives
\begin{equation}
 \label{app-geom:bridge-damping}
 0<R_{\vartheta_\varsigma}(-Hz\kappa_J(\varsigma rz))
   \le\eexp^{-cHz},
 \qquad
 cH\le z-z^+\le CH.
\end{equation}
In the original normalized variables the same estimate is
\(c_dhr^2\le x^+-x\le C_dhr^2\) for the positive branch.  This is direct
on the attracting side and follows by exact inversion of the coreward
repelling branch on the other side.

To specify the continuation, choose
\(t_0=3\delta/(4r)\) inside the collar--bridge overlap.  If \(w_0\) is the
scaled collar graph and \(t_1=T_{w_0}(t_0)\), the restriction of \(w_0\) to
\([t_1,t_0]\) is a fundamental arc.  Repeatedly apply the same coreward
branch and retain the component over \(z\ge0\).  The drift in
\eqref{app-geom:bridge-damping} produces a finite index \(N\) for which the
next full arc crosses zero; on that final image retain only the component
ending at \(z=0\).  Projection monotonicity gives a graph at every step.  If
\(W=w(z,L)\), each regraphing has the form
\begin{equation}
 \label{app-geom:exact-regraphing}
 \begin{aligned}
  T_w(z,L)&=z-H\mathcal P_\varsigma(z,w(z,L),L),\\
  V_w(z,L)&=\widehat R_\varsigma(z,w(z,L),L)w(z,L)
       +H^3\{\mathcal C_{0,\varsigma}
                   +\varsigma z\mathcal C_{1,\varsigma}\}
       +Hr\mathcal F_\varsigma(z,w(z,L),L),\\
  \bar w(T_w(z,L),L)&=V_w(z,L).
 \end{aligned}
\end{equation}
For a class-uniform constant \(\eta_0>0\), chosen small with the bridge
corridor, the first-exit estimate starts from
\begin{equation}
 \label{app-geom:bridge-first-exit}
 \abs{w_{k+1}}
 \le\eexp(-cHz_k+CHr)\abs{w_k}
       +C\{Hr+H^3\langle z_k\rangle\}.
\end{equation}
The Gaussian product and the mesh relation in
\eqref{app-geom:bridge-damping} imply
\[
 \sum_{k<n}\Pi_{k,n}
   \{Hr+H^3\langle z_k\rangle\}
 \le C(r+H^2)\le Cr.
\]
Thus the canonical graph residual has accumulated size \(O(H^2)\), while
the fold-dependent term has accumulated size \(O(r)\).  Neither can trigger
an exit before the core.  Differentiation of the first line then gives
\begin{equation}
 \label{app-geom:projection-control}
 \abs{w}\le \eta_0,
 \qquad \abs{w_z}\le Cr,
 \qquad \frac12\le\partial_zT_w\le\frac32.
\end{equation}
For \(0\le j\le2\) and \(0\le i\le7-j\), triangular differentiation of
the last line yields
\begin{equation}
 \label{app-geom:bridge-jet-transport}
 \abs{\partial_z^i\partial_L^j\bar w\circ T_w}
 \le\eexp\{-cHz+CHr\}
       \abs{\partial_z^i\partial_L^jw}+C_{ij}H,
\end{equation}
after lower jets have been estimated.  Products of the multipliers are
Gaussian and the corresponding weighted sums are bounded uniformly in the
number of arcs.  In this differentiation every nonprincipal term contains
a derivative of \(T_w-z\), of the effective multiplier, of the
\(H^3\mathcal C\)-term, or of \(Hr\mathcal F_\varsigma\); the
weighted stage bounds make their combined contribution \(C_{ij}H\).
Thus the retained canonical term does not alter
\eqref{app-geom:bridge-jet-transport}.  Conversion back to \(x\) gives the bridge part of
\eqref{eq:main-selection-bounds}.

Exact collar invariance gives equality of the first two graph germs on an
open overlap.  Injectivity transports this equality across every seam.
On the final partial arc, every retained target has a unique predecessor,
so both identities in \eqref{app-geom:map-inverse-identities} hold there as
well.  The pasted bridge is consequently invariant under the positive-step
Runge--Kutta map and its contained inverse, including at every
retained endpoint.

For the flow, variation of constants in
\eqref{app-geom:scalar-fold-equation}, using
\eqref{app-geom:gaussian-propagator}, produces a bridge with the same mixed
derivative bounds.  The bridge and collar solve the same graph
equation and agree on their initialization overlap.  Pasting the two
dynamics on both sides, taking \(B_{\rm sel}\) larger than the resulting
uniform bounds, and reducing \(r_0,h_0\) completes the construction.
\end{proof}

\subsection{A slow-tube estimate independent of weight signs}

The next estimate is used for arbitrary admissible selections, not only the
references constructed above.

\begin{lemma}[Uniform slow tube]
 \label{app-geom:slow-tube}
Every admissible flow or numerical-map graph satisfies
\begin{equation}
 \label{app-geom:common-slow-tube}
 \abs{\widehat m_{\bullet,q}(x;L)-\phi_J(x)}\le Cr^2
\end{equation}
on its full side interval.  The constant is uniform for every compact
finite-stage order-two family, without a sign condition on the Runge--Kutta
weights.
\end{lemma}

\begin{proof}
For the flow, \eqref{app-geom:flow-graph-equation} directly gives
\(\abs{f_J(x,\widehat m_{\fl,q};r^2,r^2L)}\le Cr^2\).  Uniform
monotonicity in \(y\), together with
\(f_J(x,\phi_J(x);r^2,r^2L)=O(r^2)\), proves
\eqref{app-geom:common-slow-tube}.

For a map input \(z=(x,\widehat m_{\rk,q}(x;L))\), let \(Z_i\) be its
stages and put \(M=\max_i\abs{f_J(Z_i)}\).  Compactness and the uniform
stage bound give
\[
 B_A=\sup_{\theta,i}\sum_j\abs{a_{ij}(\theta)}<\infty,
 \qquad
 B_b=\sup_\theta\sum_i\abs{b_i(\theta)}<\infty.
\]
The stage equations and the slow component imply
\begin{equation}
 \label{app-geom:stage-fast-bound}
 M\le\abs{f_J(z)}+Ch(M+r^2).
\end{equation}
Because \(b_\theta^T\one=1\), no positivity is needed to estimate
\begin{equation}
 \label{app-geom:weighted-stage-bound}
 \left|f_J(z)-\sum_i b_i(\theta)f_J(Z_i)\right|
 \le CB_bh(M+r^2).
\end{equation}
The output drift \eqref{app-geom:map-drift} says
\(\sum_i b_i f_J(Z_i)=O(r^2)\).  Reducing \(h_0\) absorbs the \(hM\)
terms in \eqref{app-geom:stage-fast-bound}--
\eqref{app-geom:weighted-stage-bound}, yielding
\(M+\abs{f_J(z)}\le Cr^2\).  Uniform monotonicity in \(y\) then proves
\eqref{app-geom:common-slow-tube} for the map.
\end{proof}

\subsection{Gaussian shielding of the outer selection}

\begin{lemma}[Exponential closeness at the central section]
 \label{app-geom:shielding}
For two admissible selections of the same flow, or of the same map with the
same \(\theta\), and either side \(q\),
\begin{equation}
 \label{app-geom:shielding-estimate}
 \max_{0\le j\le2}\sup_{L\in\Lambda}
 \left|
 \partial_L^j
 \left(y_{\bullet,q}^{\sigma_\bullet}(r^2L)
      -y_{\bullet,q}^{\widetilde\sigma_\bullet}(r^2L)
 \right)
 \right|
 \le Cr^{-M}\eexp^{-c/r^2}
\end{equation}
for fixed \(C,c>0\) and an integer \(M\).
\end{lemma}

\begin{proof}
By \cref{app-geom:slow-tube}, both graphs lie in the common divided tube.
For two flow bridges, subtract their scalar graph equations before taking
absolute values.  In the signed coordinate, their difference \(d\)
satisfies a homogeneous secant equation
\begin{equation}
 \label{app-geom:flow-secant}
 \partial_zd=(z\mathcal H+r\mathcal R)d,
 \qquad \mathcal H\ge c_H>0.
\end{equation}
The coreward propagator is bounded by
\eqref{app-geom:gaussian-propagator}.  One and two \(L\)-derivatives form a
triangular system with the same principal propagator.  The terminal bounds
from \eqref{eq:main-selection-bounds} therefore give a fixed algebraic
power of \(r^{-1}\) times \(\exp(-c/r^2)\) at \(z=0\).

For two map graphs, sources with the same abscissa generally have different
targets.  Projection monotonicity in
\eqref{app-geom:projection-control} supplies a unique adjusted source on
the second graph with the same target.  Secant subtraction at this target
gives
\begin{equation}
 \label{app-geom:common-target-recurrence}
 d_k=M_kd_{k+1},
 \qquad
 \log M_k\le-cH z_{k+1}+CH,
 \qquad
 \frac12\le M_k\le2.
\end{equation}
Summing the logarithmic bound along the mesh gives a negative quadratic in
the travelled \(z\)-distance plus a linear term; completing the square and
reducing the outer corridor if necessary therefore gives
\begin{equation}
 \label{app-geom:gaussian-product}
 0<\prod_{\ell<k}M_\ell\le C\eexp^{-cz_k^2}.
\end{equation}
Differentiation at fixed Eulerian target gives, for \(j=1,2\),
\begin{equation}
 \label{app-geom:L-derivative-recurrence}
 d_{k,j}=M_kd_{k+1,j}
   +\sum_{m<j}C_{k,jm}d_{k+1,m},
 \qquad
 \abs{C_{k,jm}}\le Hr^{-N_{jm}}P(z_{k+1}).
\end{equation}
Iteration with \eqref{app-geom:gaussian-product} and the Gaussian weighted
sum bound absorbs every lower-jet term into a fixed power of \(r^{-1}\).
The seam and terminal identities from the construction give the same
recurrence on the final partial cell.  Multiplying by the scale factor
\(r^2\), and enlarging one common exponent \(M\), proves
\eqref{app-geom:shielding-estimate} for both dynamics.
\end{proof}

\subsection{The transverse flow profile}

\begin{lemma}[Gaussian flow crossing]
 \label{app-geom:flow-profile}
Every admissible flow selection satisfies
\eqref{eq:flow-splitting-profile} in \(C_L^2(\Lambda)\), has one simple
threshold root in a fixed neighbourhood of \(L_0(J)\), and satisfies
\eqref{eq:flow-physical-slope} at that root.
\end{lemma}

\begin{proof}
Use first the reference flow selection from
\cref{app-geom:construction}.  The normalized finite jet gives
\begin{equation}
 \label{app-geom:critical-jet-expansions}
 \begin{aligned}
  \phi_J(x)&=x^2+(A+B)x^3+O(x^4),
  &q_J(x)&=x+(E+F_s)x^2+O(x^3),\\
  a_J(x)&=C_\eps x+O(x^2),
  &b_J(x)&=Dx+O(x^2),\\
  c_J(x)&=-1+Bx+O(x^2).
 \end{aligned}
\end{equation}
For \(x=\varsigma rz\), let
\(E_\varsigma(z,L)=e(\varsigma rz,L)\) and
\(U_\varsigma=E_\varsigma/r\).  Equation
\eqref{app-geom:scalar-fold-equation} becomes
\begin{equation}
 \label{app-geom:scaled-flow-equation}
 \partial_zU_\varsigma
 =z\mathcal H_{\varsigma,r}(z,L)U_\varsigma
  +\varsigma\mathcal G_{\varsigma,r}(z,L),
\end{equation}
where, through two \(L\)-derivatives,
\begin{equation}
 \label{app-geom:flow-coefficient-limit}
 \mathcal H_{\varsigma,r}=4+O(r(1+z)),
 \qquad
 \mathcal G_{\varsigma,r}
 =(D+2)(L_0(J)-L)+O(r(1+z)).
\end{equation}
Substitution of \eqref{app-geom:critical-jet-expansions} into
\eqref{app-geom:reference-graph} gives both the constant term in
\eqref{app-geom:flow-coefficient-limit} and the formula
\eqref{eq:main-L0-definition}.

The coreward kernel of \eqref{app-geom:scaled-flow-equation} satisfies
\begin{equation}
 \label{app-geom:flow-kernel-limit}
 \mathcal K_{\varsigma,r}(t,L)
 =\eexp^{-2t^2}+rP(t)\eexp^{-ct^2}
\end{equation}
in \(C_L^2\).  Finite variation of constants from \(\delta/r\) to zero
therefore gives
\begin{equation}
 \label{app-geom:one-side-flow-value}
 U_\varsigma(0,L)
 =\varsigma(D+2)(L-L_0(J))
   \int_0^\infty\eexp^{-2t^2}\dd t
  +O_{C_L^2}(r)
  +O_{C_L^2}(r^{-M}\eexp^{-c/r^2}).
\end{equation}
At \(x=0\), the common terms
\(\phi_J(0)+r^2\Gamma_J^0(0,L)\) cancel between the two sides, and
\[
 \Delta_\fl^{\rm ref}(r^2L)
 =r^3\{U_1(0,L)-U_{-1}(0,L)\}.
\]
Since
\(2\int_0^\infty\eexp^{-2t^2}\dd t=\sqrt{\pi/2}\), this is
\eqref{eq:flow-splitting-profile} for the reference selection.  The
shielding estimate \eqref{app-geom:shielding-estimate} transfers the
profile, with its first two \(L\)-derivatives, to every flow selection.

The compact lower bound on
\(a(J)=(D+2)\sqrt{\pi/2}\) makes the splitting strictly increasing on a
fixed neighbourhood of \(L_0(J)\) for small \(r\).  Its endpoint signs are
opposite, so it has exactly one simple zero there.  Finally,
\[
 \partial_\lambda\Delta_\fl
 =r^{-2}\partial_L\Delta_\fl
 =a(J)r+O(r^2)+O(r^{-M}\eexp^{-c/r^2}),
\]
which proves the \(\lambda\)-slope assertion.
\end{proof}

The construction lemma, shielding lemma, and flow-profile lemma prove
\cref{prop:fold-geometry}.

\section{Exact Runge--Kutta residuals on the fold chain}
\label{app:exact-residual}

This appendix proves \cref{prop:exact-residual}.  We first identify the
order-three step coefficient at the canonical fold and then obtain the
residual remainder by Hadamard division of the stage equations.  The
coefficient calculation therefore does not replace the numerical map by a
modified flow.

Compactness and the uniform stage bound in \cref{def:rk-family} imply
\begin{equation}
 \label{app-res:tableau-bound}
 B_\Theta:=\sup_{\theta\in\ThetaRK}
 \left\{\max_i\sum_j\abs{a_{ij}(\theta)}
 +\sum_i\abs{b_i(\theta)}+\norm{\cvec_\theta}_\infty\right\}<\infty .
\end{equation}
After a fixed enlargement, the same bound holds on the adjoint image of
the family.

\subsection{Stage branches and the adjoint inverse}

For an inner vector field \(\mathcal V\) and a tableau \(\vartheta\), let
\(\Psi_\eta^{\mathcal V,\vartheta}\) be defined by the stage system
\begin{equation}
 \label{app-res:actual-stage-system}
 Z_i=Z+\eta\sum_j a_{ij}^{\vartheta}\mathcal V(Z_j),
 \qquad
 \Psi_\eta^{\mathcal V,\vartheta}(Z)
 =Z+\eta\sum_i b_i^{\vartheta}\mathcal V(Z_i).
\end{equation}
For an implicit method this always denotes the analytic branch continuing
from \(Z_1=\cdots=Z_s=Z\) at \(\eta=0\).

\begin{lemma}[Uniform branches and tableau identities]
 \label{app-res:adjoint-lemma}
After reducing \(r_0,h_0\), the branch in
\eqref{app-res:actual-stage-system} exists uniquely on every fold-chain
neighbourhood used below, for
\(\vartheta\in\{\theta,\theta^\dagger\}\) and \(\abs\eta\le H\).
Its stages, output and the derivatives used below have uniform polynomial
bounds in \(1+\abs X\).  On contained input--output pairs,
\begin{equation}
 \label{app-res:actual-adjoint-inverse}
 \bigl(\Psi_\eta^{\mathcal V,\theta}\bigr)^{-1}
 =\Psi_{-\eta}^{\mathcal V,\theta^\dagger}.
\end{equation}
Moreover,
\begin{equation}
 \label{app-res:adjoint-algebra}
 \alpha_{\theta^\dagger}=\alpha_\theta,
 \quad \beta_{\theta^\dagger}=\beta_\theta,
 \quad \delta_{\theta^\dagger}=\delta_\theta,
 \quad q_{\theta^\dagger}=q_\theta,
 \quad d_{\theta^\dagger}=d_\theta,
 \quad \gamma_{\theta^\dagger}=-\gamma_\theta,
 \quad b_\theta^Td_\theta=0.
\end{equation}
\end{lemma}

\begin{proof}
Put \(\langle X\rangle=1+\abs X\).  Around the canonical parabola use a
tube with horizontal and vertical radii proportional to
\(\langle X\rangle\) and \(\langle X\rangle^2\), and measure increments by
\[
 \norm{(u,v)}_X=\frac{\abs u}{\langle X\rangle}
                +\frac{\abs v}{\langle X\rangle^2}.
\]
Weighted analyticity gives
\(\norm{\mathcal V}_X+\norm{D\mathcal V}_{X\to X}
\le C\langle X\rangle\) on a slightly larger tube.  Along the chain,
\begin{equation}
 \label{app-res:stage-small-quantity}
 \abs\eta\langle X\rangle
 \le H(1+\delta/r)=h(r+\delta).
\end{equation}
The right-hand side is uniformly small.  The stage operator is therefore a
contraction on a product tube, and its Jacobian is a uniformly invertible
perturbation of the identity.  Repeated implicit differentiation proves
the stated bounds.  This uses only \eqref{app-res:tableau-bound}; it uses
neither weight signs nor triangularity of \(\Amat\).

If stages \(Z_i\) take \(Z_-\) to
\(Z_+=Z_-+\eta\sum_i b_i\mathcal V(Z_i)\), then
\[
 Z_i=Z_+-\eta\sum_j(b_j-a_{ij})\mathcal V(Z_j).
\]
This is the negative-step adjoint stage system, so branch uniqueness proves
\eqref{app-res:actual-adjoint-inverse}.

The order-two identities give
\[
 b^T(\one-\cvec)=\frac12,
 \qquad \Amat(\one-\cvec)=\cvec-\Amat\cvec.
\]
They imply invariance of \(\alpha\) and \(\beta\), and hence of \(\delta\)
and \(q\).  Direct substitution gives
\[
 (\cvec^\dagger)^{\circ2}-2\Amat^\dagger\cvec^\dagger-2\delta\one
 =\cvec^{\circ2}-2\Amat\cvec-2\delta\one=d.
\]
Also
\[
 b^Td=b^T\cvec^{\circ2}-2b^T\Amat\cvec-2(\alpha-\beta)=0,
\]
and therefore
\(\gamma_{\theta^\dagger}
=-b^T(\one b^T-\Amat)d=b^T\Amat d=-\gamma_\theta\).
\end{proof}

\subsection{Residual reduction}

We state the reduction before proving its two local inputs: the canonical
stage calculation and the first weighted datum jet.

For a smooth vector field \(\mathcal V\), differentiation of the
stage system three times at zero step gives the map-minus-flow coefficient
\begin{equation}
 \label{app-res:third-defect-operator}
 \mathcal M_3(\mathcal V)
 =\alpha\,\mathcal V''(\mathcal V,\mathcal V)
  +\beta\,\mathcal V'\mathcal V'\mathcal V.
\end{equation}
This finite differentiation identifies the coefficient; division of the
map residual follows in the next subsection.

\begin{lemma}[First fold-dependent coefficient]
 \label{app-res:Q10-coefficient}
Let \(\ell_r=(-Y'_{r,\varsigma,L},1)\).  The coefficient of
\(r\eta^3\) in output minus shifted graph is
\begin{equation}
 \label{app-res:Q10-contraction}
 Q_{10}(X)
 =\left.\partial_r\right|_{r=0}
 \left[
  \ell_r\left\{
   \mathcal M_3(\mathcal V_0+r\mathcal V_1)
   +\frac{\alpha-\beta}{2}
       \partial_Y(\mathcal V_0+r\mathcal V_1)
  \right\}
 \right]_{Y=Y_r(X)}.
\end{equation}
It equals the polynomial in \eqref{eq:Q10}.
\end{lemma}

\begin{proof}
The first term in braces is \eqref{app-res:third-defect-operator}.  Moving
the input graph by \(q\eta^2\) changes the third step derivative by
\(q\partial_Y\mathcal V\), with \(q=(\alpha-\beta)/2\).  Contraction with
\(\ell_r\) converts the state defect into output minus graph.  Its
zeroth-order value vanishes by the choice of \(q\), consistently with
\cref{app-res:canonical-residual}.

Here is the contraction before the compatibility relations are used.  Write
the derivative in \eqref{app-res:Q10-contraction} as
\(\alpha C_\alpha(X)+\beta C_\beta(X)\).  Direct differentiation of the
two components of \(\mathcal V_0+r\mathcal V_1\), followed by contraction
with \((-Y_r',1)\), gives
\begin{align}
 C_\alpha(X)={}&
 4(F_3-A-B)X^4+\{3(F_3-A-B)-2G_2\}X^2
 +\frac{E+F_s}{2},
 \label{app-res:Q10-alpha-before-compatibility}\\
 C_\beta(X)={}&
 8(F_3-A-B)X^6
 +\{-2A-2B+4E+2F_3+4F_s-4G_2\}X^4
 \notag\\
 &+\left\{
 \frac32A+\frac52B-2C_\eps-2DL_0+E-3F_3-F_s
 +2G_2-4L_0+4S_\eps
 \right\}X^2
 \notag\\
 &-\frac B4+\frac{C_\eps}{2}+\frac{DL_0}{2}
   +\frac{F_s}{2}+L_0-S_\eps.
 \label{app-res:Q10-beta-before-compatibility}
\end{align}
The identities \(F_3=A+B\) and \(G_2=E+F_s\) remove the fourth- and
sixth-degree terms.  Solving \eqref{eq:main-L0-definition} for
\(S_\eps\) in the remaining constant and quadratic terms gives
\begin{equation}
 C_\alpha(X)=\frac{G_2}{2}-2G_2X^2,
 \qquad
 C_\beta(X)=-\frac{3F_3}{8}+\frac{G_2}{4}+2G_2X^2.
 \label{app-res:Q10-two-contractions}
\end{equation}
This also displays explicitly where all terms containing
\(C_\eps,D,L_0,S_\eps\), or \(\vartheta_0\) cancel.  Therefore
\[
 Q_{10}(X)
 =-\frac38\beta F_3
  +\left(\frac12\alpha+\frac14\beta\right)G_2
  +2(\beta-\alpha)G_2X^2,
\]
which is \eqref{eq:Q10}.

\end{proof}

\subsubsection{Uniform division}

For a local datum \(\mathcal D=(\mathcal V,Y)\), define the residual
directly from its stages by
\begin{equation}
 \label{app-res:general-residual-definition}
 \begin{aligned}
 \mathfrak E_\vartheta(\eta;\mathcal D)(X)
 ={}&\pi_Y\Psi_\eta^{\mathcal V,\vartheta}
       (X,Y(X)+q\eta^2)\\
 &-Y\!\left(\pi_X\Psi_\eta^{\mathcal V,\vartheta}
       (X,Y(X)+q\eta^2)\right)-q\eta^2.
 \end{aligned}
\end{equation}

\begin{lemma}[Exact divided residual]
 \label{app-res:exact-division}
For either signed side and
\(\vartheta\in\{\theta,\theta^\dagger\}\), every enhanced admissible flow
datum in \cref{app-res:datum-jet} satisfies
\begin{equation}
 \label{app-res:exact-residual-expanded}
 \begin{aligned}
 \mathfrak E_\vartheta(\eta;\mathcal D_{r,\varsigma,L})(X)
 =\eta^3\bigg[&rQ_{10}(X)
 +\eta\gamma_\vartheta\left(X^2-\frac14\right)\\
 &+\ord\!\left(r^2P(X)+r\abs\eta P(X)+\eta^2P(X)\right)\bigg]
 \end{aligned}
\end{equation}
uniformly for \(\abs\eta\le H\).  The quotient at \(\eta=0\) is defined
continuously.  It and its first \(L\)-derivative are uniformly defined;
differentiating the displayed value estimate costs at most the one power
of \(r\) described in \cref{app-res:datum-jet}.
\end{lemma}

\begin{proof}
Use the weighted tube from \cref{app-res:adjoint-lemma} at each \(X\).
By \cref{app-res:datum-jet}, the segment
\[
 \mathcal D_t=\mathcal D_0
  +t(\mathcal D_{r,\varsigma,L}-\mathcal D_0),
 \qquad 0\le t\le1,
\]
lies in its inner part.  The shifted input, stages and output stay strictly
inside the larger tube by \eqref{app-res:stage-small-quantity}.  In
unscaled normalized variables, all these tubes lie in one fixed set
\(\Omega\Subset\Omega^+\).  Thus the stage margins remain uniform as
\(\abs X\) grows to \(O(r^{-1})\).

Locally regard \(\mathcal D\) as an element of
\(C^5(\mathcal T_X;\R^2)\times C^5(I_X)\), with its weighted norm.  The
implicit-function theorem for the stages and the chain rule for the final
graph evaluation show that \(\mathfrak E\) is five times differentiable in
\(\eta\) and twice in the datum.  Finite differentiation of the stage
system gives
\begin{equation}
 \label{app-res:ambient-derivative-bounds}
 \begin{aligned}
  \norm{D_{\mathcal D}\partial_\eta^4\mathfrak E}&\le P(X),\\
  \norm{D_{\mathcal D}\partial_\eta^3\mathfrak E}
  +\norm{D_{\mathcal D}^2\partial_\eta^3\mathfrak E}&\le P(X).
 \end{aligned}
\end{equation}
Only derivatives of \(\mathcal V\) and \(Y\) through order five occur.

Both endpoint data \(\mathcal D_0\) and
\(\mathcal D_{r,\varsigma,L}\) pair a vector field with one of its
flow-invariant graphs.  At either endpoint, consistency and the two
classical order conditions give
\begin{equation}
 \label{app-res:triple-zero}
 \partial_\eta^j\mathfrak E_\vartheta(0;\mathcal D)(X)=0,
 \qquad j=0,1,2.
\end{equation}
The shift is \(O(\eta^2)\) and first contributes to the third derivative.
Taylor's formula therefore yields the exact Hadamard division
\begin{equation}
 \label{app-res:hadamard-division}
 \frac{\mathfrak E_\vartheta(\eta;\mathcal D)(X)}{\eta^3}
 =\frac12\int_0^1(1-s)^2
   \partial_\eta^3\mathfrak E_\vartheta(s\eta;\mathcal D)(X)\dd s.
\end{equation}
for \(\mathcal D=\mathcal D_0\) and
\(\mathcal D=\mathcal D_{r,\varsigma,L}\).  The interpolating data
\(\mathcal D_t\) are used only to estimate the difference of the two
endpoint integrands; no flow compatibility is asserted for
\(0<t<1\).

Define the centred quotient
\[
 \widetilde{\mathcal Q}(r,\eta,X)
 =\frac{\mathfrak E_\vartheta
       (\eta;\mathcal D_{r,\varsigma,L})(X)
       -\mathfrak E_\vartheta(\eta;\mathcal D_0)(X)}{\eta^3}.
\]
Integrating \eqref{app-res:ambient-derivative-bounds} first in the datum
and then in the step gives
\begin{equation}
 \label{app-res:centered-step-bound}
 \widetilde{\mathcal Q}(r,\eta,X)
 -\widetilde{\mathcal Q}(r,0,X)
 =\ord(r\abs\eta P(X)).
\end{equation}
Taylor expansion in the datum, still at the exactly divided value, gives
\begin{equation}
 \label{app-res:centered-data-bound}
 \widetilde{\mathcal Q}(r,0,X)
 =\frac r{3!}D_{\mathcal D}\partial_\eta^3\mathfrak E_\vartheta
       (0;\mathcal D_0)(X)[\mathcal D_1]+\ord(r^2P(X)).
\end{equation}
The displayed coefficient is \(Q_{10}\) by
\cref{app-res:Q10-coefficient}.  Finally,
\cref{app-res:canonical-residual} gives
\begin{equation}
 \label{app-res:canonical-divided}
 \frac{\mathfrak E_\vartheta(\eta;\mathcal D_0)(X)}{\eta^3}
 =\eta\gamma_\vartheta\left(X^2-\frac14\right)
  +\ord(\eta^2P(X)).
\end{equation}
Combining the last three displays proves
\eqref{app-res:exact-residual-expanded}.  Differentiation in \(L\) uses
the same integral identities.
\end{proof}

For the transport orientation, set
\begin{equation}
 \label{app-res:oriented-method-step}
 \vartheta_\varsigma=
 \begin{cases}\theta,&\varsigma=-1,\\
 \theta^\dagger,&\varsigma=1,\end{cases}
 \qquad \widehat H_\varsigma=-\varsigma H.
\end{equation}
Then \eqref{app-res:adjoint-algebra} gives
\begin{equation}
 \label{app-res:oriented-residual}
 \begin{aligned}
 \mathfrak E_{\vartheta_\varsigma}
 (\widehat H_\varsigma;\mathcal D_{r,\varsigma,L})(X)
 =-\varsigma H^3\bigl\{&rQ_{10}(X)
 +H\gamma_\theta\mathscr H_2(X)\\
 &+\ord(r^2P+rHP+H^2P)\bigr\}.
 \end{aligned}
\end{equation}
Thus the attracting forward branch and repelling adjoint inverse have the
same oriented forcing, as required in \cref{eq:one-side-response}.

An explicit \(\theta\) generally has an implicit adjoint.  This is no
restriction: the adjoint only represents the contained inverse of the
already defined positive-step map on the repelling side.

\subsection{The canonical shifted parabola}

Consider
\begin{equation}
 \label{app-res:canonical-datum}
 \mathcal V_0(X,Y)=(X^2-Y,X),
 \qquad Y_0(X)=X^2-\frac12.
\end{equation}
Abbreviate
\[
 \delta=\alpha-\beta,
 \quad q=\frac\delta2,
 \quad d=\cvec^{\circ2}-2\Amat\cvec-2\delta\one,
 \quad u=\frac d4.
\]

\begin{lemma}[Canonical residual from the stage equations]
 \label{app-res:canonical-residual}
Start \eqref{app-res:actual-stage-system} at
\((X,Y_0(X)+q\eta^2)\), and define
\begin{equation}
 \label{app-res:canonical-residual-definition}
 \begin{aligned}
 \mathfrak E_\vartheta(\eta;\mathcal D_0)(X)
 ={}&\pi_Y\Psi_\eta^{\mathcal V_0,\vartheta}
       (X,Y_0(X)+q\eta^2)\\
 &-Y_0\!\left(\pi_X\Psi_\eta^{\mathcal V_0,\vartheta}
       (X,Y_0(X)+q\eta^2)\right)-q\eta^2.
 \end{aligned}
\end{equation}
Then, for \(\vartheta\in\{\theta,\theta^\dagger\}\),
\begin{equation}
 \label{app-res:canonical-hermite-exact}
 \mathfrak E_\vartheta(\eta;\mathcal D_0)(X)
 =\eta^4\gamma_\vartheta\left(X^2-\frac14\right)
  +\ord\!\left(\abs\eta^5P(X)\right).
\end{equation}
\end{lemma}

\begin{proof}
For \(\vartheta=\theta\), expansion of the fixed analytic stage branch gives
\begin{align}
 X_i&=X+\frac\eta2c_i+\eta^3(\Amat u)_i
          +\ord(\eta^4P(X)),
 \label{app-res:canonical-X-stages}\\
 X_i^2-Y_i&=\frac12+\eta^2u_i
        +2X\eta^3(\Amat u)_i+\ord(\eta^4P(X)).
 \label{app-res:canonical-fast-stages}
\end{align}
Indeed, the second-order term in the second display is
\[
 \frac14\{c_i^2-2(\Amat\cvec)_i-4q\}=\frac14d_i=u_i.
\]
Since \(b^Tu=0\), put \(V=b^T\Amat u=-\gamma/4\).  The exact output then
satisfies
\begin{align*}
 X_+&=X+\frac\eta2+2X\eta^4V+\ord(\eta^5P(X)),\\
 Y_+&=X^2-\frac12+q\eta^2+\eta X+\frac{\eta^2}{4}
          +\eta^4V+\ord(\eta^5P(X)).
\end{align*}
Subtracting \(X_+^2-1/2+q\eta^2\) gives
\(\eta^4V(1-4X^2)=\eta^4\gamma(X^2-1/4)\).
The adjoint case is identical, with
\(\gamma_{\theta^\dagger}=-\gamma_\theta\).  Uniformity of the remainder
follows from \cref{app-res:adjoint-lemma}.
\end{proof}

Thus failure to preserve the shifted parabola produces a Hermite mode; no
assumption \(d=0\) is made.

\subsection{The first weighted fold jet}

Let \(\mathcal D_{r,\varsigma,L}=(\mathcal V_{r,L},Y_{r,\varsigma,L})\)
be the inner field and any admissible flow graph with the enhanced
weighted \(C_x^5C_L^2\) bounds of \cref{def:enhanced-selection}, on the side
\(\varsigma=-1\) (attracting) or \(\varsigma=1\) (repelling).  This class
includes the constructed references.  Use the response window
\(\mathcal W_r(J)\) from \eqref{eq:main-response-window}, with \(C_W\)
chosen larger than the root-location constant in
\cref{app-geom:flow-profile}.

If a final-cell stage crosses \(X=0\), continue the corresponding one-sided
flow graph across the core over one fixed inner interval by its scalar
flow equation.  This is an evaluation extension determined by that side;
it neither changes the selection nor identifies the two side graphs.

\begin{lemma}[Uniform local datum jet]
 \label{app-res:datum-jet}
On \(0\le\varsigma X\le\delta/r\), including the evaluation extensions,
\begin{align}
 \mathcal V_{r,L}&=\mathcal V_0+r\mathcal V_1+\ord(r^2P(X)),
 \label{app-res:field-first-jet}\\
 \mathcal V_1(X,Y)
 &=\left(AX^3+BXY+C_\eps X+DXL_0,
 -L_0+EX^2+F_sY+S_\eps\right),
 \label{app-res:field-jet-formula}\\
 Y_{r,\varsigma,L}(X)
 &=X^2-\frac12+r\{F_3X^3+2\vartheta_0X\}
   +\ord(r^2P(X)),
 \label{app-res:graph-first-jet}
\end{align}
through the five state derivatives needed below, uniformly on
\(\mathcal W_r(J)\), where
\begin{equation}
 \label{app-res:compatibility-coefficients}
 \mu_0=C_\eps-\frac B2+DL_0,
 \qquad
 \vartheta_0=\frac{\mu_0}{2}+\frac{3F_3-2G_2}{8}.
\end{equation}
After one \(L\)-derivative, the remainder bounds may lose one power of
\(r\); this is the only derivative loss used later.
\end{lemma}

\begin{proof}
The exact fold conjugacy gives
\[
 \begin{aligned}
  X'&=X^2-Y+r(AX^3+BXY+C_\eps X+DXL)+\ord(r^2P),\\
  Y'&=X+r(-L+EX^2+F_sY+S_\eps)+\ord(r^2P).
 \end{aligned}
\]
Since \(L-L_0=O(r)\) on \(\mathcal W_r(J)\), freezing \(L\) at \(L_0\)
changes only the value remainder in \eqref{app-res:field-first-jet}.
Substitute \(Y=X^2-1/2+rY_1+\cdots\) into the flow tangency equation.  At
\(L=L_0\), the polynomial solution is
\(Y_1=F_3X^3+2\vartheta_0X\); equality of its constant and quadratic
coefficients is \eqref{app-res:compatibility-coefficients}, while the last
compatibility relation is \eqref{eq:main-L0-definition}.

The scalar fold equation and its Gaussian coreward propagator first give
this expansion for a constructed reference on both sides.  Compare an
enhanced selected graph with that reference.  The homogeneous secant
equation and its differentiated triangular systems give, for the required
jets,
\begin{equation}
 \label{app-res:enhanced-selection-comparison}
 \left|
 \partial_X^i\partial_L^j
 (Y_{r,\varsigma,L}^{\sigma}
  -Y_{r,\varsigma,L}^{\rm ref})
 \right|
 \le Cr^{-M}P(X)
 \exp\!\left[-c\left\{(\delta/r)^2-X^2\right\}\right].
\end{equation}
On the effective core this is flat; on the complementary outer part it is
bounded by an enlarged \(r^2P(X)\).  Thus every enhanced selection has the
same first weighted jet with the stated remainder.  Multiplication by the
central Gaussian product also shows that its selection-dependent part
contributes only \(r^{-M}\exp(-c/r^2)\) to the weighted chain sum.
Differentiating tangency five times in \(X\) and once in \(L\) gives the
remaining claimed bounds.  The evaluation extension solves the same
equation.
\end{proof}

\subsection{Gaussian projection and discrete quadrature}

Let
\[
 I_0=\int_0^\infty\eexp^{-2X^2}\dd X=\sqrt{\frac\pi8},
 \qquad
 \mathbb E[p]=I_0^{-1}\int_0^\infty\eexp^{-2X^2}p(X)\dd X.
\]
Integration by parts gives
\begin{equation}
 \label{app-res:gaussian-moments}
 \mathbb E[X^2]=\frac14,
 \qquad \mathbb E[\mathscr H_2]=0.
\end{equation}
Applying this to \eqref{eq:Q10} yields
\begin{equation}
 \label{app-res:Q10-projection}
 \begin{aligned}
 \mathbb E[Q_{10}]
 &=-\frac38\beta F_3
  +\left(\frac12\alpha+\frac14\beta\right)G_2
  +\frac12(\beta-\alpha)G_2\\
 &=-\frac38\beta(F_3-2G_2).
 \end{aligned}
\end{equation}

The following estimate links the local residual to the discrete Green
identity.

\begin{lemma}[Weighted-chain quadrature]
 \label{app-res:quadrature-interface}
Suppose an aligned coreward chain has nodes
\(0=z_0<\cdots<z_N\), sources \(\xi_{k+1}\), and multiplier products
\(P_k\), with \(z_N\asymp r^{-1}\), and suppose
\begin{align*}
 z_{k+1}-z_k&=\frac H2+\ord(HrP_*(z_{k+1})+H^2P_*(z_{k+1})),\\
 \abs{\xi_{k+1}-z_k}&\le CHP_*(z_{k+1}),\\
 \abs{P_k-\eexp^{-2z_k^2}}
 &\le C(r+H)P_*(z_k)\eexp^{-cz_k^2}
   +Cr^{-M}\eexp^{-c/r^2},\\
 H\sum_{k<N}P_k(1+z_k)^m&\le C_m
\end{align*}
for the fixed degrees used below.  Then every fixed polynomial \(p\)
satisfies
\begin{equation}
 \label{app-res:weighted-quadrature}
 H\sum_{k<N}P_kp(\xi_{k+1})
 =2\int_0^\infty\eexp^{-2z^2}p(z)\dd z
  +\ord(r+H)+\ord(r^{-M'}\eexp^{-c'/r^2}).
\end{equation}
In particular,
\begin{equation}
 \label{app-res:discrete-hermite-cancellation}
 H\sum_{k<N}P_k\mathscr H_2(\xi_{k+1})
 =\ord(r+H)+\ord(r^{-M'}\eexp^{-c'/r^2}).
\end{equation}
\end{lemma}

\begin{proof}
The product and source comparisons, followed by the weighted moment bound,
replace \(P_kp(\xi_{k+1})\) by
\(\exp(-2z_k^2)p(z_k)\) with total error \(O(r+H)\).  The mesh relation
turns \(H\) into twice the cell length.  The resulting Riemann sum has
error \(O(r+H)\), because every derivative of a polynomial times a
Gaussian is integrable.  The tail beyond \(z_N\asymp r^{-1}\) is a fixed
power of \(r^{-1}\) times \(\exp(-c/r^2)\).  The Hermite conclusion follows
from \eqref{app-res:gaussian-moments}.
\end{proof}

The hypotheses of \cref{app-res:quadrature-interface} are verified for the
aligned chains in \cref{app:transport-root}.  Retaining the factor
\(H\) is essential: although the chain has \(O(H^{-1})\) cells, the
accumulated estimate contains no inverse step-size factor.  Together,
\cref{app-res:exact-division} and
\cref{app-res:oriented-residual,app-res:Q10-projection} prove
\cref{prop:exact-residual}.

\section{Exact transport and capture of the numerical root}
\label{app:transport-root}

This appendix proves \cref{prop:fold-response,lem:root-transfer}.  The
argument has two logically separate parts.  First, an exact common-target
identity transports the one-step residual without replacing the numerical
map by a modified flow.  Second, an interior interpolation estimate uses
the transported value bound to recover monotonicity, while the location of
the root is obtained from the value bound alone.  We retain the factor
\(H\) in every discrete convolution.  This prevents the number of cells
from introducing a lower restriction on the step size.

Throughout the appendix, reduce \(h_0\) so that \(h_0\le1\).  Then
\begin{equation}
 \label{app-tr:small-rectangle}
 0<r\le r_0,\qquad 0<h\le h_0,\qquad H=hr\le r.
\end{equation}
Write \(\varsigma=-1\) on the attracting side and \(\varsigma=1\) on
the repelling side.  The corresponding coreward tableau is \(\theta\) or
\(\theta^\dagger\), respectively, and its signed inner step is
\(\widehat H=-\varsigma H\).

\subsection{An exact common-target identity}

We first isolate the elementary identity used on every fold cell.  It is
stated for a general map because the adjustment of the two source points is
essential.

\begin{lemma}[Common-target cell identity]
 \label{app-tr:common-target}
Let
\[
 \Phi(z,v)=\bigl(T(z,v),V(z,v)\bigr),
 \qquad T(z,v)=z-H\mathcal P(z,v),
\]
be \(C^1\) on a convex neighbourhood.  Let \(w\) be an invariant graph
and let \(e\) be a comparison graph.  Suppose that the projected graph
maps are strictly monotone and that \(\zeta\) on \(w\) and \(\xi\) on
\(e\) have the same target \(t\):
\begin{equation}
 \label{app-tr:aligned-target}
 T(\zeta,w(\zeta))=T(\xi,e(\xi))=t.
\end{equation}
For \(G\in\{\mathcal P,V\}\), denote by \(\overline G_z\) and
\(\overline G_v\) the secant averages of its two partial derivatives on
the segment joining \((\xi,e(\xi))\) to \((\zeta,w(\zeta))\), and put
\begin{equation}
 \label{app-tr:exact-multiplier}
 a=\frac{\overline{\mathcal P}_v}
         {1-H\overline{\mathcal P}_z},
 \qquad
 \overline e'=
 \int_0^1e'\bigl(\xi+s(\zeta-\xi)\bigr)\,\dd s,
 \qquad
 M=\frac{\overline V_v+Ha\overline V_z}
         {1-Ha\overline e'}.
\end{equation}
If the output-minus-graph residual of \(e\) is
\begin{equation}
 \label{app-tr:cell-residual-definition}
 \rho=V(\xi,e(\xi))-e(t),
\end{equation}
then, for \(d=w-e\),
\begin{equation}
 \label{app-tr:exact-cell-recurrence}
 d(t)=M d(\zeta)+\rho.
\end{equation}
Consequently, a contained chain with coreward targets
\(0=z_0<z_1<\cdots<z_N\) satisfies the finite identity
\begin{equation}
 \label{app-tr:exact-green}
 d(z_0)=\Pi_Nd(z_N)+\sum_{k=0}^{N-1}\Pi_k\rho_k,
 \qquad
 \Pi_0=1,\qquad \Pi_j=\prod_{k<j}M_k.
\end{equation}
\end{lemma}

\begin{proof}
Let \(\Delta z=\zeta-\xi\) and
\(\Delta v=w(\zeta)-e(\xi)\).  Secant subtraction in
\eqref{app-tr:aligned-target} gives
\[
 \Delta z=H\{\overline{\mathcal P}_z\Delta z
                  +\overline{\mathcal P}_v\Delta v\}
            =Ha\Delta v.
\]
On the other hand,
\[
 \Delta v=d(\zeta)+e(\zeta)-e(\xi)
          =d(\zeta)+\overline e'\Delta z,
 \qquad
 \Delta v=\frac{d(\zeta)}{1-Ha\overline e'}.
\]
Invariance of \(w\), the definition of \(\rho\), and another secant
subtraction give
\[
 d(t)=\overline V_z\Delta z+\overline V_v\Delta v+\rho,
\]
which is \eqref{app-tr:exact-cell-recurrence}.  Iteration proves
\eqref{app-tr:exact-green}.
\end{proof}

For the fold comparison, \(w\) is a numerical invariant graph and
\begin{equation}
 \label{app-tr:shifted-comparator}
 e=E+q_\theta H^2
\end{equation}
is the flow graph in the inner \(Y\)-coordinate with the common
comparator offset.  The graph construction in
\cref{app-geom:exact-regraphing} provides the contained chain.  Reversing
its stopped coreward images gives
\begin{equation}
 \label{app-tr:stopped-chain}
 0=z_0<z_1<\cdots<z_N,
 \qquad \frac{c_*}{r}\le z_N\le\frac{C_*}{r}.
\end{equation}
The source of the numerical graph in cell \(k\) is \(z_{k+1}\); the
aligned source on \(e\) is denoted by \(\xi_{k+1}\).  No shortened
Runge--Kutta step is inserted at either end.

The stopping count \(N\) can change with \((J,\theta,r,h,L)\).  On an
open piece on which it is constant, all cell data are as differentiable in
\(L\) as the graphs.  On the overlap of two adjacent pieces, adding the
last contained cell replaces
\[
 \Pi_Nd(z_N)
 \quad\hbox{by}\quad
 \Pi_N\{M_Nd(z_{N+1})+\rho_N\}
 =\Pi_{N+1}d(z_{N+1})+\Pi_N\rho_N.
\]
Thus the two Green expressions agree exactly, together with their
\(L\)-derivatives.  This add/delete-one-cell compatibility will be used
again for the moving-terminal limit.

\subsection{Gaussian products and mesh-weighted quadrature}

The following estimates contain the only summation over the fold chain.
Here and below \(P\) denotes a fixed positive polynomial majorant whose
degree may increase from line to line.  Fixed algebraic powers of
\(r^{-1}\) are absorbed by increasing the integer \(M\).

\begin{lemma}[Products and half-Gaussian quadrature]
 \label{app-tr:gaussian-quadrature}
Uniformly on both signed chains, the cell mesh and multiplier obey,
as long as \(z_{k+1}\le\eta/r\) for a fixed small \(\eta>0\),
\begin{align}
 z_{k+1}-z_k
 &=\frac H2+\ord\bigl(HrP(z_{k+1})\bigr),
 \label{app-tr:mesh}\\
 M_k
 &=1-2Hz_{k+1}
   +\ord\bigl(HrP(z_{k+1})+H^2P(z_{k+1})\bigr),
 \label{app-tr:multiplier-jet}\\
 \abs{\xi_{k+1}-z_k}&\le CHP(z_{k+1}).
 \label{app-tr:aligned-source}
\end{align}
On a fixed-stopping-index parameter piece, let \(\mathsf D_L\) denote the
total \(L\)-derivative of cell data, including the motion of their sources
and targets.  After increasing \(P\),
\begin{align}
 \abs{\mathsf D_Lz_k}+\abs{\mathsf D_L\xi_{k+1}}
 &\le CrP(z_{k+1}),
 \label{app-tr:mesh-L-derivative}\\
 \abs{\mathsf D_LM_k}&\le CHrP(z_{k+1}).
 \label{app-tr:multiplier-L-derivative}
\end{align}
On every contained cell one has positive spacing comparable to \(H\),
\(M_k\ge1/2\), and
\begin{equation}
 \label{app-tr:coarse-products}
 \begin{aligned}
 0<\Pi_j&\le C\eexp^{-cz_j^2},
 &\abs{\mathsf D_L\Pi_j}&\le CrP(z_j)\eexp^{-cz_j^2},\\
 H\sum_{j<N}
 \left(\Pi_j+r^{-1}\abs{\mathsf D_L\Pi_j}\right)
 (1+z_j)^m&\le C_m.
 \end{aligned}
\end{equation}
for every degree needed in the proof.  More precisely,
\begin{equation}
 \label{app-tr:product-comparison}
 \norm{\Pi_j-\eexp^{-2z_j^2}}_{C_L^1,\mathrm{pw}}
 \le C(r+H)P(z_j)\eexp^{-cz_j^2}
      +Cr^{-M}\eexp^{-c/r^2}.
\end{equation}
Consequently,
\begin{align}
 \left\|H\sum_{k<N}\Pi_kQ_{10}(\xi_{k+1})
 -2\int_0^\infty\eexp^{-2X^2}Q_{10}(X)\,\dd X
 \right\|_{C_L^1,\mathrm{pw}}
 &\le C\bigl(r+H+r^{-M}\eexp^{-c/r^2}\bigr),
 \label{app-tr:Q-quadrature}\\
 \left\|H\sum_{k<N}\Pi_k\mathscr H_2(\xi_{k+1})
 \right\|_{C_L^1,\mathrm{pw}}
 &\le C\bigl(r+H+r^{-M}\eexp^{-c/r^2}\bigr).
 \label{app-tr:H-quadrature}
\end{align}
\end{lemma}

\begin{proof}
The slow-tube estimate in \cref{app-geom:slow-tube}, the projection bound
\eqref{app-geom:projection-control}, and the stage bounds in
\cref{app-res:adjoint-lemma} keep both aligned sources and their joining
segment in a common graph neighbourhood.  Substitution of the fold jets
into the secant formula \eqref{app-tr:exact-multiplier} gives
\eqref{app-tr:mesh}--\eqref{app-tr:aligned-source}.  The argument uses
only compact tableau bounds.  In particular, it does not use signs of the
Runge--Kutta weights.

We record the parameter recurrence before estimating products.  A naive
sum of cellwise parameter bounds would lose a factor \(H^{-1}\).  On a
fixed-combinatorics piece set \(\ell_k=\mathsf D_Lz_k\); since \(z_0=0\),
\(\ell_0=0\).  Differentiate the target equation and solve for the
next source.  The inner fold jet and the common slow-tube bounds give
\begin{equation}
 \ell_{k+1}
 =\{1+HrA_k\}\ell_k+HrB_k,
 \qquad
 \abs{A_k}\le C,\qquad \abs{B_k}\le P(z_{k+1}).
 \label{app-tr:mesh-sensitivity-recurrence}
\end{equation}
The factor \(H\) is the cell length; the additional \(r\) is the
parameter weight in the fold jet.  Iteration retains this mesh factor:
\begin{equation}
 \abs{\ell_k}
 \le CrH\sum_{j<k}P(z_{j+1})
       \prod_{j<\ell<k}\{1+CHr\}
 \le CrP_1(z_k).
 \label{app-tr:mesh-sensitivity-bound}
\end{equation}
Since \(Hr k\le Crz_k\le C\delta\), the product in
\eqref{app-tr:mesh-sensitivity-bound} is uniformly bounded; the remaining
mesh sum is bounded by a polynomial integral.  Differentiating the
aligned-target equation gives the
same bound for \(\mathsf D_L\xi_{k+1}\).  Substitution into the
secant formula \eqref{app-tr:exact-multiplier} gives
\eqref{app-tr:multiplier-L-derivative}.  Thus every differentiated cell
factor still contains \(H\); none is merely \(O(r)\).

The transverse factor in \eqref{app-geom:bridge-damping} and the
secant formula give, on every contained segment,
\begin{equation}
 \prod_{a\le k<b}M_k
 \le C\exp\{-c(z_b^2-z_a^2)\},
 \qquad
 \frac H C\le z_{k+1}-z_k\le CH.
 \label{app-tr:segment-product}
\end{equation}
For the finitely many cells nearest the core, the first factor is absorbed
in \(C\); farther out it follows by summing
\(\log M_k\le-c_0Hz_{k+1}+C_0H\) and using the spacing bounds.  Completing
the square absorbs the accumulated linear term.  Taking
\(a=0\) proves the undifferentiated product estimate.  Moreover,
\begin{equation}
 H\sum_{k<N}\eexp^{-cz_k^2}(1+z_k)^m
 \le C_m
 \label{app-tr:weighted-polynomial-sum}
\end{equation}
by comparison, cell by cell, with the corresponding Gaussian integral.
This proves the undifferentiated parts of
\eqref{app-tr:coarse-products} without replacing the sum by the number
\(O(H^{-1})\) of cells.

For the derivative, differentiate the finite recurrence
\(\Pi_{j+1}=M_j\Pi_j\):
\begin{equation}
 \mathsf D_L\Pi_j
 =\sum_{\ell<j}
   \left(\prod_{k<\ell}M_k\right)
   (\mathsf D_LM_\ell)
   \left(\prod_{\ell<k<j}M_k\right).
 \label{app-tr:differentiated-product}
\end{equation}
Each summand contains \(HrP(z_{\ell+1})\) by
\eqref{app-tr:multiplier-L-derivative}.  The two segment products in
\eqref{app-tr:segment-product} combine to
\(CP(z_j)\eexp^{-cz_j^2}\).  Summing the remaining factor \(H\) gives
\[
 \abs{\mathsf D_L\Pi_j}
 \le CrP(z_j)\eexp^{-cz_j^2}.
\]
Applying \eqref{app-tr:weighted-polynomial-sum} once more proves the last
bound in \eqref{app-tr:coarse-products}.  This calculation exhibits why
one \(L\)-derivative costs a factor \(r\), but never a factor \(H^{-1}\).

Put
\[
 \overline M_k
 =\exp\{-2(z_{k+1}^2-z_k^2)\}.
\]
Equations \eqref{app-tr:mesh} and
\eqref{app-tr:multiplier-jet} give
\(\abs{M_k-\overline M_k}\le CH(r+H)P(z_{k+1})\) on effective
cells.  The differentiated cell formulas also give
\begin{equation}
 \abs{\mathsf D_L(M_k-\overline M_k)}
 \le CH(r+H)P(z_{k+1}).
 \label{app-tr:differentiated-cell-comparison}
\end{equation}
The exact product--Duhamel identity
\begin{equation}
 \label{app-tr:product-duhamel}
 \Pi_j-\prod_{k<j}\overline M_k
 =\sum_{\ell<j}
   \left(\prod_{k<\ell}M_k\right)
   (M_\ell-\overline M_\ell)
   \left(\prod_{\ell<k<j}\overline M_k\right)
\end{equation}
and the segment estimates prove the value part of
\eqref{app-tr:product-comparison}; the ideal product telescopes exactly to
\(\eexp^{-2z_j^2}\).  Differentiating
\eqref{app-tr:product-duhamel} distributes one \(L\)-derivative among a
prefix product, the cell difference, and an ideal suffix product.
Equations \eqref{app-tr:differentiated-product} and
\eqref{app-tr:differentiated-cell-comparison} show that each resulting sum
again contains a mesh factor \(H\).  The same Gaussian convolution proves
the \(C_L^1\) part of \eqref{app-tr:product-comparison}.

For clarity, we separate the fold core from the outer tail.  Let
\(k_\eta\) be the first index for which \(z_{k_\eta}\ge\eta/r\).  At that
index, \eqref{app-tr:segment-product} already gives
\(\Pi_{k_\eta}\le C\eexp^{-c\eta^2/r^2}\).  Applying the segment estimate
again after \(k_\eta\), and using
\eqref{app-tr:differentiated-product} for the derivative, gives
\begin{equation}
 \label{app-tr:outer-tail}
 H\sum_{k\ge k_\eta}
 \left(\Pi_k+r^{-1}\abs{\mathsf D_L\Pi_k}\right)(1+z_k)^m
 +\int_{\eta/(2r)}^{\infty}\eexp^{-cX^2}(1+X)^m\,\dd X
 \le Cr^{-M}\eexp^{-c'/r^2}.
\end{equation}
All algebraic losses incurred beyond the effective region are absorbed by
the displayed power \(r^{-M}\).

It remains to calculate the core sum.  For either
\(F=Q_{10}\) or \(F=\mathscr H_2\), write it as
\begin{equation}
 \begin{aligned}
 H\sum_{k<k_\eta}\Pi_kF(\xi_{k+1})
 ={}&2\sum_{k<k_\eta}(z_{k+1}-z_k)
       \eexp^{-2z_k^2}F(z_k)+E_1+E_2+E_3,\\
 E_1&=H\sum_{k<k_\eta}
       (\Pi_k-\eexp^{-2z_k^2})F(\xi_{k+1}),\\
 E_2&=H\sum_{k<k_\eta}\eexp^{-2z_k^2}
       \{F(\xi_{k+1})-F(z_k)\},\\
 E_3&=\sum_{k<k_\eta}
       \{H-2(z_{k+1}-z_k)\}\eexp^{-2z_k^2}F(z_k).
 \end{aligned}
 \label{app-tr:quadrature-error-decomposition}
\end{equation}
The product comparison and
\eqref{app-tr:weighted-polynomial-sum} give
\(\norm{E_1}_{C_L^1,\mathrm{pw}}\le C(r+H)\).  The aligned-source bound
gives \(\norm{E_2}_{C_L^1,\mathrm{pw}}\le C(r+H)\), and the mesh expansion
gives \(\norm{E_3}_{C_L^1,\mathrm{pw}}\le Cr\).  In each differentiated
sum, \eqref{app-tr:mesh-L-derivative}--
\eqref{app-tr:multiplier-L-derivative} supplies the factor that replaces
the derivative; the weighted polynomial sum then gives the same bound.
Finally, the first line of
\eqref{app-tr:quadrature-error-decomposition} differs from
\(2\int_0^{z_{k_\eta}}\eexp^{-2X^2}F(X)\,\dd X\) by \(O(H)\), including
one total \(L\)-derivative.  Adding the outer estimate
\eqref{app-tr:outer-tail} proves \eqref{app-tr:Q-quadrature}.  For
\(F=\mathscr H_2\), the limiting integral is zero by
\eqref{eq:hermite-zero-moment}, which proves
\eqref{app-tr:H-quadrature}.
\end{proof}

The mesh-weighted form of \eqref{app-tr:coarse-products} prevents a loss
proportional to the number of cells.  Every subsequent absolute bound is uniform as
\(H\downarrow0\), including when \(h\ll r^2=\eps\); only the
\(h\)-independent ambiguity of arbitrary selections remains.

\subsection{The two terminal-data regimes}

The exact residual of the shifted flow comparator can be written, on the
side \(\varsigma\), as
\begin{equation}
 \label{app-tr:oriented-residual}
 \rho_{\varsigma,k}
 =-\varsigma H^3\mathcal Q_\varsigma
       (r,H,\xi_{k+1},L),
\end{equation}
where \cref{prop:exact-residual} gives
\begin{equation}
 \label{app-tr:full-residual-quotient}
 \mathcal Q_\varsigma(r,H,X,L)
 =rQ_{10}(X)+H\gamma_\theta\mathscr H_2(X)
  +\ord\bigl(r^2P(X)+rHP(X)+H^2P(X)\bigr).
\end{equation}
The adjoint identities in \eqref{eq:adjoint-defects} are what put both
sides into the single oriented formula \eqref{app-tr:oriented-residual}.

For arbitrary independently selected graphs, take first the fixed
reference flow and map graphs from \cref{app-geom:construction}.  Their
outer divided-coordinate separation is at most \(Cr^{-M}\).  Therefore
\begin{equation}
 \label{app-tr:arbitrary-terminal}
 \abs{\Pi_Nd_\varsigma(z_N,L)}
 \le Cr^{-M}\eexp^{-c/r^2}.
\end{equation}
Substituting \eqref{app-tr:oriented-residual} into the Green identity,
using \eqref{app-tr:coarse-products} for the remainder and
\cref{app-tr:Q-quadrature,app-tr:H-quadrature} for the two displayed
modes, gives
\begin{equation}
 \label{app-tr:arbitrary-one-side}
 \begin{aligned}
 d_\varsigma(0,L)
 &=-2\varsigma rH^2 I_Q
   +\ord(r^2H^2+rH^3)
   +\ord(r^{-M}\eexp^{-c/r^2}),\\
 I_Q&=\int_0^\infty\eexp^{-2X^2}Q_{10}(X)\,\dd X.
 \end{aligned}
\end{equation}
Indeed, the residual remainder contributes
\[
 H^2\,H\sum_{k<N}\Pi_k
 \ord(r^2P+rHP+H^2P)
 =\ord(r^2H^2+rH^3),
\]
and the Hermite contribution is
\(H^3\{H\sum\Pi_k\mathscr H_2\}=O(rH^3+H^4)\).
Both estimates use \(H\le r\).

For a fold-matched pair \(\mathfrak p\), the terminal datum itself carries
the response scale.  In the inner coordinate, condition
\eqref{eq:fold-matched-condition} and the comparator shift define on the
original normalized collar
\[
 \gamma_{\varsigma,h}(z,L)
 =H^{-2}d_{\varsigma,h}(z,L).
\]
At the moving terminal this gives
\begin{equation}
 \label{app-tr:matched-terminal}
 d_{\varsigma,h}(z_N,L)
 =H^2\gamma_{\varsigma,h}(z_N(L),L),
 \qquad
 \norm{\gamma_{\varsigma,h}(z_N(\cdot),\cdot)}_{C_L^1,\mathrm{pw}}
 \le C_{\mathfrak p}r^{-M}.
\end{equation}
Here \(\mathrm{pw}\) refers to the fixed-stopping-index pieces described
after \eqref{app-tr:stopped-chain}.  The Gaussian endpoint product now
gives
\begin{equation}
 \label{app-tr:matched-terminal-shielding}
 \norm{\Pi_N\gamma_{\varsigma,h}(z_N(\cdot),\cdot)}_{C^0}
 \le C_{\mathfrak p}r^{-M}\eexp^{-c/r^2}.
\end{equation}
After division of \eqref{app-tr:exact-green} by \(H^2\), the same
mesh-weighted calculation therefore yields, on the full rectangle,
\begin{equation}
 \label{app-tr:matched-one-side}
 \left\|
 \frac{d_{\varsigma,h}(0,\cdot)}{H^2}
 +2\varsigma rI_Q
 \right\|_{C^0}
 \le C_{\mathfrak p}
 \left(r^2+rH+r^{-M}\eexp^{-c/r^2}\right).
\end{equation}
No division of an \(h\)-independent terminal floor by \(H^2\) occurs in
this estimate.

We next record the stronger fixed-\(r\) consequence of matching.  It is
not obtained by differentiating the Gaussian approximation.

\begin{lemma}[Moving-terminal cancellation]
 \label{app-tr:moving-terminal}
For each fixed \(0<r\le r_0\) and every fold-matched rule
\(\mathfrak p\), there is a function
\(\chi_\varsigma^{\mathfrak p}(J,r,L,\theta)\) such that
\begin{equation}
 \label{app-tr:fixed-r-one-side-limit}
 \left\|
 \frac{d_{\varsigma,h}(0,\cdot)}{H^2}
 -\chi_\varsigma^{\mathfrak p}(J,r,\cdot,\theta)
 \right\|_{C_L^1}
 \le C_{r,\mathfrak p}h.
\end{equation}
Moreover,
\begin{equation}
 \label{app-tr:chi-fold-asymptotic}
 \chi_\varsigma^{\mathfrak p}
 =-2\varsigma rI_Q
  +\ord\bigl(r^2+r^{-M}\eexp^{-c/r^2}\bigr)
\end{equation}
in \(C_L^0\), uniformly over the fold and method classes as
\(r\downarrow0\).
\end{lemma}

\begin{proof}
Fix \(r>0\).  The response window is covered by open sets
\(\mathcal I_N\) on which the same \(N\) full cells are retained before
the terminal collar arc.  Projection monotonicity and the strict stage
margins imply that
\[
 z_k=z_k(L),\quad \xi_{k+1}=\xi_{k+1}(L),\quad M_k=M_k(L)
\]
are \(C^1\) on \(\mathcal I_N\).  Differentiating the fixed-
\(r\) target recurrence gives
\begin{equation}
 \mathsf D_Lz_{k+1}
 =\{1+O_r(H)\}\mathsf D_Lz_k+O_r(H),
 \qquad \mathsf D_Lz_0=0,
 \label{app-tr:fixed-r-terminal-motion}
\end{equation}
and discrete Gronwall over \(O_r(H^{-1})\) cells gives
\begin{equation}
 \max_{k\le N}
 \left(\abs{\mathsf D_Lz_k}+\abs{\mathsf D_L\xi_{k+1}}\right)
 \le C_{r,\mathfrak p}.
 \label{app-tr:fixed-r-terminal-motion-bound}
\end{equation}
Thus differentiating a quantity evaluated at the moving terminal does not
introduce an \(H^{-1}\) factor.

On the collar, let \(U_{\varsigma,2}\) be the coefficient in
\eqref{eq:fold-matched-condition} and set
\begin{equation}
 \label{app-tr:continuous-terminal-coefficient}
 \gamma_{\varsigma,0}(z,L)
 =U_{\varsigma,2}(\varsigma rz,L)-q_\theta.
\end{equation}
The matching estimate gives, at the same moving terminal source,
\begin{equation}
 \label{app-tr:terminal-rate}
 \norm{
 \gamma_{\varsigma,h}(z_N(\cdot),\cdot)
 -\gamma_{\varsigma,0}(z_N(\cdot),\cdot)
 }_{C_L^1,\mathrm{pw}}
 \le C_{r,\mathfrak p}h.
\end{equation}
Indeed, the total derivative here is
\(\partial_L+(\mathsf D_Lz_N)\partial_z\); the bound follows from
\eqref{eq:fold-matched-condition} and
\eqref{app-tr:fixed-r-terminal-motion-bound}.

Let \(p_\varsigma(z,r,L)>0\) be the continuous coreward speed in the
signed fold coordinate, and let
\(G_\varsigma(0,z;r,L)\) be its normal propagator from \(z\) to zero,
normalized by \(G_\varsigma(0,0)=1\).  Divide the graph-invariance
equations for the matched map and flow graphs by \(H^2\).  On every
contained collar cell the resulting normalized secant
\(\Gamma_h=d_{\varsigma,h}/H^2\) satisfies the exact identity
\begin{equation}
 \label{app-tr:normalized-cell-compatibility}
 \Gamma_h(z_k,L)
 =M_k\Gamma_h(z_{k+1},L)
  -\varsigma H\mathcal Q_\varsigma
      (r,H,\xi_{k+1},L).
\end{equation}
By fold matching,
\(\Gamma_h=U_{\varsigma,2}-q_\theta+O_{C_x^3C_L^1,r}(h)\) on the
original normalized collar.  Let
\(G_k=G_\varsigma(0,z_k;r,L)\).  The cell expansions proved below give
\[
 z_{k+1}-z_k=Hp_\varsigma(z_{k+1},r,L)+O_r(H^2),
 \qquad
 M_k=G_{k+1}/G_k+O_r(H^2).
\]
Multiplying \eqref{app-tr:normalized-cell-compatibility} by \(G_k\),
moving its two graph values to opposite sides, and dividing by
\(z_{k+1}-z_k\) therefore gives
\begin{equation}
 \frac{G_{k+1}\Gamma_h(z_{k+1},L)-G_k\Gamma_h(z_k,L)}
      {z_{k+1}-z_k}
 =\varsigma G_k
   \frac{\mathcal Q_\varsigma(r,H,\xi_{k+1},L)}
        {p_\varsigma(z_{k+1},r,L)}+O_{C_L^1,r}(H).
 \label{app-tr:cell-to-continuous-compatibility}
\end{equation}
The two additional \(x\)-jets in the matching hypothesis control the
difference quotient after one \(L\)-derivative.  Letting \(h\downarrow0\)
in \eqref{app-tr:cell-to-continuous-compatibility} gives on the collar
overlap
\begin{equation}
 \label{app-tr:continuous-compatibility}
 \partial_z\!\left\{
 G_\varsigma(0,z;r,L)\gamma_{\varsigma,0}(z,L)
 \right\}
 =\varsigma G_\varsigma(0,z;r,L)
   \frac{\mathcal Q_\varsigma(r,0,z,L)}
        {p_\varsigma(z,r,L)}.
\end{equation}
This first-order equation uniquely continues the collar coefficient
coreward through the original normalized bridge.
It follows that the complete stopped expression
\begin{equation}
 \label{app-tr:continuous-stopped-expression}
 \begin{aligned}
 \chi_\varsigma^{\mathfrak p}
 ={}&G_\varsigma(0,Z;r,L)\gamma_{\varsigma,0}(Z,L)\\
 &-\varsigma\int_0^Z
 G_\varsigma(0,z;r,L)
 \frac{\mathcal Q_\varsigma(r,0,z,L)}
      {p_\varsigma(z,r,L)}\,\dd z
 \end{aligned}
\end{equation}
is independent of every contained stopping point \(Z\).  Indeed, denoting
the right-hand side by \(\mathcal C_\varsigma(Z,L)\),
\begin{equation}
 \partial_Z\mathcal C_\varsigma(Z,L)
 =\partial_Z\{G_\varsigma(0,Z)\gamma_{\varsigma,0}(Z,L)\}
  -\varsigma G_\varsigma(0,Z)
    \frac{\mathcal Q_\varsigma(r,0,Z,L)}
         {p_\varsigma(Z,r,L)}=0
 \label{app-tr:continuous-endpoint-cancellation}
\end{equation}
by \eqref{app-tr:continuous-compatibility}.  Consequently, for every
moving endpoint \(Z=Z(L)\),
\begin{equation}
 \frac{\dd}{\dd L}\mathcal C_\varsigma(Z(L),L)
 =\partial_L\mathcal C_\varsigma(Z(L),L)
  +Z'(L)\partial_Z\mathcal C_\varsigma(Z(L),L)
 =\partial_L\mathcal C_\varsigma(Z(L),L).
 \label{app-tr:continuous-phase-derivative-cancellation}
\end{equation}
Thus the derivative of the terminal phase cancels before any estimate is
taken.

For fixed \(r\), all coefficients on \([0,C_*/r]\) have uniform
\(C_L^1\) bounds.  Stage differentiation and the common-target
formula give, cell by cell,
\begin{align}
 z_{k+1}-z_k
 &=Hp_\varsigma(z_{k+1},r,L)+\ord_{C_L^1,r}(H^2),
 \label{app-tr:fixed-r-mesh}\\
 M_k
 &=\frac{G_\varsigma(0,z_{k+1};r,L)}
         {G_\varsigma(0,z_k;r,L)}
   +\ord_{C_L^1,r}(H^2),
 \label{app-tr:fixed-r-multiplier}\\
 \xi_{k+1}&=z_k+\ord_{C_L^1,r}(H),
 \label{app-tr:fixed-r-source}\\
 \mathcal Q_\varsigma(r,H,\xi_{k+1},L)
 &=\mathcal Q_\varsigma(r,0,z_k,L)
   +\ord_{C_L^1,r}(H).
 \label{app-tr:fixed-r-residual}
\end{align}
These expansions follow directly from the stage implicit equations: the
inverse stage Jacobian is uniform at fixed \(r\), and differentiating once
in \(L\) preserves the indicated powers of \(H\).  In particular, the
local numerical and continuous variational maps agree to first order,
which is why the multiplier error has two powers of \(H\).

Set
\[
 \overline M_k
 =\frac{G_\varsigma(0,z_{k+1};r,L)}
        {G_\varsigma(0,z_k;r,L)},
 \qquad
 \overline\Pi_j=\prod_{k<j}\overline M_k
 =G_\varsigma(0,z_j;r,L).
\]
There are \(O_r(H^{-1})\) cells on this fixed interval, but every cell
difference \(M_k-\overline M_k\), including its total \(L\)-derivative,
is \(O_r(H^2)\).  The product--Duhamel identity therefore gives
\[
 \abs{\Pi_j-\overline\Pi_j}
 \le C_r\sum_{k<j}H^2\le C_rH.
\]
After one \(L\)-derivative, a derivative falling on a prefix or suffix
product produces a sum of \(O_r(H)\) differentiated factors and is hence
\(O_r(1)\); the distinguished cell difference remains \(O_r(H^2)\).
Alternatively, this follows by differentiating the finite Duhamel identity
before estimating it.  In both cases the result is
\begin{equation}
 \label{app-tr:fixed-r-product}
 \max_{k\le N}
 \norm{\Pi_k-G_\varsigma(0,z_k;r,\cdot)}_{C_L^1}
 \le C_{r,\mathfrak p}H.
\end{equation}

To obtain the Riemann sum, use the same moving endpoint in both formulas.
Cellwise, \eqref{app-tr:fixed-r-mesh} gives
\begin{equation}
 H=\frac{z_{k+1}-z_k}
         {p_\varsigma(z_{k+1},r,L)}+O_{C_L^1,r}(H^2).
 \label{app-tr:fixed-r-cell-measure}
\end{equation}
Replace successively \(\Pi_k\), \(\xi_{k+1}\), and \(H\) by their
continuous counterparts using
\cref{app-tr:fixed-r-product,app-tr:fixed-r-source,app-tr:fixed-r-cell-measure}.
The local \(C_L^1\) error is \(O_r(H^2)\),
so summing the \(O_r(H^{-1})\) cells gives
\begin{equation}
 \label{app-tr:fixed-r-riemann}
 \begin{aligned}
 \bigg\|H\sum_{k<N}\Pi_k
 \mathcal Q_\varsigma(r,H,\xi_{k+1},\cdot)
 -\int_0^{z_N}G_\varsigma(0,z;r,\cdot)
 \frac{\mathcal Q_\varsigma(r,0,z,\cdot)}
      {p_\varsigma(z,r,\cdot)}\,\dd z
 \bigg\|_{C_L^1,\mathrm{pw}}
 \le C_{r,\mathfrak p}h.
 \end{aligned}
\end{equation}
Here \(O_r(H)=O_r(h)\).  No terminal interval has been estimated
separately: the upper endpoint of the integral is exactly the discrete
endpoint \(z_N(L)\).

The terminal term has the equally explicit comparison
\begin{equation}
 \begin{aligned}
 &\Pi_N\gamma_{\varsigma,h}(z_N,L)
 -G_\varsigma(0,z_N;r,L)\gamma_{\varsigma,0}(z_N,L)\\
 &\quad=(\Pi_N-G_\varsigma(0,z_N;r,L))
        \gamma_{\varsigma,h}(z_N,L)\\
 &\qquad\quad+G_\varsigma(0,z_N;r,L)
   \{\gamma_{\varsigma,h}(z_N,L)-
      \gamma_{\varsigma,0}(z_N,L)\}.
 \end{aligned}
 \label{app-tr:fixed-r-terminal-decomposition}
\end{equation}
Equations \eqref{app-tr:terminal-rate},
\eqref{app-tr:fixed-r-product}, and
\eqref{app-tr:fixed-r-terminal-motion-bound} bound this by
\(C_{r,\mathfrak p}h\) in \(C_L^1\) on the patch.  The normalized
Green identity therefore differs from
\eqref{app-tr:continuous-stopped-expression}, evaluated at the same point
\(Z=z_N(L)\), by \(O_{C_L^1,r}(h)\).

It remains to paste the fixed-combinatorics estimates.  On a patch with
terminal index \(N\), write the normalized discrete expression as
\begin{equation}
 \mathcal D_N(L)
 =\Pi_N\gamma_{\varsigma,h}(z_N,L)
  -\varsigma H\sum_{k<N}\Pi_k
       \mathcal Q_\varsigma(r,H,\xi_{k+1},L).
 \label{app-tr:discrete-stopped-expression}
\end{equation}
Where the adjacent \(N\)- and \(N+1\)-cell descriptions are both valid,
the normalized cell identity
\[
 \gamma_{\varsigma,h}(z_N,L)
 =M_N\gamma_{\varsigma,h}(z_{N+1},L)
  -\varsigma H\mathcal Q_\varsigma(r,H,\xi_{N+1},L)
\]
gives, without an estimate,
\begin{equation}
 \mathcal D_{N+1}-\mathcal D_N
 =\Pi_N\bigl[
 M_N\gamma_{\varsigma,h}(z_{N+1},L)
 -\varsigma H\mathcal Q_\varsigma(r,H,\xi_{N+1},L)
 -\gamma_{\varsigma,h}(z_N,L)
 \bigr]=0.
 \label{app-tr:discrete-one-cell-gluing}
\end{equation}
The equality holds on an open overlap supplied by the fundamental collar
arc.  Differentiating the bracketed recurrence gives
\begin{equation}
 \mathsf D_L(\mathcal D_{N+1}-\mathcal D_N)=0,
 \label{app-tr:discrete-one-cell-derivative-gluing}
\end{equation}
so values and first \(L\)-derivatives have identical one-sided limits at a
stopping transition.  On the continuous side,
 \cref{app-tr:continuous-endpoint-cancellation,app-tr:continuous-phase-derivative-cancellation}
 give the corresponding
value and derivative identities for the two endpoints.  The piecewise
estimates therefore paste, with the same constant, to the global
\(C_L^1\) estimate \eqref{app-tr:fixed-r-one-side-limit}.

Finally, \eqref{app-tr:matched-one-side} is uniform in \(h\).  Passing to
the fixed-\(r\) limit and then using
\cref{app-tr:Q-quadrature,app-tr:H-quadrature} proves
\eqref{app-tr:chi-fold-asymptotic}.
\end{proof}

\subsection{The two-sided section response}

The offset \(q_\theta H^2\) in \eqref{app-tr:shifted-comparator} is the
same on both sides.  It therefore cancels algebraically before any
estimate is taken.  Since the unscaled normalized coordinate \(y\) equals
\(r^2Y\),
\begin{equation}
 \label{app-tr:two-side-conversion}
 \Delta_{\rk,\theta}(r^2L)-\Delta_\fl(r^2L)
 =r^2\{d_1(0,L)-d_{-1}(0,L)\}.
\end{equation}
Furthermore, \eqref{eq:Q10-projection} gives
\begin{equation}
 \label{app-tr:S-from-integral}
 4I_Q
 =-\frac34\beta_\theta\sqrt{\frac\pi2}\,\XiJ(J)
 =S_\theta(J).
\end{equation}

For the reference graphs, subtracting
\eqref{app-tr:arbitrary-one-side} with the two orientations and using
\cref{app-tr:two-side-conversion,app-tr:S-from-integral} proves
\begin{equation}
 \label{app-tr:reference-section-response}
 -\{\Delta_{\rk,\theta}^{\rm ref}(r^2L)
     -\Delta_\fl^{\rm ref}(r^2L)\}
 =S_\theta(J)r^3H^2
  +\ord(r^4H^2+r^3H^3+r^{-M}\eexp^{-c/r^2}).
\end{equation}
Exponential shielding \eqref{eq:selection-shielding}, applied within the
flow and within the map, transfers this identity to every pair of
independent selections.  This proves
\cref{eq:all-selection-section-response,eq:all-selection-response-remainder}.

For a fold-matched pair, apply \eqref{app-tr:matched-one-side} directly to
its two invariant graphs.  Then
\cref{app-tr:two-side-conversion,app-tr:S-from-integral} give
\eqref{eq:matched-response-remainder}.  At fixed \(r\), define
\begin{equation}
 \label{app-tr:matched-physical-response}
 \mathscr R_{\mathfrak p}(J,r,L,\theta)
 =-r^4\left(
 \chi_1^{\mathfrak p}(J,r,L,\theta)
 -\chi_{-1}^{\mathfrak p}(J,r,L,\theta)
 \right).
\end{equation}
Then \cref{app-tr:moving-terminal} yields
\begin{equation}
 \label{app-tr:matched-C1-response}
 \left\|
 -\frac{\Delta_{\rk,\theta}^{\mathfrak p}(r^2\cdot)
       -\Delta_\fl^{\mathfrak p}(r^2\cdot)}{h^2}
 -\mathscr R_{\mathfrak p}(J,r,\cdot,\theta)
 \right\|_{C_L^1}
 \le C_{r,\mathfrak p}h,
\end{equation}
and
\begin{equation}
 \label{app-tr:matched-response-asymptotic}
 \mathscr R_{\mathfrak p}(J,r,L,\theta)
 =S_\theta(J)r^5
  +\ord\bigl(r^6+r^{-M}\eexp^{-c/r^2}\bigr).
\end{equation}
This completes the proof of \cref{prop:fold-response}.

\subsection{Interior interpolation and root capture}

For each sufficiently small \(r\), choose nested response windows
\[
 I_r\Subset I_r^+\Subset\mathcal W_r(J),
 \qquad
 \operatorname{dist}(I_r,\partial I_r^+)\ge c_I r,
\]
which contain all flow roots from \cref{prop:fold-geometry} with the
uniform interior margin: every such root lies at distance at least
\(c_Ir\) from \(\partial I_r\).
Such a choice is possible by
\cref{app-geom:flow-profile,eq:main-response-window}, after reducing
\(c_I\) if necessary.  For one
pair of selections put
\begin{equation}
 \label{app-tr:root-notation}
 f(L)=\Delta_\fl(r^2L),
 \qquad m(L)=\Delta_{\rk,\theta}(r^2L),
 \qquad f(u_f)=0,
 \qquad n=f'(u_f).
\end{equation}
Uniformly over the choices,
\begin{equation}
 \label{app-tr:flow-secant-control}
 n=a(J)r^3+\ord(r^4+r^{-M}\eexp^{-c/r^2}),
 \qquad
 \sup_{L\in I_r}\abs{f'(L)/n-1}\le Cr.
\end{equation}
After reducing \(r_0\), \(n\asymp r^3>0\).

We use two elementary one-dimensional facts.  If
\(G\in C^2(I_r^+)\), \(\norm G_{C^0}\le A\),
\(\norm{G''}_{C^0}\le B\), and the distance from \(u\) to
\(\partial I_r^+\) is at least \(t\), Taylor's formula gives
\begin{equation}
 \label{app-tr:interpolation}
 \abs{G'(u)}\le\frac{2A}{t}+\frac B2t.
\end{equation}
Also, suppose \(F=m-f+n\kappa S\),
\(\norm F_{C^0}\le A\), and \(m'>0\) on a tube around \(u_f\).
If that tube contains the zero of \(m\), then, writing it as
\(u_m=u_f+q\),
\begin{equation}
 \label{app-tr:value-only-location}
 f(u_f+q)=n\kappa S-F(u_f+q).
\end{equation}
Thus a secant estimate for \(f\) gives the location from \(A\); no bound
on \(F'\) enters \eqref{app-tr:value-only-location}.

For completeness, existence of the required tube follows by a buffered
contraction.  If
\begin{equation}
 \label{app-tr:abstract-derivative-smallness}
 \sup_{I_r}\abs{f'/n-1}
 +\frac1n\left(\frac{2A}{t}+\frac B2t\right)=:\eta<1,
\end{equation}
then \(m/n\) has secants in \([1-\eta,1+\eta]\).  On a closed tube of
radius \(R\{S+A/n\}\), consider
\[
 \mathcal T(q)=q-\frac{m(u_f+q)}n.
\]
Its Lipschitz constant is at most \(\eta\), and
\(\abs{\mathcal T(0)}\le\abs\kappa S+A/n\).  Choosing \(R\) uniformly
larger than \(1+\sup\abs\kappa\) and then reducing \(r_0,h_0\) makes
\(\mathcal T\) a strict self-map.  It has one fixed point, and strict
monotonicity makes this the unique zero of \(m\) on the whole interval
\(I_r\).

Apply this argument to the corrected difference
\begin{equation}
 \label{app-tr:corrected-difference}
F(L)=m(L)-f(L)+nK_\theta(J)H^2.
\end{equation}
Thus the abstract quantities above are
\(\kappa=K_\theta(J)\) and \(S=H^2\).
The two-sided response, the relation
\(S_\theta=aK_\theta\), and \eqref{app-tr:flow-secant-control} give, for
arbitrary selections,
\begin{align}
 \norm F_{C^0(I_r^+)}
 &\le A_0
 :=C\{r^4H^2+r^3H^3+r^{-M}\eexp^{-c/r^2}\},
 \label{app-tr:arbitrary-corrected-C0}\\
 \norm{F''}_{C^0(I_r^+)}&\le B,
 \label{app-tr:corrected-C2}
\end{align}
where \(B\) is uniform by \cref{prop:fold-geometry}.  Choose
\(t=2\sqrt{A_0/B}\), with the zero cases interpreted by continuity.
Since \(H\le r\), this choice satisfies
\(t=O(hr^3+h^{3/2}r^3)+O(r^{-M}\eexp^{-c/(2r^2)})=o(r)\).
Thus, for small parameters, the interpolation segment lies inside the
interval of radius \(c_I r\), and
\begin{equation}
 \label{app-tr:arbitrary-interpolated-rate}
 \frac{2\sqrt{A_0B}}n
 \le C\left(
 h+h^{3/2}+r^{-M}\eexp^{-c/(2r^2)}
 \right).
\end{equation}
Together with \eqref{app-tr:flow-secant-control}, this is smaller than
one uniformly after reducing \(r_0,h_0\).  The buffered contraction gives
existence, strict monotonicity, and uniqueness.

The location is sharper.  The value-only identity
\eqref{app-tr:value-only-location}, the flow secant estimate, and
\begin{equation}
 \label{app-tr:A0-over-n}
 \frac{A_0}{n}
 \le C\{rH^2+H^3+r^{-M}\eexp^{-c/r^2}\}
\end{equation}
give
\begin{equation}
 \label{app-tr:arbitrary-inner-root}
 L_{\rk,\theta}^{\sigma_m}-L_\fl^{\sigma_f}
 =K_\theta(J)H^2
 +\ord\{rH^2+H^3+r^{-M}\eexp^{-c/r^2}\}.
\end{equation}
Since \(H^2\le r^2\) and the exponential floor is superalgebraically
small,
\begin{equation}
 \label{app-tr:root-tube-buffer}
 \abs{K_\theta(J)}H^2+A_0/n=o(r)
\end{equation}
uniformly on the compact fold and method classes.  Hence the contraction
tube is contained in \(I_r\) after reducing \(r_0,h_0\).
This proves the arbitrary-selection part of \cref{lem:root-transfer}.

For a fold-matched rule the same argument uses instead
\begin{equation}
 \label{app-tr:matched-corrected-C0}
 \norm F_{C^0(I_r^+)}
 \le A_{\mathfrak p}
 :=C_{\mathfrak p}H^2
 \{r^4+r^3H+r^{-M}\eexp^{-c/r^2}\}.
\end{equation}
It gives the stronger full-rectangle estimate
\begin{equation}
 \label{app-tr:matched-inner-root}
 L_{\rk,\theta}^{\mathfrak p}-L_\fl^{\mathfrak p}
 =K_\theta(J)H^2
 +\ord_{\mathfrak p}\!\left[
 H^2\{r+H+r^{-M}\eexp^{-c/r^2}\}
 \right].
\end{equation}

For fixed \(r\), convert \eqref{app-tr:matched-C1-response} from \(L\)
to \(\lambda=r^2L\); this costs only an \(r\)-dependent constant.  In a
fixed neighbourhood of the flow root,
\begin{equation}
 \label{app-tr:physical-C1-response}
 \Delta_{\rk,\theta}^{\mathfrak p}(\lambda)
 =\Delta_\fl^{\mathfrak p}(\lambda)
  -h^2\mathscr R_{\mathfrak p}(J,r,\lambda/r^2,\theta)
  +\ord_{C_\lambda^1,r,\mathfrak p}(h^3).
\end{equation}
Let
\(N_\lambda=\partial_\lambda\Delta_\fl^{\mathfrak p}
(\lambda_\fl^{\mathfrak p})\).  Transversality and the ordinary implicit
root estimate applied to \eqref{app-tr:physical-C1-response} yield
\begin{equation}
 \label{app-tr:fixed-r-root}
 \lambda_{\rk,\theta}^{\mathfrak p}
 -\lambda_\fl^{\mathfrak p}
 =h^2\kappa_\theta^{\mathfrak p}(J,r)
  +\ord_{r,\mathfrak p}(h^3),
 \qquad
 \kappa_\theta^{\mathfrak p}
 =\frac{
 \mathscr R_{\mathfrak p}
 (J,r,L_\fl^{\mathfrak p},\theta)}{N_\lambda}.
\end{equation}
Since \(N_\lambda=a(J)r+O(r^2)+O(r^{-M}\eexp^{-c/r^2})\), equations
\eqref{app-tr:matched-response-asymptotic} and \(S_\theta=aK_\theta\)
give
\begin{equation}
 \label{app-tr:kappa-asymptotic}
 \kappa_\theta^{\mathfrak p}(J,r)
 =K_\theta(J)r^4
  +\ord\bigl(r^5+r^{-M}\eexp^{-c/r^2}\bigr).
\end{equation}
This proves the matched and fixed-\(r\) parts of
\cref{lem:root-transfer}.

Finally, \(\lambda=r^2L\) and \(H=hr\) convert
\eqref{app-tr:arbitrary-inner-root} into
\begin{equation}
 \label{app-tr:normalized-physical-scale}
 \lambda_{\rk}-\lambda_\fl
 =K_\theta(J)h^2r^4
  +\ord(h^2r^5+h^3r^5)
  +\ord(r^{-M}\eexp^{-c/r^2}),
\end{equation}
with the exponential term carrying an additional factor \(H^2\) in the
matched case as in \eqref{app-tr:matched-inner-root}.  Under the affine
normalization of \cref{prop:affine-normalization},
\(h=k/\tau\), \(r^2=\eta/a_\eps\), and
\(\mu=a_\lambda\lambda\).  Hence the leading term in the original units is
\begin{equation}
 \label{app-tr:physical-unit-conversion}
 \mu_{\rk}-\mu_\fl
 =\frac{a_\lambda}{\tau^2a_\eps^2}
   K_\theta(J)k^2\eta^2+\text{the transformed remainder},
\end{equation}
which is the coefficient transformation
\eqref{eq:physical-coefficient-transform}.

\section{Construction of fold-matched invariant continuations}
\label{app:matched}

This appendix first proves the obstruction in
\cref{prop:pairing-necessary}.  It then constructs the fold-matched family
used in \cref{prop:matched-existence}.  The construction is confined to
normally hyperbolic collars.  Its restrictions are invariant manifolds of
the normalized flow and numerical map; the completion fixes a common local
choice for the two dynamics.

\subsection{Independent continuations cannot have a common step-first law}
\label{app:matched-no-go}

\begin{proof}[Proof of \cref{prop:pairing-necessary}]
Consider the canonical datum

\begin{equation}
 \dot x=x^2-y,
 \qquad
 \dot y=r^2(x-r^2L),
\label{eq:no-go-canonical-datum}
\end{equation}

which is a normalized analytic fold datum.  Fix the attracting side and
write \(x_o=-2\delta\) for the outer endpoint of its collar.  In the
construction underlying \cref{prop:fold-geometry}, choose the selection
bound strictly above the uniform graph bound and reduce the collar once.
This gives an admissible attracting flow graph \(m_0(x,L)\), defined on the
whole pasted interval \([x_o,0]\), for which the corridor, drift, and jet
inequalities have a uniform margin.  In particular, with
\[
 q_0(x,L)=x^2-m_0(x,L),
\]
one has \(c_dr^2\le q_0\le C_dr^2\).  Prescribe at \(x_o\) the perturbed
value
\(m_\zeta(x_o,L)=m_0(x_o,L)+\zeta\rho(r)\), where
\(\zeta\ne0\) is fixed and sufficiently small and

\[
 \rho(r)=\exp(-1/r^4).
\]

Solve the scalar graph equation
\begin{equation}
 \partial_xm_\zeta
 =\frac{r^2(x-r^2L)}{x^2-m_\zeta},
 \qquad m_\zeta(x_o,L)=m_0(x_o,L)+\zeta\rho(r),
 \label{eq:no-go-perturbed-ivp}
\end{equation}
towards \(x=0\), and leave the repelling graph unchanged.  This single
solution supplies both the collar restriction and the bridge restriction,
so there is no new pasting seam.

We verify that the solution reaches the fold and remains admissible.  Put
\(d=m_\zeta-m_0\) and \(q_\zeta=x^2-m_\zeta=q_0-d\).  On the bootstrap
set \(\abs d\le c_dr^2/2\), subtraction gives

\begin{equation}
 \partial_xd=A_\zeta d,
 \qquad
 A_\zeta(x,L)=
 \frac{r^2(x-r^2L)}{q_\zeta(x,L)q_0(x,L)}.
 \label{eq:no-go-homogeneous-difference}
\end{equation}
Both denominator factors are at least \(c_dr^2/2\).  The positive part of
\(A_\zeta\) is confined to \(x>r^2L\), an interval of length \(O(r^2)\)
inside the fold core; on that interval \(\abs{x-r^2L}=O(r^2)\), and hence
\begin{equation}
 \int_{x_o}^{x}(A_\zeta(s,L))_+\,\dd s\le C
 \quad (x_o\le x\le0).
 \label{eq:no-go-positive-part}
\end{equation}
Away from this core, \(A_\zeta\le0\).  The formula
\begin{equation}
 d(x,L)=\zeta\rho(r)
  \exp\left\{\int_{x_o}^{x}A_\zeta(s,L)\,\dd s\right\}
 \label{eq:no-go-difference-formula}
\end{equation}
therefore gives \(\abs d\le C\abs\zeta\rho(r)\).  Since
\(\rho(r)=o(r^N)\) for every \(N\), the bootstrap closes for all small
\(r\), and \eqref{eq:no-go-perturbed-ivp} reaches \(x=0\) with
\(q_\zeta\ge c_dr^2/2\).

The parameter and jet bounds follow from the same equation, rather than
from an absolute-value Gronwall estimate that would discard the damping.
For example,
\begin{align*}
 \partial_xd_L&=A_\zeta d_L+(\partial_LA_\zeta)d,\\
 \partial_xd_{LL}&=A_\zeta d_{LL}
   +2(\partial_LA_\zeta)d_L+(\partial_L^2A_\zeta)d.
\end{align*}
After further \(x\)-differentiation, the highest derivative always has
the same homogeneous coefficient \(A_\zeta\); all other terms contain a
strictly lower derivative of \(d\).  On the bootstrap tube, differentiation
of the two denominators gives, for the finite jet down-set used in
\eqref{eq:main-selection-bounds},
\[
 \abs{\partial_x^i\partial_L^jA_\zeta}
 \le C_{ij}r^{-N_{ij}}.
\]
Variation of constants with the signed propagator controlled by
\eqref{eq:no-go-positive-part}, followed by induction over \((i,j)\), now
gives
\begin{equation}
 \sup_{[x_o,0]\times\Lambda}
 \abs{\partial_x^i\partial_L^j(m_\zeta-m_0)}
 \le C_{ij}r^{-N_{ij}}\rho(r),
 \qquad 0\le i\le4,\quad0\le j\le2.
 \label{eq:no-go-selection-perturbation}
\end{equation}
Thus the collar bounds, the bridge weights, the stage neighbourhood, and
the directed-drift inequalities all retain their strict margins after
reducing \(r_0\).  The perturbed graph is consequently an admissible
flow continuation on the full collar--bridge domain.

Formula \eqref{eq:no-go-difference-formula} also shows that \(d\) never
vanishes.  Let
\(\Delta_0(L)=m_{\rep}(0,L)-m_0(0,L)\) and
\(\Delta_\zeta(L)=\Delta_0(L)-d(0,L)\).  By
\cref{app-geom:flow-profile}, \(\Delta_0\) has a unique simple zero
\(L_0(r)\) in the response window and
\(\partial_L\Delta_0\ge c r^3\) there.  The \(C_L^2\) estimate
\eqref{eq:no-go-selection-perturbation} is superalgebraically smaller than
this slope and than the endpoint margin.  Hence \(\Delta_\zeta\) has a
unique simple zero \(L_\zeta(r)\) in the same window.  Since
\[
 \Delta_\zeta(L_0(r))=-d(0,L_0(r))\ne0,
\]
the two roots, and therefore the two normalized parameters
\(r^2L_0(r)\) and \(r^2L_\zeta(r)\), are distinct.

Now fix any numerical-map continuation.  If both quotients in
\eqref{eq:unpaired-step-first} had finite limits, subtracting them would
make

\[
 h^{-2}\{\lambda_{\fl}^{1}(r)-\lambda_{\fl}^{0}(r)\}
\]

bounded as \(h\downarrow0\), contrary to the strict inequality of the two
fixed flow roots.  This proves the proposition.
\end{proof}

\subsection{A common completion on the normally hyperbolic collars}
\label{app:matched-construction}

Put \(\nu=r^2\), write \(w=(x,L)\), and use the divided normal coordinate

\begin{equation}
 y=\phi_J(x)+\nu v.
\label{eq:matched-divided-coordinate}
\end{equation}

On the attracting side we use the forward field and the tableau \(\theta\).
On the repelling side we reverse the field and use the adjoint tableau
\(\theta^\dagger\); its negative-step branch is the contained local inverse
of the positive-step numerical map.  Let \(q\in\{\att,\rep\}\) denote the
two oriented sides.

\begin{definition}[Common collar completion]
\label{def:common-completion}
A common collar completion consists of nested neighbourhoods
\(\mathcal U_0\Subset\mathcal U_1\) in \((w,v)\), a fixed normal strip,
and fixed cutoff and normal-extension operators with the following
properties.

The completed base is
\(\mathcal B_q=\mathbb T_q\times\overline{\mathcal I_L^+}\), where the
original normalized \(x\)-collar is embedded as a proper arc of the circle
\(\mathbb T_q\), and all coefficients extend to an open neighbourhood of
the compact parameter interval \(\overline{\mathcal I_L^+}\).  The
oriented base velocity has one sign outside the collar arc, and every
completed base map is a diffeomorphism of \(\mathbb T_q\).  Completed
graphs are periodic in this circle coordinate.  This fixed periodic
completion, rather than a free boundary value at the end of a finite
collar, is part of the convention.

The completed oriented vector fields agree with the divided normalized fields
on \(\mathcal U_1\), have uniformly bounded derivatives through order
seven, preserve the normal strip, and satisfy on the full completed graph
neighbourhood the strict normal inequality

\begin{equation}
 \partial_v\widetilde B_q(w,v)\le-2\kappa<0
\label{eq:completion-normal-gap}
\end{equation}

(in particular, on its outer part).  The local difference between the
divided normalized Runge--Kutta map and the corresponding time-\(h\) flow
map is extended linearly
with the same cutoffs, after the Runge--Kutta map has been formed in the
original normalized affine variables.  The extension agrees with that difference on
\(\mathcal U_0\).  The normal integration constant is fixed on the
completed flow graph, so the extension vanishes there whenever the local
defect does.  All neighbourhoods and extension operators are independent of
\(J,r,h\) and of the individual \(\theta\in\ThetaRK\).
\end{definition}

Embed the original normalized base interval in an arc of a circle and
choose fixed cutoffs
\(\chi_0,\chi_1\) with \(\chi_i=1\) on \(\mathcal U_i\).  A bounded linear
extension operator \(\mathcal E\), fixed on these neighbourhoods, first extends a
function in the base variables and then multiplies it by the appropriate
cutoff.  Choose a smooth exterior base speed of the same oriented sign as
the local speed.  If \(A_q^{\rm loc},B_q^{\rm loc}\) are the divided
normalized coefficients, set on the smaller neighbourhood
\[
 \widetilde A_q=A_q^{\rm loc},\qquad
 \partial_v\widetilde B_q=\partial_vB_q^{\rm loc},
\]
and outside the larger neighbourhood set
\[
 \widetilde A_q=A_q^{\rm out},\qquad
 \partial_v\widetilde B_q=-3\kappa.
\]
On the intervening annulus these functions are joined by
\(\mathcal E\) and the fixed cutoffs.  Reducing the local neighbourhood first
ensures \(\partial_vB_q^{\rm loc}\le-3\kappa\) there, so the interpolation
retains \(\partial_v\widetilde B_q\le-2\kappa\).  Fix a reference level
\(v_*\) in the normal strip and define
\begin{equation}
 \widetilde B_q(w,v)=b_q(w)
  +\int_{v_*}^{v}\partial_v\widetilde B_q(w,s)\,\dd s.
 \label{eq:completion-normal-extension}
\end{equation}
The periodic function \(b_q\) is chosen equal to the local value on
\(\mathcal U_1\), and on the exterior arc so that the two strip boundaries
point inward.  This specifies a periodic completed field; no boundary value
at an endpoint of the original normalized collar remains free.

The extension of the map defect uses the same linear data.  Let
\(D_{q,h,\theta}^{\rm loc}=P_{q,h,\theta}^{\rm orig}-F_{q,h}^{\rm loc}\),
where the superscript indicates the map already formed in the original
normalized \((x,y)\)-coordinates and then conjugated to \((w,v)\).  Extend its base
component by \(\mathcal E\).  Extend \(\partial_vD_v^{\rm loc}\) by
\(\mathcal E\), and recover the normal component by
\begin{equation}
 \widetilde D_v(w,v)
 =\mathcal E\!\left[D_v^{\rm loc}(\,\cdot\,,S_{\fl,q})\right](w)
  +\int_{S_{\fl,q}(w)}^{v}
     \mathcal E[\partial_vD_v^{\rm loc}](w,s)\,\dd s.
 \label{eq:completion-defect-extension}
\end{equation}
Here the extension of the first term is taken along the completed flow
graph and agrees with the local value on \(\mathcal U_0\).  Both operations
are bounded and linear on the strong and weak jet spaces.  In particular,
they commute with division by the scalar factors \(h\) and \(\nu\).

These choices produce the completions in
\cref{def:common-completion}.  The strip boundary points inward, while the
base map is the identity plus a uniformly small perturbation.  The
completed flow and numerical maps therefore preserve the same strip,
have invertible base maps, and retain a strict graph-transform gap.

More explicitly, write the four derivative blocks of either completed map
as \(\mathcal L_{ij}\), with the identity part removed from the base block.
Uniformly on the closed parameter rectangle,
\begin{equation}
 \mathcal L_{11}+\mathcal L_{12}\le Ch\nu,\qquad
 \mathcal L_{21}\le Ch,\qquad
 \mathcal L_{22}\le1-\kappa h .
\label{eq:completion-graph-blocks}
\end{equation}
After reducing \(r_0\), the cross coupling
\(2\sqrt{\mathcal L_{12}\mathcal L_{21}}\) is smaller than the remaining
normal gap.  To use the Banach-space theorem literally, lift the circle
coordinate periodically to \(\mathbb R\), extend the coefficients in \(L\)
from \(\mathcal I_L^+\) to \(\mathbb R\) by one fixed bounded extension, and
take \(X=\mathbb R^2_{(x,L)}\), \(Y=\mathbb R_v\).  The normal strip is the
invariant subset supplied by Remark~0, while \(L\) is a zero-velocity base
parameter.  The identifications and hypotheses H1, \((9\mathrm a^*)\), B4
and B5 are exactly those checked in
\cref{app-geom:NS-identification,app-geom:strict-graph-gap} and the jet
recurrence following them.  The invariant-graph theorem therefore gives a
unique lifted graph with derivatives through order seven
\cite{NippStoffer1992}.  Translation by one period gives another such graph,
so uniqueness makes it periodic; it consequently descends to
\(\mathbb T_q\times\mathcal I_L^+\).  Differentiating the graph
equation uses the same \(O(h)\) normal gap at each jet; the forcing at that
jet also contains an \(O(h)\) factor, so the derivative bounds are uniform
as \(h\downarrow0\).

Fix one such completion.  On side \(q\), its vector field has the form

\begin{equation}
 \dot w=\nu\widetilde A_q(w,v),
 \qquad
 \dot v=\widetilde B_q(w,v).
\label{eq:completed-divided-field}
\end{equation}

Let \(S_{\fl,q}\) be its invariant graph.  Denote the time-\(h\) map of
\eqref{eq:completed-divided-field} by \(\widetilde F_{q,h}\).  Conjugate the
Runge--Kutta branch in the original normalized variables to the divided
coordinates, subtract its local time-\(h\) flow branch, and extend that
difference according to
\cref{def:common-completion}, and add it to \(\widetilde F_{q,h}\).  The
resulting completed numerical map is denoted \(\widetilde P_{q,h,\theta}\).
By construction,

\begin{equation}
 \widetilde P_{q,h,\theta}=P_{q,h,\theta}^{\rm orig}
 \quad\hbox{on }\mathcal U_0,
\label{eq:completion-original-agreement}
\end{equation}

where the right-hand side is the conjugate of the already discretized map
in the original normalized variables.  Let \(S_{\rk,q}(h,\theta)\) be the
invariant graph of
\(\widetilde P_{q,h,\theta}\), and set
\(S_{\rk,q}(0,\theta)=S_{\fl,q}\).

The next lemma isolates the only singular small-step issue.  The contraction
of the numerical graph transform degenerates like \(h\), while the
order-two local defect on the slow graph has the factor \(h^3\nu\).  Their
quotient is the matching scale \(h^2\nu=H^2\).

Let \(\mathcal X_q=C_x^3C_L^1\) be the completed graph space equipped with
the downward-triangular weighted norm from the collar derivative induction,
and let \(\mathcal X_q^+=C_x^5C_L^2\).  These norms are uniformly equivalent
to the corresponding rectangular norms on the fixed neighbourhoods.

\begin{lemma}[Normalized graph resolvent]
\label{lem:completion-resolvent}
There are \(r_0,h_0,C,\kappa_0>0\), uniform in
\(J\in\Gclass\), \(\theta\in\ThetaRK\), the two sides, and
\(0\le\nu\le r_0^2\), with the following properties.

Let \(\mathfrak G_{q,h,\theta}\) be the target-aligned graph transform of
\(\widetilde P_{q,h,\theta}\).  On the completed flow graph,

\begin{equation}
 \mathfrak G_{q,h,\theta}(S_{\fl,q})-S_{\fl,q}
 =h^3\nu\{R_{q,3}+hR_{q,4}(h)\},
\qquad
 \norm{R_{q,3}}_{\mathcal X_q^+}
 +\norm{R_{q,4}(h)}_{\mathcal X_q}\le C .
\label{eq:completion-divided-defect}
\end{equation}

If \(\mathscr W_{q,h}\) is the secant graph operator between
\(S_{\rk,q}\) and \(S_{\fl,q}\), then

\begin{equation}
 \norm{\mathscr W_{q,h}}_{\mathcal X_q\to\mathcal X_q}
 \le1-\kappa_0h .
\label{eq:completion-secant-gap}
\end{equation}

Set

\[
 \mathscr L_{q,r,h}=h^{-1}(I-\mathscr W_{q,h}).
\]

There is a continuous normal graph operator
\(\mathscr L_{q,r}:\mathcal X_q^+\to\mathcal X_q\) such that

\begin{equation}
 \mathscr L_{q,r}U_{q,2}=R_{q,3}
\label{eq:completion-continuous-resolvent}
\end{equation}

has a unique bounded solution \(U_{q,2}\in\mathcal X_q^+\), and

\begin{equation}
 \norm{\mathscr L_{q,r,h}^{-1}}_{\mathcal X_q\to\mathcal X_q}
 +\norm{U_{q,2}}_{\mathcal X_q^+}\le C,
 \qquad
 \norm{(\mathscr L_{q,r,h}-\mathscr L_{q,r})U_{q,2}}_{\mathcal X_q}
 \le Ch .
\label{eq:completion-resolvent-bounds}
\end{equation}
\end{lemma}

\begin{proof}
We first justify the two scalar factors in
\eqref{eq:completion-divided-defect}.  The critical equation gives the
analytic factorisation

\[
 f_J(x,\phi_J(x)+\nu v;\nu,\nu L)
 =\nu\widehat f_J(x,v;\nu,L).
\]

Thus every stage satisfies \(x_i-x=O(h\nu)\).  The quotient created by
\eqref{eq:matched-divided-coordinate} is regular because

\[
 \frac{\phi_J(x_i)-\phi_J(x)}{\nu}
 =\frac{x_i-x}{\nu}\int_0^1
   \phi_J'(x+t(x_i-x))\dd t .
\]

The stage equations and their required rectangular derivatives therefore
extend smoothly to \(\nu=0\).  The same holds for the time-\(h\) flow map and,
by linearity, for the extended map defect.

To check that target alignment does not destroy the scalar factors, write a
completed map as \(P=(T,V)\).  For a graph \(S\), let
\(\sigma_P^S(w)\) be the unique source determined by
\[
 T\bigl(\sigma_P^S(w),S(\sigma_P^S(w))\bigr)=w,
\]
so that
\[
 \mathfrak G_P(S)(w)
 =V\bigl(\sigma_P^S(w),S(\sigma_P^S(w))\bigr).
\]
The inverse-base implicit-function theorem is uniform because the base
Jacobian is the identity plus \(O(h\nu)\).  Interpolate between the time-\(h\)
flow map and the numerical map by
\(P_t=F+t\mathcal D=(T_t,V_t)\), and write
\(\mathcal D=(\mathcal D^b,\mathcal D^n)\) for its base and normal
components.  At the source
\(\sigma_t=\sigma_{P_t}^{S_{\fl,q}}(w)\), put
\[
 A_t=\partial_wT_t+\partial_vT_tD_wS_{\fl,q},\qquad
 B_t=\partial_wV_t+\partial_vV_tD_wS_{\fl,q}.
\]
Differentiating the target equation gives
\(\partial_t\sigma_t=-A_t^{-1}\mathcal D^b\).  Differentiating the normal
output gives the exact identity
\begin{equation}
 \mathfrak G_{P}(S_{\fl,q})-
 \mathfrak G_{F}(S_{\fl,q})
 =\int_0^1
 \left\{\mathcal D^n-B_tA_t^{-1}\mathcal D^b\right\}
 (\sigma_t,S_{\fl,q}(\sigma_t))\,\dd t .
 \label{eq:completion-target-aligned-factor}
\end{equation}
No division by \(h\) or \(\nu\) occurs.  The bounded linear
extension \eqref{eq:completion-defect-extension} and target alignment
therefore preserve every common scalar factor of \(\mathcal D\).

Order two makes the time-\(h\) flow and Runge--Kutta branches agree through
their
second \(h\)-derivatives, so their difference on \(S_{\fl,q}\) is
\(h^3\mathcal A_q(\nu,h)\), in both the map and target-aligned graph
coordinates.  At \(\nu=0\) the base is fixed and the flow
graph consists of equilibria of the normal equation.  The local time-\(h\)
flow fixes this graph, and all Runge--Kutta stages there equal their input;
the exterior extension was
normalized to vanish on the same graph.  Hence
\(\mathcal A_q(0,h)=0\).  Hadamard division first in \(\nu\) and then in
\(h\) gives the explicit representation
\begin{equation}
 \begin{aligned}
 \mathcal A_q(\nu,h)
 &=\nu\int_0^1\partial_\nu\mathcal A_q(t\nu,h)\,\dd t,\\
 R_{q,3}
 &=\int_0^1\partial_\nu\mathcal A_q(t\nu,0)\,\dd t,\\
 R_{q,4}(h)
 &=\int_0^1\!\int_0^1
   \partial_h\partial_\nu\mathcal A_q(t\nu,sh)\,\dd s\,\dd t.
 \end{aligned}
 \label{eq:completion-two-hadamard-divisions}
\end{equation}
Thus the defect is exactly \(h^3\nu\{R_{q,3}+hR_{q,4}(h)\}\), with
\(\nu=r^2\).  Taking the third \(h\)-coefficient before estimating the
remainder gives the \(\mathcal X_q^+\) bound for \(R_{q,3}\); the integral
remainder has the stated \(\mathcal X_q\) bound.  This proves
\eqref{eq:completion-divided-defect}, including its two scalar factors.

The differentiated normal gap in
\eqref{eq:completion-normal-gap}, with downward-triangular weights absorbing
the lower-jet couplings, gives \eqref{eq:completion-secant-gap}.  More
explicitly, if \(D^a\) is a derivative in the finite down-set defining
\(\mathcal X_q\), target alignment and secant subtraction have the
triangular form
\begin{equation}
 D^a(\mathscr W_{q,h}U)
 =\omega_{a,h}\,(D^aU)\circ\sigma_h
  +h\sum_{b<a}C_{ab,h}\,(D^bU)\circ\sigma_h,
 \qquad
 \abs{\omega_{a,h}}\le1-2\kappa_ah,
 \label{eq:completion-secant-jet}
\end{equation}
where \(\sigma_h\) is the inverse base map and the finitely many
\(C_{ab,h}\) are uniformly bounded.  Choose the weight of each lower jet
successively so that the sum in \eqref{eq:completion-secant-jet} consumes
at most half of \(\kappa_ah\).  The resulting norm is equivalent to the
rectangular norm and yields
\(\norm{\mathscr W_{q,h}}\le1-\kappa_0h\) with one \(\kappa_0\) for the
whole compact class.

Secant subtraction of the graph equations, followed by
\cref{eq:completion-divided-defect,eq:completion-secant-gap}, first yields
\[
 \norm{S_{\rk,q}-S_{\fl,q}}_{\mathcal X_q}\le Ch^2\nu.
\]
The first variation of the continuous
graph equation defines

\begin{equation}
\begin{split}
 \mathscr L_{q,r}U={}&
 \nu\widetilde A_q(w,S_{\fl,q})\mathbin{\cdot}D_wU\\
 &+\{\nu D_wS_{\fl,q}\mathbin{\cdot}
       \partial_v\widetilde A_q(w,S_{\fl,q})
       -\partial_v\widetilde B_q(w,S_{\fl,q})\}U .
\end{split}
\label{eq:completion-continuous-operator}
\end{equation}

Put
\begin{equation}
 b_q(w)=\nu\widetilde A_q(w,S_{\fl,q}(w)),\qquad
 c_q(w)=\nu D_wS_{\fl,q}\mathbin{\cdot}
           \partial_v\widetilde A_q-\partial_v\widetilde B_q.
 \label{eq:completion-characteristic-coefficients}
\end{equation}
After reducing \(r_0\), \(c_q\ge\kappa\).  If \(\varphi_t\) is the
completed base flow generated by \(b_q\), the bounded periodic solution of
\(\mathscr L_{q,r}U=R\) is
\begin{equation}
 U(w)=\int_0^\infty
 \exp\left\{-\int_0^t c_q(\varphi_{-s}w)\,\dd s\right\}
 R(\varphi_{-t}w)\,\dd t.
 \label{eq:completion-continuous-inverse}
\end{equation}
The integral converges at rate \(\eexp^{-\kappa t}\), including when
\(\nu=0\), where it reduces to pointwise division by
\(-\partial_v\widetilde B_q\).  It is the only bounded periodic solution:
the difference of two solutions has normal multiplier at most
\(\eexp^{-\kappa t}\) around successive periods.  Differentiating
\eqref{eq:completion-continuous-inverse} in the down-set defining
\(\mathcal X_q^+\) gives bounds of the form
\(P(t)\eexp^{C\nu t}\).  Reducing \(r_0\) so that
\(C\nu\le\kappa/2\) makes them integrable against the normal factor
\(\eexp^{-\kappa t}\).  Hence
\begin{equation}
 \norm{\mathscr L_{q,r}^{-1}R}_{\mathcal X_q^+}
 \le C\norm R_{\mathcal X_q^+},
 \label{eq:completion-continuous-inverse-bound}
\end{equation}
uniformly down to \(\nu=0\).  This proves existence, uniqueness, and the
strong bound for \(U_{q,2}\) without imposing a boundary value at an
endpoint of the original normalized collar.

It remains to compare the discrete difference quotient with
\eqref{eq:completion-continuous-operator}.  First use the time-\(h\) map of
the completed flow and linearize its target-aligned graph transform at
\(S_{\fl,q}\).  It has the scalar transport form
\begin{equation}
 \mathscr W^f_{q,h}U(w)=\omega_h(w)U(\sigma_h(w)),
 \quad
 \begin{cases}
  \sigma_h(w)=w-hb_q(w)+h^2a_h(w),\\
  \omega_h(w)=1-hc_q(w)+h^2c_h(w),
 \end{cases}
 \label{eq:completion-flow-linearization}
\end{equation}
where \(a_h,c_h\), with the derivatives required in the weak norm, are
uniformly bounded.  Taylor's formula with integral remainder gives
\begin{equation}
 \left\|
 h^{-1}(I-\mathscr W^f_{q,h})U-\mathscr L_{q,r}U
 \right\|_{\mathcal X_q}
 \le Ch\norm U_{\mathcal X_q^+}.
 \label{eq:completion-flow-strong-weak}
\end{equation}
Indeed, the translation remainder after three \(x\)-derivatives uses at
most two further \(x\)-derivatives of \(U\); differentiating it once in
\(L\) uses one further \(L\)-derivative.  These are precisely the jets in
\(\mathcal X_q^+\).

The order-two numerical map and the time-\(h\) flow map differ by \(O(h^3)\)
with the same weak rectangular derivatives on the fixed strip.  Their
target-aligned linear graph operators therefore differ by \(O(h^3)\), and
after division by \(h\) this is \(O(h^2)\).  Finally, replacing the
linearization at \(S_{\fl,q}\) by the secant between
\(S_{\rk,q}\) and \(S_{\fl,q}\) costs
\begin{equation}
 h^{-1}O(h)\norm{S_{\rk,q}-S_{\fl,q}}_{\mathcal X_q}
 =O(h^2\nu).
 \label{eq:completion-secant-linear-cost}
\end{equation}
Combining these two terms with
\eqref{eq:completion-flow-strong-weak} proves, for
\(U\in\mathcal X_q^+\),

\begin{equation}
 \norm{(\mathscr L_{q,r,h}-\mathscr L_{q,r})U}_{\mathcal X_q}
 \le Ch\norm U_{\mathcal X_q^+}.
\label{eq:completion-strong-weak}
\end{equation}

This is a strong-to-weak estimate; operator-norm convergence on
\(C_x^3C_L^1\) is neither used nor asserted.  Finally,
\eqref{eq:completion-secant-gap} gives

\[
 \norm{\mathscr L_{q,r,h}^{-1}}
 =h\norm{(I-\mathscr W_{q,h})^{-1}}
 \le\kappa_0^{-1}.
\]

Applying \eqref{eq:completion-strong-weak} to \(U_{q,2}\) completes
\eqref{eq:completion-resolvent-bounds}.
\end{proof}

\begin{proposition}[The common completion is fold matched]
\label{prop:completion-produces-match}
For each side \(q\in\{\att,\rep\}\),

\begin{equation}
 \norm{S_{\rk,q}-S_{\fl,q}-h^2\nu U_{q,2}}_{C_x^3C_L^1}
 \le Ch^3\nu,
 \qquad
 \norm{U_{q,2}}_{C_x^3C_L^1}\le C .
\label{eq:completion-matching-expansion}
\end{equation}

After restriction to the original normalized collars and continuation to
the fold section, these graphs give the fold-matched invariant family in
\cref{prop:matched-existence}.
\end{proposition}

\begin{proof}
Secant subtraction of the two graph fixed-point equations gives

\[
 (I-\mathscr W_{q,h})(S_{\rk,q}-S_{\fl,q})
 =h^3\nu\{R_{q,3}+hR_{q,4}(h)\}.
\]

The inverse estimate in \cref{lem:completion-resolvent} first yields

\[
 \norm{S_{\rk,q}-S_{\fl,q}}_{\mathcal X_q}\le Ch^2\nu.
\]

For \(h>0\), set
\(U_{q,h}=(S_{\rk,q}-S_{\fl,q})/(h^2\nu)\).  Division of the secant
equation by \(h^3\nu\) gives

\[
 \mathscr L_{q,r,h}U_{q,h}=R_{q,3}+hR_{q,4}(h).
\]

Subtract \eqref{eq:completion-continuous-resolvent} and use
\eqref{eq:completion-resolvent-bounds}; then

\[
 \norm{U_{q,h}-U_{q,2}}_{\mathcal X_q}\le Ch.
\]

This proves \eqref{eq:completion-matching-expansion}.  The completed graphs
agree with the normalized flow and numerical map on \(\mathcal U_0\) by
\eqref{eq:completion-original-agreement}.  Their restrictions are therefore
invariant collar graphs.  On the repelling side the adjoint identity
identifies the oriented branch with the contained inverse of the
positive-step Runge--Kutta map.  Regraphing across the collar--bridge overlap preserves
the local invariant germs and the matching estimate.
Since \(h^2\nu=H^2\), all clauses of \cref{def:fold-matched} follow.
\end{proof}

\section*{Acknowledgements}
OpenAI Codex (GPT--5 service version, accessed 23 August 2026) assisted with
drafting, language editing, and code development; the author verified all
retained content and assumes full responsibility.

\paragraph{Funding.}
The author received no specific funding for this work.

\paragraph{Data and code availability.}
Source code and data for both figures are included as ancillary files with
this preprint.  No external experimental data were used.

\bibliographystyle{unsrtnat}
\bibliography{references}

\end{document}